\documentclass[a4paper, 11pt]{article}
\usepackage[utf8]{inputenc}
\usepackage{placeins}
\RequirePackage{etex} 
\usepackage{authblk}
\usepackage{blindtext}
\usepackage{comment}
\usepackage[utf8]{inputenc}
\usepackage{amsmath}
\usepackage{amsthm}
\usepackage{amsfonts}
\usepackage{amssymb}
\usepackage{algorithm}
\usepackage{algorithmic}
\usepackage{pifont}
\usepackage{empheq}
\usepackage{mathtools}
\usepackage{amsbsy}
\usepackage{braket}
\usepackage[font=footnotesize]{caption}
\usepackage[export]{adjustbox}
\usepackage[shortlabels]{enumitem}
\usepackage{xcolor}
\usepackage{multicol}
\usepackage{graphicx}
\usepackage[english]{babel}
\usepackage[margin=0.8in]{geometry}

\usepackage{float}
\usepackage{longtable}
\usepackage{subcaption}
\usepackage{listings}
\usepackage{hyperref}
\usepackage{cleveref}
\usepackage{autonum}
\usepackage{siunitx}
\usepackage{tikz}
\usepackage{enumitem}
\usepackage[makeroom]{cancel}
\usepackage[square, numbers, sort&compress]{natbib} 
\usepackage{pgfplots}
\pgfplotsset{compat=1.7}

\theoremstyle{definition}
\newtheorem{definition}{Definition}
\numberwithin{definition}{section}

\newtheorem{assumption}{Assumption}
\numberwithin{assumption}{section}

\newtheorem{theorem}{Theorem}[section]
\theoremstyle{plain}

\newtheorem{proposition}{Proposition}
\numberwithin{proposition}{section}

\theoremstyle{plain}

\newtheorem{remark}{Remark}[section]
\theoremstyle{plain}

\newtheorem{lemma}{Lemma}[section]
\theoremstyle{plain}

\newcommand{\jump}[1]{[\mkern-1.5mu [#1] \mkern-1.5mu]}
\newcommand{\avg}[1]{\{ \mkern-5mu \{#1 \} \mkern-5mu \}}

\renewcommand{\div}[1]{\nabla \mkern-2.5mu \cdot \mkern-2.5mu {#1}}
\newcommand{\divh}[1]{\nabla_h \mkern-2.5mu \cdot \mkern-2.5mu {#1}}

\newcommand{\triplenorm}[1]{\|\mkern-1.7mu| #1 \|\mkern-1.7mu|}  

\definecolor{myred}{rgb}{0.9, 0.0, 0.0}
\definecolor{myblue}{rgb}{0.0, 0.28, 0.67}
\definecolor{mygreen}{rgb}{0.0, 0.7, 0.0}
\definecolor{myyellow}{rgb}{1.0, 0.55, 0.0}
\definecolor{mypurple}{rgb}{0.5, 0.1, 0.5}
\definecolor{mygrey}{rgb}{0.7, 0.7, 0.7}

\title{\textbf{Conforming and discontinuous discretizations of non-isothermal Darcy–Forchheimer flows}}
\date{}

\author[1]{Stefano Bonetti}
\author[1]{Michele Botti}
\author[1]{Paola F. Antonietti}
\affil[1]{\small{\textit{MOX-Dipartimento di Matematica, Politecnico di Milano, Piazza Leonardo da Vinci 32, Milan, 20133, Italy}}}

\begin{document}

\maketitle

\begin{center}
\begin{minipage}[c]{0.8\textwidth}
We present and analyze in a unified setting two schemes for the numerical discretization of a Darcy-Forchheimer fluid flow model coupled with an advection-diffusion equation modeling the temperature distribution in the fluid. The first approach is based on fully discontinuous Galerkin discretization spaces. In contrast, in the second approach, the velocity is approximated in the Raviart-Thomas space, and the pressure and temperature are still piecewise discontinuous. A fixed-point linearization strategy, naturally inducing an iterative splitting solution, is proposed for treating the nonlinearities of the problem. We present a unified stability analysis and prove the convergence of the iterative algorithm under mild requirements on the problem data. A wide set of two- and three-dimensional simulations is presented to assess the error decay and demonstrate the practical performance of the proposed approaches in physically sound test cases.
\end{minipage}
\end{center}
\maketitle

\section{Introduction}
\label{sec:Introduction}
This research addresses the numerical modeling of temperature distribution in fluids governed by the Darcy–Forchheimer (DF) law, which relates the fluid velocity vector to the pressure gradient and is applicable in regimes where the fluid velocity is sufficiently high. In such cases, the classical Darcy law—where the velocity is linearly proportional to the pressure gradient—fails to capture the underlying physics accurately. Previous studies \cite{Coussy2003,Jumah2001} have demonstrated that Darcy's law provides reliable predictions primarily for low velocities and low-porosity media. However, nonlinear effects become significant whenever high flow velocities and/or variable porosity are involved, leading to the need for an additional quadratic velocity term. The DF law incorporates both the Darcy term and the nonlinear inertial term, the latter scaling with the square of the fluid velocity and arising from inertial forces. Moreover, in relevant geophysical processes such as geothermal energy extraction and greenhouse gas sequestration, thermal effects, modeled via an advection–diffusion equation,  are also particularly relevant since temperature variations strongly influence fluid properties.
The resulting model is therefore fully coupled: the velocity field governs thermal advection, while temperature variations affect fluid density and viscosity, thereby altering the flow field. Moreover, it is nonlinear due to (i) the temperature-dependent viscosity, (ii) the quadratic Forchheimer term, and (iii) the nonlinear advection term in the temperature equation.
\par
Extensive research has been conducted on the Forchheimer equation, starting from experimental investigations \cite{Ward1964}. The mathematical analysis of the DF model has been considered in \cite{Salas2013}. Concerning the numerical discretization, we refer, e.g., to \cite{Pan2012,Rui2012,Urquiza2006} for mixed finite element discretizations. Regarding the coupling of flow models with the heat equation, we refer, e.g., to  \cite{Bernardi2015,Bernardi2018}, where a comprehensive analysis of the Darcy model coupled with the heat equation is presented. Heat convection in a liquid medium whose motion is described by the Stokes/Navier Stokes equations has been extensively studied; among others, we mention the papers \cite{Agroum2015,Agroum2015spectral,Antonietti2022,Deteix2014,Djoko2020}. Recently, there has been a growing interest in the study of a coupled DF-heat model. In \cite{Allendes2023,Deugoue2022,Sayah2021} the authors propose an approximation based on standard finite elements (with the introduction of \textit{bubble} functions for treating the velocity field). Moreover, in \cite{Allendes2023} a singular source term for the heat equation is considered. In \cite{Huang2025} the problem is approximated via a block-centered finite difference scheme.
\par
In this work, we study the coupled Darcy–Forchheimer–heat model and present and analyze two numerical discretization schemes for its numerical discretization. The first scheme (referred to as dG-dG-dG) employs an entirely discontinuous finite element formulation for the velocity, pressure, and temperature fields. The second scheme (referred to as RT-dG-dG) adopts a conforming Raviart–Thomas finite element space for the velocity, combined with discontinuous approximations for the pressure and temperature fields. The difference between the two approaches lies in the choice of the discretization space for the velocity field, that is, discontinuous Galerkin \cite{Arnold2002} for the first case and Raviart-Thomas \cite{Nedelec1980,Raviart1977} for the second one. 
The two approaches offer complementary advantages. The dG–dG–dG scheme is remarkably versatile, as it supports polytopal meshes (see \cite{Antonietti2013,Bassi2012,Cangiani2014} for dG schemes on polytopal meshes).  We remark that polytopal methods are particularly advantageous in computational geoscience due to their geometric flexibility in the process of mesh design, efficient handling of highly heterogeneous media, as well as natural support of domain-specific features (e.g., fractures) without prohibitive computational costs.
In the context of geophysics, PolyDG methods have been applied to flows in fracture porous media \cite{Antonietti2019}, poroelasticity \cite{Antonietti2021}, and thermo-hydro-mechanical models \cite{Antonietti2023,Bonetti2023}. For heat-type problems discretized using PolyDG methods, we refer, e.g., \cite{Cangiani2017}.
We refer to \cite{CangianiDong2017} for a comprehensive monograph on PolyDG methods, and to \cite{Antonietti2021_review} for a review of the current development of PolyDG methods for geophysical applications.

In this work, we propose a robust treatment of the convection term, following the thermo-poroelasticity framework in \cite{Bonetti2024}, with the key distinction that here the temperature field is advected by the fluid velocity rather than the Darcy velocity. The corresponding nonlinearities are addressed via an iterative linearization algorithm, which naturally leads to a splitting-based solution strategy. The main contributions of the proposed numerical framework are:
\textit{(a)} a detailed formulation of the dG–dG–dG and RT–dG–dG schemes, emphasizing the treatment of the nonlinear advection term and the corresponding linearization strategy;
\textit{(b)} a unified analysis establishing stability of the discrete problem, its well-posedness, and convergence of the fixed-point iteration; and \textit{(c)} an extensive set of two- and three-dimensional numerical experiments demonstrating convergence properties and validating the method in established benchmarks.

The rest of the paper is organized as follows: the model problem and its weak formulation are presented in Section~\ref{sec:ModelProblem}. In Section~\ref{sec:Discretization}, we present the two discretization schemes together with the corresponding linearization strategy. In Section~\ref{sec:StabilityAnalysis}, we prove, in a unified setting, the stability of the discrete problem, the well-posedness of the discrete problem, and the convergence of the fixed-point linearization algorithm. Two- and three-dimensional numerical results assessing the convergence properties and benchmark test cases are shown discussed in Section~\ref{sec:NumericalResults}.

\section{Model problem}
\label{sec:ModelProblem}
Let $\Omega \subset \mathbb{R}^d$, with $d \in \{2; 3\}$, be an open bounded Lipschitz domain. The non-isothermal DF problem reads: \textit{find $(\mathbf{u}, p, T)$ such that in $\Omega$ it holds:} 
\begin{subequations}
	\label{eq:model_problem}
	\begin{align}
		& \nu(T) \mathbf{K}^{-1} \, \mathbf{u} + \beta \lvert \mathbf{u}\rvert \mathbf{u} + \nabla p = \mathbf{f}, \label{eq:DF_flux} \\
		& \div{\mathbf{u}} = 0, \label{eq:incompressibility} \\
		& -\div{(\boldsymbol{\Theta}\nabla T)} + (\mathbf{u} \mkern-2.5mu \cdot \mkern-2.5mu \nabla) \, T = g,  \label{eq:heat_transfer}
	\end{align}
\end{subequations}
where $\lvert \cdot \rvert$ denotes the Euclidean vector norm, namely $\lvert \mathbf{u} \rvert^2 = \mathbf{u}^T \, \mathbf{u}$. In \eqref{eq:model_problem}, the variables $(\mathbf{u}, p, T)$ represent the velocity, the pressure, and the temperature, respectively. The function $\mathbf{f}: \Omega \rightarrow \mathbb{R}^d$ represents an external body force, while $g: \Omega \rightarrow \mathbb{R}$ is the heat source. Equations \eqref{eq:DF_flux}, \eqref{eq:incompressibility}, and \eqref{eq:heat_transfer} represent the DF fluid flow, the incompressibility of the fluid, and the heat transfer, respectively. The description of the DF flow is characterized by the Forchheimer coefficient $\beta$ that represents the strength of the inertial effects in the porous medium: it is the ratio between the pressure drop caused by fluid–solid interactions and the one caused by viscous and inertia resistances. If $\beta$ were null, then \eqref{eq:DF_flux} would reduce to the Darcy equation. We observe that the couplings between \eqref{eq:DF_flux} and \eqref{eq:heat_transfer} are bi-directional. Namely, we observe that the temperature $T$ influences the kinematic viscosity of the fluid $\nu(\cdot)$; moreover, in the heat transfer equation we have the presence of a convective term of the form $(\mathbf{u} \mkern-2.5mu \cdot \mkern-2.5mu \nabla) \, T$, hence the fluid flow affects the temperature field.
\begin{table}[H]
	\centering 
	\begin{tabular}{ c | c  | l c | c | l}
		\textbf{Symbol} & \textbf{Unit} & \textbf{Quantity} & \textbf{Symbol} & \textbf{Unit} & \textbf{Quantity} \\[3pt]
		$\nu$ & \si[per-mode = symbol]{\pascal \second} & fluid viscosity & $\beta$ & \si[per-mode = symbol]{\pascal \second\squared \per \metre\cubed} & Forchheimer coefficient \\
		$\mathbf{K}$ & \si[per-mode = symbol]{\metre\squared} & permeability & $\boldsymbol{\Theta}$ & \si[per-mode = symbol]{\metre\squared \per \second} & thermal conductivity  \\
	\end{tabular}
	\\[5pt]
	\caption{Model parameters appearing in \eqref{eq:model_problem}}
	\label{tab:DF_params}
\end{table}
In Table~\ref{tab:DF_params} we detail the parameters characterizing problem \eqref{eq:model_problem} specifying their physical interpretation and unit of measure. 
To close problem \eqref{eq:model_problem}, different types of boundary conditions can be considered (e.g. Dirichlet, Robin, or Neumann). If a full Neumann condition on the velocity is imposed, then the mean value of the pressure must be prescribed. For the sake of simplicity, in the rest of the article, we consider homogeneous Dirichlet conditions for both the pressure and temperature fields.

\subsection{Weak formulation}
In this section, we present the weak formulation of problem \eqref{eq:model_problem}. To this aim, we first introduce some notation and assumptions.

For $X\subseteq\Omega$, we denote by $L^p(X)$ the Lebesgue space of index $p\in [1, \infty]$ and by $H^q(X)$ the Sobolev space of index $q \geq 0$ of real-valued functions defined on $X$, with the convention that $H^0(X)=L^2(X)$.  The notation $\mathbf{L}^p(X)$ and $\mathbf{H}^q(X)$ is adopted in place of $\left[ L^p(X) \right]^d$ and $\left[ H^q(X) \right]^d$, respectively. 
In addition, we denote by $\mathbf{H}(\textrm{div}, X)$ the space of $\mathbf{L}^2(X)$ vector fields whose divergence is square integrable. These spaces are equipped with natural inner products and norms denoted by $(\cdot, \cdot)_X = (\cdot, \cdot)_{L^2(X)}$ and $||\cdot||_X = ||\cdot||_{L^2(X)}$, respectively, with the convention that the subscript can be omitted whenever $X=\Omega$.
For the sake of brevity, we make use of the symbol $x \lesssim y$ to denote $x \le C y$, where $C$ is a positive constant independent of the discretization parameters.
We introduce the following assumptions on the problem data.
\begin{assumption}[Regularity of the problem data]
	\label{ass:pbdata}
	We assume \\ that the problem data satisfy the following regularity conditions.
	\hspace{0pt}
	\label{assumption:model_problem}
	\begin{enumerate}
		\item The permeability $\textup{\textbf{K}}= (K)^d_{i,j=1}$ and thermal conductivity $\mathbf{\Theta}=(\Theta)^d_{i,j=1}$ are symmetric tensor fields which, for strictly positive real numbers $k_M>k_m$ and $\theta_M> \theta_m$, satisfy for a.e. $x \in \Omega$:
		$$
		k_m |\zeta|^2 \leq \zeta^T \textup{\textbf{K}}(x)\zeta \leq k_M |\zeta|^2
		\quad \text{and} \quad
		\theta_m |\zeta|^2 \leq \zeta^T \mathbf{\Theta}(x) \zeta \leq \theta_M |\zeta|^2
		\quad \forall\zeta \in \mathbb{R}^d.
		$$
		\item The fluid viscosity $\nu$ and the Forchheimer coefficient $\beta$ are scalar fields such that $\nu : \Omega\to[\nu_m,\nu_M]$ and $\beta : \Omega \to [0, \beta_{M}]$ with $0 < \nu_m \le \nu_M$ and $0 \le  \beta_{M}$.
		\item The kinematic viscosity $\nu(\cdot): \mathbb{R} \rightarrow \mathbb{R}^+$ is a bounded $\mathcal{C}^1(\mathbb{R}^+)$, globally Lipschitz function (we denote by $L_{\nu}$ its Lipschitz constant). 
		\item The forcing terms satisfy $\mathbf{f} \in \mathbf{L}^2(\Omega)$ and $g \in L^2(\Omega)$.
	\end{enumerate}
\end{assumption}

We can now introduce the weak formulation of problem \eqref{eq:model_problem}. Let $\mathbf{Z} = \{ \mathbf{z} \in \mathbf{L}^3(\Omega) : \div{\mathbf{z}} \in L^2(\Omega) \}$, $W = L^2(\Omega)$, and $V = H_0^1(\Omega)$,
the weak formulation formulation reads: \textit{find $(\mathbf{u}, p, T) \in \mathbf{Z} \times W \times V$ such that $\forall (\mathbf{v}, q, S) \in \mathbf{Z} \times W \times V$ it holds:}
\begin{equation}
	\label{eq:nonlinear_weak_form}
	\begin{aligned}
		& \mathcal{M}_{\nu}(T, \mathbf{u}, \mathbf{v}) + \mathcal{M}_{\beta}(\mathbf{u}, \mathbf{u}, \mathbf{v}) - (p, \div{\mathbf{v}}) + (q, \div{\mathbf{u}}) \\
		& + (\boldsymbol{\Theta} \nabla T, \nabla S) + \left((\mathbf{u} \mkern-2.5mu \cdot \mkern-2.5mu \nabla) \, T, S \right) = (\mathbf{f}, \mathbf{v}) + (g, S), 
	\end{aligned}
\end{equation}
where the non-linear forms $\mathcal{M}_{\nu}: V \times \mathbf{Z} \times \mathbf{Z} \rightarrow \mathbb{R}$, $\mathcal{M}_{\beta}: \mathbf{Z} \times \mathbf{Z} \times \mathbf{Z} \rightarrow \mathbb{R}$ are given by
$$
\label{eq:mass_bilinear_forms}
\begin{aligned}
	\mathcal{M}_{\nu}(T, \mathbf{u}, \mathbf{v}) = (\nu(T) \mathbf{K}^{-1} \, \mathbf{u}, \mathbf{v}), \quad  \mathcal{M}_{\beta}( \mathbf{w}, \mathbf{u}, \mathbf{v}) =  (\beta \lvert \mathbf{w} \rvert \mathbf{u}, \mathbf{v}).
\end{aligned}
$$

The well-posedness of problem \eqref{eq:nonlinear_weak_form} is addressed in \cite[Section 2.3]{Deugoue2022}. Therein, uniqueness is established under a suitable assumption on the data.

\section{Discretization}
\label{sec:Discretization}
This section introduces the dG–dG–dG and RT–dG–dG discretizations of problem \eqref{eq:nonlinear_weak_form}. In both approaches, discontinuous elements are employed for the pressure and temperature unknowns. For the velocity, the first approach uses dG schemes, while the second employs RT elements.

\subsection{Preliminaries}
\label{sec:Preliminariesall}
We start by introducing a subdivision $\mathcal{T}_h$ of the computational domain $\Omega$ made of disjoint open polytopal elements.  We remark that, in the general case, the dG method supports general polytopal meshes (cf. e.g., \cite{Antonietti2013,Bassi2012,Cangiani2014,CangianiDong2017}). 
We define an interface as the intersection of the $(d-1)$-dimensional facets of two neighbouring elements. If $d=3$, we further assume that each interface consists of a general polygon, which may be decomposed into a set of co-planar triangles.
We denote with $\mathcal{F}$, $\mathcal{F}_I$, and $\mathcal{F}_B$ the set of faces, interior faces, and boundary faces, respectively. We introduce the following definition.
\begin{definition}[Polytopal regular mesh \cite{Cangiani2014,CangianiDong2017}]
	\label{def:unif_regular}
	A mesh $\mathcal{T}_h$ is said to be polytopal regular if $\forall \kappa \in \mathcal{T}_h$, there exist a set of non-overlapping $d$-dimensional simplices contained in $\kappa$ -- denoted by $\{S_{\kappa}^F\}_{F \subset \partial \kappa}$ -- such that, for any face $F \subset \partial \kappa$, it holds $h_{\kappa} \lesssim d \ |S_{\kappa}^F| \ |F|^{-1}$, with $h_{\kappa}$ denoting the diameter of $\kappa$.
\end{definition}
We next introduce the mesh assumptions.
\begin{assumption}
	\label{assumption:mesh}
	Given $\{\mathcal{T}_h\}_h, h>0$, we assume that the following properties are uniformly satisfied:
	\begin{enumerate}[start=1,label={\bfseries A.\arabic* }]
		\item \label{assumption:mesh_A1} $\mathcal{T}_h$ is polytopal-regular in the sense of Defintion \ref{def:unif_regular};
		\item \label{assumption:mesh_A3} For neighbouring elements $\kappa^+, \kappa^- \in \mathcal{T}_h$, hp-local bounded variation property holds, i.e. $h_{\kappa^+} \lesssim h_{\kappa^-} \lesssim h_{\kappa^+}$, $p_{\kappa^+} \lesssim p_{\kappa^-} \lesssim p_{\kappa^+}$.
	\end{enumerate}
\end{assumption}
Note that the bounded variation hypothesis \ref{assumption:mesh_A3} is introduced to avoid technicalities. Under \ref{assumption:mesh_A1} the following inequality (\textit{trace-inverse} inequality) holds \cite{Cangiani2017}: 
\begin{equation}
	\label{eq:RobQSTPE_trace_inverse_ineq}
	\| v\|_{L^q(\partial \kappa)} \lesssim 
	h_{\kappa}^{-\frac{1}{q}}\, \ell_{\kappa}^{\frac{2}{q}}\, \|v\|_{L^q(\kappa)} \quad \forall v \in \mathbb{P}^{\ell}(\kappa),
\end{equation}
where $\mathbb{P}^{\ell}(\kappa)$ is the space of polynomials of maximum degree equal to $\ell$ in $\kappa$ and the hidden constant is independent of $\ell, h$, the number of faces per element, and the relative size of a face compared to the diameter of the element it belongs to. 

We also introduce the average and jump operators. We start by defining them on each interface $F\in\mathcal{F}_I$ shared by the elements $\kappa^{\pm}$ as in \cite{Arnold2002}:
$$
\label{eq:avg_jump_operators}
\begin{aligned}
	& \jump{a} = a^+ \mathbf{n^+} + a^- \mathbf{n^-}, \ 
	&& \jump{\mathbf{a}} = \mathbf{a}^+ \odot \mathbf{n^+} + \mathbf{a}^- \odot \mathbf{n^-}, \ 
	&&\jump{\mathbf{a}}_n = \mathbf{a}^+ \cdot \mathbf{n^+} + \mathbf{a}^- \cdot \mathbf{n^-}, \\ 
	& \avg{a} = \frac{a^+ + a^-}{2}, \
	&& \avg{\mathbf{a}} = \frac{\mathbf{a}^+ + \mathbf{a}^-}{2}, \ && \avg{\mathbf{A}} = \frac{\mathbf{A}^+ + \mathbf{A}^-}{2},
\end{aligned}
$$
where $\mathbf{a} \odot \mathbf{n} = \mathbf{a}\mathbf{n}^T$, and $a, \ \mathbf{a}, \ \mathbf{A}$ are scalar-, vector-, and tensor-valued functions, respectively. The notation $(\cdot)^{\pm}$ is used for the trace on $F$ taken within the interior of $\kappa^\pm$ and $\mathbf{n}^\pm$ is the outer normal vector to $\partial \kappa^\pm$. On boundary faces $F\in\mathcal{F}_B$, we set
$$
\jump{a} = a \mathbf{n},\ \
\avg{a} = a,\ \
\jump{\mathbf{a}} = \mathbf{a} \odot \mathbf{n},\ \
\avg{\mathbf{a}} = \mathbf{a},\ \
\jump{\mathbf{a}}_n = \mathbf{a} \cdot \mathbf{n},\ \
\avg{\mathbf{A}} = \mathbf{A}.
$$ 

We introduce some further notation for the RT discretization space \cite{Boffi2013}. Let us now assume that $\mathcal{T}_h$ is made of a conforming, shape-regular simplicial elements. We start by considering $\pmb{\mathbb{P}}^{{\ell}}(\kappa) = \left( \mathbb{P}^{\ell}(\kappa) \right)^d$ the space of piecewise polynomial vectors of degree $\ell$ defined on $\kappa$. The local RT space $\mathbb{RT}^{\ell}(\kappa)$ \cite{Raviart1977} is defined as
$$\mathbb{RT}^{\ell}(\kappa) = \left\{ \mathbf{v} \in \left(\pmb{\mathbb{P}}^{{\ell}}(\kappa) + \mathbf{x} \mathbb{P}^{\ell}(\kappa) \right): \ \mathbf{v} \cdot \mathbf{n} \in R_{\ell}(\partial \kappa) \right\},$$
with  $\mathbf{x} = (x_1, x_2, \dots, x_d)$ and 
$R_{\ell}(\partial \kappa)$ being the space of $L^2(\partial \kappa)$ functions which are piecewise polynomials of degree $\ell$ on each of the faces of $\kappa$.
$$
$$

Finally, for the sake of simplicity, we assume that the heat conductivity  $\boldsymbol{\Theta}$ and the permeability $\mathbf{K}$ are element-wise constant. Then, we can introduce the following quantity: $\Theta_{\kappa} = \left(|\sqrt{\boldsymbol{\Theta}\rvert_{\kappa}}|_2^2 \right)$,
where $|\cdot|_2$ is the $\ell^2$-norm in $\mathbb{R}^{d \times d}$. This assumption is reasonable in the context of geophysics, e.g. for groundwater flow models, where the data are obtained via local measurements.

\subsection{The dG-dG-dG discrete formulation}
\label{sec:dGdGdG_discrete}
In this section, we introduce the dG-dG-dG scheme. Given $\ell, m \geq 1$ such that $\ell+1 \geq m$, we define:
$$
\begin{aligned}
	& V_h^{\ell} = \left\{ S \in L^2(\Omega) : S |_{\kappa} \in \mathbb{P}^{\ell}(\kappa) \  \forall \kappa \in \mathcal{T}_h \right\}, \quad \mathbf{V}_h^{\ell} = \left[ V_h^{\ell} \right]^d, \\
	& W_h^{m} = \left\{ q \in L^2(\Omega) : q |_{\kappa} \in \mathbb{P}^{m}(\kappa) \  \forall \kappa \in \mathcal{T}_h \right\}.
\end{aligned}
$$
The dG-dG-dG discretization of problem \eqref{eq:nonlinear_weak_form} reads: \textit{find $(\mathbf{u}_h, p_h, T_h) \in \mathbf{V}_h^{\ell} \times W_h^{m} \times V_h^{\ell}$ such that $\forall (\mathbf{v}_h, q_h, S_h) \in \mathbf{V}_h^{\ell} \times W_h^{m} \times V_h^{\ell}$ it holds:}
\begin{equation}
	\label{eq:nonlinear_discrete_form_DG}
	\begin{aligned}
		& \mathcal{M}_{\nu}(T_h, \mathbf{u}_h, \mathbf{v}_h) + \mathcal{M}_{\beta}(\mathbf{u}_h, \mathbf{u}_h, \mathbf{v}_h) + \mathcal{B}_h(p_h, \mathbf{v}_h) - \mathcal{B}_h(q_h, \mathbf{u}_h) + \mathcal{A}_h(T_h,S_h) \\ 
		& + \mathcal{C}_h(\mathbf{u}_h, T_h, S_h) + \mathcal{D}_u(\mathbf{u}_h, \mathbf{v}_h) + \mathcal{D}_p(p_h, q_h) = ((\mathbf{f},g),  (\mathbf{v}_h, S_h)),
	\end{aligned}
\end{equation}
where the discrete bilinear and trilinear forms are defined by:
\begin{equation}
	\label{eq:bilinear_forms_discr_DG}
	\begin{aligned}
		& \mathcal{A}_h(T,S) = \left(\boldsymbol{\Theta}\nabla_h T, \nabla_h S\right) - \hspace{-6pt} \sum_{F \in \mathcal{F}} \int_F \left(\avg{\boldsymbol{\Theta}\nabla_h T} \mkern-2.5mu \cdot \mkern-2.5mu \jump{S} + \jump{T} \mkern-2.5mu \cdot \mkern-2.5mu \avg{\boldsymbol{\Theta}\nabla_h S} - \sigma \jump{T} \mkern-2.5mu \cdot \mkern-2.5mu \jump{S}\right),\\
		& \mathcal{B}_h(q,\mathbf{v}) = - (q,\nabla_h \mkern-2.5mu \cdot \mkern-2.5mu \mathbf{v}) + \sum_{F \in \mathcal{F}_I} \int_F \avg{q} \mkern-2.5mu \cdot \mkern-2.5mu \jump{\mathbf{v}}_n ,\\
		& \begin{aligned}
			\mathcal{C}_h (\mathbf{u}, T, S) = \ & \left( \mathbf{u} \cdot \nabla_h T, S \right) + \frac{1}{2}(\divh{\mathbf{u}} \, T, S) \\
			& - \sum_{F \in \mathcal{F}_I} \int_{F} \left( \avg{\mathbf{u}} \cdot \jump{T} \right) \avg{S} - \frac{1}{2}\sum_{F \in \mathcal{F}_I} \int_F \jump{\mathbf{u}}_n \cdot \avg{T \, S} \\
			& + \frac{1}{2}\sum_{F \in \mathcal{F}} \int_F \ \left| \avg{\mathbf{u}} \cdot \mathbf{n} \right| \jump{T} \mkern-2.5mu \cdot \mkern-2.5mu \jump{S} - \frac{1}{2}\sum_{F \in \mathcal{F}_B} \int_F \  (\mathbf{u} \cdot \mathbf{n}) \, T \, S  
		\end{aligned} \\
		& \mathcal{D}_u(\mathbf{u},\mathbf{v}) = \sum_{F \in \mathcal{F}_I} \int_F  \xi \, \jump{\mathbf{u}}_n \mkern-2.5mu \cdot \mkern-2.5mu \jump{\mathbf{v}}_n,\\
		& \mathcal{D}_p(p,q) = \sum_{F \in \mathcal{F}_I} \int_F \varrho \, \jump{p} \, \jump{v}.\\
	\end{aligned}
\end{equation}
For all $w\in V_h^{\ell}$ and $\mathbf{w}\in \mathbf{V}_h^{\ell}$, $\nabla_h w$ and $\divh{\mathbf{w}}$ denote the broken differential operators whose restrictions to each element $\kappa \in \mathcal{T}_h$ are defined as $\nabla w_{|\kappa}$ and $\div{\mathbf{w}}_{|\kappa}$, respectively. Due to the choice of Dirichlet boundary conditions for the pressure field, in the bilinear forms $\mathcal{B}_h$ and $\mathcal{D}_u$, the interface terms are summed only on the set of internal faces $\mathcal{F}_I$ of the mesh. The stabilization functions $\sigma, \xi, \varrho \in L^{\infty}(\mathcal{F}_h)$ appearing in \eqref{eq:bilinear_forms_discr_DG} are defined according to \cite{Arnold1982,Cangiani2017,Ern2021,Wheeler1978} as:
\begin{equation}
	\label{eq:stabilization_functions}
	\begin{aligned}
		\sigma &= \left\{\begin{aligned}
			&\alpha_1 \underset{\kappa \in \{\kappa^+,\kappa^-\}}{\mbox{max}} \left(\overline{\Theta}_{\kappa}h_{\kappa}^{-1}\ell^2 \right) \, & F \in \mathcal{F}_I,\\
			&\alpha_1 \overline{\Theta}_{\kappa} h_{\kappa}^{-1}\ell^2 \, & F \in \mathcal{F}_B,\\
		\end{aligned}
		\right.
		\quad
		\xi = \left\{\begin{aligned}
			&\alpha_2 \underset{\kappa \in \{\kappa^+,\kappa^-\}}{\mbox{max}}\left(h_{\kappa}^{-1}\ell^2\right) & F \in \mathcal{F}_I,\\
			&\alpha_2 h_{\kappa}^{-1}\ell^2 & F \in \mathcal{F}_B,\\
		\end{aligned}
		\right.\\
		\quad
		\varrho &= \left\{\begin{aligned}
			&\alpha_3 \underset{\kappa \in \{\kappa^+,\kappa^-\}}{\mbox{min}}\left(h_{\kappa}m^{-1}\right) & F \in \mathcal{F}_I,\\
			&\alpha_3 h_{\kappa}m^{-1} & F \in \mathcal{F}_B,\\
		\end{aligned}
		\right.\\
	\end{aligned}   
\end{equation}
where $\alpha_1, \alpha_2, \alpha_3 \in \mathbb{R}$ are positive constants to be properly defined.

To handle the non-linear convective term in the temperature equation, we consider the so-called Temam's modification of the trilinear form that classically models the non-linear advection term \cite{DiPietro2012}. This modification aims to recover the skew-symmetry property of the trilinear form that, in the semi-discrete framework, is generally lost. Indeed, in this framework, the convective velocity is not divergence-free, but only weakly divergence-free. To this aim, we add two consistent terms to the trilinear form. Indeed, we can see the trilinear form $\mathcal{C}_h$ appearing in \eqref{eq:bilinear_forms_discr_DG} as:
$$
\label{eq:Ch_temam}
\mathcal{C}_h(\mathbf{u},T,S) = \widetilde{\mathcal{C}}_h(\mathbf{u},T,S) + \frac{1}{2}(\divh{\mathbf{u}} \, T, S) - \frac{1}{2}\sum_{F \in \mathcal{F}_I} \int_F \jump{\mathbf{u}}_n \cdot \avg{T \, S},
$$
where $\widetilde{\mathcal{C}}_h(\mathbf{u},T,S)$ is the dG-form that discretizes the convection operator with upwind and inflow stabilizations \cite{Bonetti2024}. At the same time, the last two terms are two consistent terms of Temam's modification. We recall the following result.
\begin{lemma}
	\label{lem:Ch_temam_skewsymm}
	For all $\mathbf{v} \in \mathbf{V}_h^{\ell}$, for all $S \in V_h^{\ell}$ it holds:
	$$
	\begin{aligned}
		\mathcal{C}_h(\mathbf{v},S,S) = \frac{1}{2}\sum_{F \in \mathcal{F}_I} \int_F \left\lvert \avg{\mathbf{v}} \cdot \mathbf{n} \right\rvert \jump{S}^2 + \frac{1}{2}\sum_{F \in \mathcal{F}_B} \int_F \left(\left\lvert \mathbf{v} \cdot \mathbf{n} \right\rvert - \mathbf{v} \cdot \mathbf{n} \right) S^2 \geq 0.
	\end{aligned}
	$$
\end{lemma}
We remark that in Lemma~\ref{lem:Ch_temam_skewsymm} we do not recover the skew-symmetry of the trilinear form $\mathcal{C}_h$ due to the presence of the stabilization terms. 
However, thanks to the Temam trick, we can ensure its positivity.

\begin{remark}
	In the discrete formulation above, we consider the same polynomial degree for $\mathbf{V}_h^{\ell}$ and $V_h^{\ell}$ because we are interested in approximation schemes yielding the same accuracy for the velocity and temperature. To ensure inf-sup stability of the discrete system, we need the pressure field $p_h$ to belong to $W_h^{m}$, with $\ell + 1 \geq m$.
\end{remark}

\begin{remark}
	In the trilinear form $\mathcal{C}_h$, we have added two stabilization terms in the spirit of \cite{Bonetti2024} for making the scheme robust to the advection-dominated regime. We highlight that this configuration is relevant in this context, as -- with the DF law -- we intend to describe phenomena in which the velocity of the flow is high. 
\end{remark}


\subsection{The RT-dG-dG discrete problem}
\label{sec:RTdGdG_discrete}
In this section, we introduce the RT-dG-dG discretization, highlighting the differences with respect to the dG-dG-dG one. Given the polynomial degrees of approximation $\ell$, $m$, and the discrete spaces $W_h^m$, $V_h^{\ell}$ defined as before, we introduce the following discrete space for the velocity field:
$$
\mathbf{Z}_h^{m} = \left\{ \mathbf{v} \in \mathbf{Z} : \mathbf{v} |_{\kappa} \in \mathbb{RT}^{m}(\kappa) \  \forall \kappa \in \mathcal{T}_h \right\}.
$$
Then, the RT-dG-dG discretization of problem \eqref{eq:nonlinear_weak_form} reads: \textit{find $(\mathbf{u}_h, p_h, T_h) \in \mathbf{Z}_h^{m} \times W_h^{m} \times V_h^{\ell}$ such that $\forall (\mathbf{v}_h, q_h, S_h) \in \mathbf{Z}_h^{m} \times W_h^{m} \times V_h^{\ell}$ it holds:}
\begin{equation}
\label{eq:nonlinear_discrete_form_RT}
\begin{aligned}
	& \mathcal{M}_{\nu}(T_h, \mathbf{u}_h, \mathbf{v}_h) + \mathcal{M}_{\beta}(\mathbf{u}_h, \mathbf{u}_h, \mathbf{v}_h) +\widehat{\mathcal{B}}_h(p_h, \mathbf{v}_h) - \widehat{\mathcal{B}}_h(q_h, \mathbf{u}_h) \\
	& + \mathcal{A}_h(T_h,S_h) + \widehat{\mathcal{C}}_h(\mathbf{u}_h, T_h, S_h) = ((\mathbf{f},g),  (\mathbf{v}_h, S_h)),
\end{aligned}
\end{equation}
with
$$
\label{eq:bilinear_forms_discr_RT}
\begin{aligned}
& \widehat{\mathcal{B}}_h(\varphi,\mathbf{v}) = - (\varphi,\nabla_h \mkern-2.5mu \cdot \mkern-2.5mu \mathbf{v}) = \mathcal{B}_h(\varphi,\mathbf{v}),\\
& \begin{aligned}
	\widehat{\mathcal{C}}_h (\mathbf{v}, T, S) = \ & \left( \mathbf{v} \cdot \nabla_h T, S \right) + \frac{1}{2}(\divh{\mathbf{v}} \, T, S) - \sum_{F \in \mathcal{F}_I} \int_{F} \left( \mathbf{v} \cdot \jump{T} \right) \avg{S} \\
	& + \frac{1}{2}\sum_{F \in \mathcal{F}} \int_F \ \left| \mathbf{v} \cdot \mathbf{n} \right| \jump{T} \mkern-2.5mu \cdot \mkern-2.5mu \jump{S} - \frac{1}{2}\sum_{F \in \mathcal{F}_B} \int_F \  (\mathbf{v} \cdot \mathbf{n}) \, T \, S = \mathcal{C}_h (\mathbf{v}, T, S)
\end{aligned} \\
\end{aligned}
$$
for all $\mathbf{v} \in \mathbf{Z}_h^m$, $ \varphi \in W_h^m$, and $T, S \in V_h^{\ell}$, and where the remaining bilinear and trilinear forms are defined as before. Notice that $\jump{\mathbf{z}}_n = 0$ for $\mathbf{z} \in \mathbf{Z}_h^m$, which implies that some interface terms vanish. 
For ensuring the inf-sup stability, we take $W_h^{m}$ as discrete space for the pressure, where $m$ is the degree of the RT space of the velocity. 
\begin{remark}
We observe that the discrete space $\mathbf{Z}_h^{m}$ is a subspace of $\mathbf{V}_h^{\ell}$ for $\ell\ge m+1$. 
Following this observation, in Section~\ref{sec:StabilityAnalysis} the theoretical analysis is carried out for the dG-dG-dG formulation and all the results naturally hold for the RT-dG-dG scheme as well. Moreover, the Brezzi-Douglas-Marini (BDM) discrete spaces \cite{Boffi2013,Brezzi1985} can be seen as subspaces of dG-discrete spaces as well. Then, all the results presented in this article can be extended to a BDM-dG-dG scheme.
\end{remark}

\subsection{Linearization}
\label{sec:linearization}
For tackling the non-linear terms appearing in \eqref{eq:nonlinear_discrete_form_DG} (and in \eqref{eq:nonlinear_discrete_form_RT}), we introduce a fixed-point iterative algorithm. Let $k \geq 1$ be the iteration step and let $\mathbf{u}_h^{k-1}, T_h^{k-1}$ be the approximated velocity and temperature fields computed at the $(k-1)^{\text{th}}$ iteration, respectively. Then, at the  $k^{\text{th}}$ step we solve: \textit{find $(\mathbf{u}_h^k, p_h^k, T_h^k) \in \mathbf{V}_h^{\ell} \times W_h^{m} \times V_h^{\ell}$ such that for all $(\mathbf{v}_h, q_h, S_h) \in \mathbf{V}_h^{\ell} \times W_h^{m} \times V_h^{\ell}$ it holds:}
\begin{equation}
\label{eq:linearized_discrete_form}
\begin{aligned}
	&\hspace{-1mm} \mathcal{M}_{\nu}(T_h^{k-1}, \mathbf{u}_h^{k}, \mathbf{v}_h) \hspace{-0.5mm}
	+ \mathcal{M}_{\beta}(\mathbf{u}_h^{k-1}, \mathbf{u}_h^k, \mathbf{v}_h) + \mathcal{B}_h(p_h^k, \mathbf{v}_h) - \mathcal{B}_h(q_h, \mathbf{u}_h^k) 
	\hspace{-0.5mm}+ \mathcal{A}_h(T_h^k,S_h) \\
	&\hspace{-1mm} + \mathcal{C}_h(\mathbf{u}_h^{k-1}, T_h^k, S_h) + \mathcal{D}_u(\mathbf{u}_h^k, \mathbf{v}_h) + \mathcal{D}_p(p_h^k, q_h) = ((\mathbf{f},g),  (\mathbf{v}_h, S_h)).
\end{aligned}
\end{equation}
\begin{remark}
We observe that the fluid flow and heat problems are decoupled. Indeed, a splitting of the two physics is naturally induced by the linearization scheme \eqref{eq:linearized_discrete_form}.
\end{remark}
The linearization algorithm is initialized by solving the linear flow problem and using the resulting velocity field in the temperature equation. Finally, the updated velocity and temperature fields are fed to the next iteration of the scheme. The convergence of algorithm \eqref{eq:linearized_discrete_form} is established in Section~\ref{sec:StabilityAnalysis} under suitable assumptions.

\section{Theoretical analysis}
\label{sec:StabilityAnalysis}
The aim of this section is to establish a stability estimate for
the nonlinear problems \eqref{eq:nonlinear_discrete_form_DG} and \eqref{eq:nonlinear_discrete_form_RT}, to prove the well posedness of the discrete problem \eqref{eq:nonlinear_discrete_form_DG}, and to prove the convergence of the iterative algorithm \eqref{eq:linearized_discrete_form}.

\subsection{Stability estimates}
\label{sec:wp_linearizedproblem}
We focus on problem \eqref{eq:nonlinear_discrete_form_DG} and stress again that these results hold automatically for problem \eqref{eq:nonlinear_discrete_form_RT}. 
We start by introducing notation and results that will be used in the analysis. First, the energy norms are defined $\forall (\mathbf{v},q,S) \in \mathbf{V}_h^{\ell} \times W_h^m \times V_h^{\ell}$ as:

\begin{equation}
	\label{eq:energy_norms}
	\begin{aligned}
		& \triplenorm{\mathbf{v}}_{dG, \text{div}} = \| \mathbf{v} \| + \| \mathbf{v} \|_{L^3} + \| \divh{\mathbf{v}} \| + \left( \sum_{F\in\mathcal{F}_I} \|\xi^{1/2} \jump{\mathbf{v}}_n \|_F^2 \right)^{1/2},\\
		& \| S \|_{dG,T}^2 = \| \sqrt{\boldsymbol{\Theta}} \, \nabla_h S \|^2 + \sum_{F\in\mathcal{F}} \|\sigma^{1/2} \jump{S} \|_F^2,\\
		&
		\| (\mathbf{v},q,S) \|_{\mathcal{E}}^2 = \| \mathbf{v} \|^2 \hspace{-1.5pt} + \| \mathbf{v} \|_{L^3}^3 + \hspace{-0.5mm}\| \divh{\mathbf{v}} \|^2 \hspace{-0.5mm}+ \hspace{-4pt} \sum_{F\in\mathcal{F}_I}\hspace{-0.5mm} \|\xi^{\frac12} \jump{\mathbf{v}}_n \|_F^2 + \mathbb{B}^2 \, \| q \|^2\hspace{-1mm} +\hspace{-0.5mm} \| S \|_{dG,T}^2. 
	\end{aligned}
\end{equation}
Next, we state the following technical lemmata.
\begin{lemma}
	\label{lem:boundcoerc_bil_forms}
	Let Assumptions~\ref{assumption:model_problem} and ~\ref{assumption:mesh} be satisfied and assume that the parameter $\alpha_1$ appearing in \eqref{eq:stabilization_functions} is chosen large enough. Then, the following bounds hold:
	$$
	\begin{aligned}
		\mathcal{A}_h(T,S) \lesssim \ & \|T\|_{dG,T} \|S\|_{dG,T}, \quad && \mathcal{A}_h(T,T) \geq \alpha_T \|T\|_{dG,T}^2 \quad &&\forall \ T,S \in V_h^{\ell},
	\end{aligned}
	$$
	where the coercivity constant $\alpha_T$ and the hidden continuity constant do not depend on the material properties and the discretization parameters.
\end{lemma}
For all $\mathbf{v} \in \mathbf{V}_h^{\ell}$, let $\|\mathbf{v}\|_{dG} = \| \nabla_h \mathbf{v} \|^2 + \sum_{F\in\mathcal{F}} \|\xi^{1/2} \jump{\mathbf{v}} \|_F^2$, we recall the following result and refer to \cite{Antonietti2020} for the proof.
\begin{lemma}
	\label{lem:infsup_stokes}
	Let Assumption~\ref{assumption:mesh} hold and let the polynomial degrees $\ell$ and $m$ satisfy $\ell+1 \geq m$. Then, there exists a positive constant $\mathbb{B}$ independent of the mesh size $h$ (but possibly dependent on $\ell$ and $m$) such that:
	\begin{equation}
		\label{eq:infsup_stokes}
		\underset{\mathbf{0} \neq \mathbf{v}_h \in \mathbf{V}^{\ell}_h}{\sup} \frac{\mathcal{B}_h(\mathbf{v}_h, q_h)}{\|\mathbf{v}_h\|_{dG}} + \mathcal{D}_p(q_h, q_h)^{\frac{1}{2}} \geq \mathbb{B} \|q_h\| \qquad \forall q_h \in W_h^m.
	\end{equation}
\end{lemma}
The following result is an extension of Lemma~\ref{lem:infsup_stokes} and it is needed for controlling the $L^2$-norm of the pressure field.
\begin{proposition}
	\label{lem:infsup}
	Under the assumptions of Lemma~\ref{lem:infsup_stokes}, there exists a positive constant $\mathbb{B}$ independent of the mesh size $h$ (but possibly dependent on $\ell$ and $m$) such that:
	\begin{equation}
		\label{eq:infsup}
		\underset{\mathbf{0} \neq \mathbf{v}_h \in \mathbf{V}^{\ell}_h}{\sup} \frac{\mathcal{B}_h(\mathbf{v}_h, q_h)}{\triplenorm{\mathbf{v}_h}_{dG,\mathrm{div}}} + \mathcal{D}_p(q_h, q_h)^{\frac{1}{2}} \gtrsim \mathbb{B} \|q_h\| \qquad \forall q_h \in W_h^m.
	\end{equation}
\end{proposition}
\begin{proof}
	We start the proof by observing that:
	\begin{equation}
		\label{eq:infsup_norms}
		\triplenorm{\mathbf{v}_h}_{dG,\text{div}} \lesssim \|\mathbf{v}_h\|_{dG} \qquad \forall\,\mathbf{v} \in \mathbf{V}_h^{\ell}.
	\end{equation}
	Indeed, the first two terms of $\triplenorm{\mathbf{v}_h}_{dG,\text{div}}$ are controlled via Poincaré-Sobolev inequalities \cite[Theorem 1.6]{Botti2025}, while the third and the fourth term are trivially controlled by terms of $\|\mathbf{v}_h\|_{dG}$.
	Then, the thesis directly follows from \eqref{eq:infsup_stokes} and \eqref{eq:infsup_norms}.
\end{proof}
The next result is instrumental for the derivation of an \textit{a-priori} stability estimate for the nonlinear discrete problem. The target of the Lemma is to obtain and prove two basic estimates. 
\begin{lemma}
	\label{lem:basic_estimates}
	Let Assumptions~\ref{assumption:model_problem} and ~\ref{assumption:mesh} be satisfied and assume that the parameter $\alpha_1$ appearing in \eqref{eq:stabilization_functions} is chosen large enough. Then, the solution $(\mathbf{u}_h,p_h,T_h) \in \mathbf{V}_h^{\ell} \times W_h^{m} \times V_h^{\ell}$ to problem \eqref{eq:nonlinear_discrete_form_DG} satisfies the following estimates:
	\begin{subequations}
		\label{eq:basic_est}
		\begin{align}
			& \,\,(i) \, \left\| \sqrt{\frac{\nu_m}{2k_M}} \, \mathbf{u}_h \right\|^2 + \|\sqrt[3]{\beta} \, \mathbf{u}_h \|_{L^3}^3 + \frac{\alpha_T}{2}\| T_h \|_{dG,T}^2 \leq \frac{k_M}{2 \nu_m}\| \mathbf{f} \|^2 + \frac{C_p^2}{2 \alpha_T \theta_m}\| g \|^2, \label{eq:basic_est_i} \\
			& \begin{aligned} (ii) \, 
				& \left\| \sqrt{\frac{\nu_m}{2k_M}} \, \mathbf{u}_h \right\|^2 + \|\sqrt[3]{\beta} \, \mathbf{u}_h \|_{L^3}^3 + \frac{1}{3}\| \divh{\mathbf{u}_h} \|^2 + \frac{1}{4}\sum_{F \in \mathcal{F}_I} \| \xi^{1/2} \, \jump{\mathbf{u}_h}_n \|_F^2 \\
				& + \frac{\alpha_T}{2}\| T_h \|_{dG,T}^2 \leq \mathcal{M}_{\nu}(T_h, \mathbf{u}_h, \mathbf{u}_h) + \mathcal{M}_{\beta}(\mathbf{u}_h, \mathbf{u}_h, \mathbf{u}_h) + \mathcal{A}_h(T_h,T_h) \\
				& + \mathcal{C}_h(\mathbf{u}_h, T_h, T_h) + \mathcal{D}_u(\mathbf{u}_h, \mathbf{u}_h). \label{eq:basic_est_ii}
			\end{aligned}
		\end{align}
	\end{subequations}
	where $C_p$ is the constant of the Poincaré inequality.
\end{lemma}
\begin{proof}
	We start the proof by taking $(\mathbf{v}_h, q_h, S_h) = (\mathbf{u}_h, p_h, T_h)$ as test functions in \eqref{eq:nonlinear_discrete_form_DG}. We obtain:
	\begin{equation}
		\label{eq:basic_est_step1}
		\begin{aligned}
			& \mathcal{M}_{\nu}(T_h, \mathbf{u}_h, \mathbf{u}_h) + \mathcal{M}_{\beta}(\mathbf{u}_h, \mathbf{u}_h, \mathbf{u}_h) + \mathcal{A}_h(T_h,T_h) + \mathcal{C}_h(\mathbf{u}_h, T_h, T_h) \\
			& + \mathcal{D}_u(\mathbf{u}_h, \mathbf{u}_h) + \mathcal{D}_p(p_h, p_h) = ((\mathbf{f},g),  (\mathbf{u}_h, T_h)).
		\end{aligned}
	\end{equation}
	Now, we observe that, by using Assumption~\ref{assumption:model_problem}, Lemma~\ref{lem:Ch_temam_skewsymm}, and Lemma~\ref{lem:boundcoerc_bil_forms} the following results hold:
	\begin{equation}
		\label{eq:basic_est_step2}
		\begin{aligned}
			& \mathcal{M}_{\nu}(T_h, \mathbf{u}_h, \mathbf{u}_h) \geq \left\| \sqrt{\nu_m / k_M} \, \mathbf{u}_h \right\|^2, && \quad \mathcal{M}_{\beta}(\mathbf{u}_h, \mathbf{u}_h, \mathbf{u}_h) = \| \sqrt[3]{\beta} \, \mathbf{u}_h \|_{L^3}^3, \\
			& \mathcal{A}_h(T_h,T_h) \geq \alpha_T \| T_h \|_{dG,T}^2, && \quad \mathcal{C}_h(\mathbf{u}_h, T_h, T_h) \geq 0, \\
			& \mathcal{D}_u(\mathbf{u}_h, \mathbf{u}_h) = \sum_{F \in \mathcal{F}_I} \| \xi^{1/2} \, \jump{\mathbf{u}_h}_n \|_F^2 \geq 0, && \quad\mathcal{D}_p(p_h, p_h) = \sum_{F \in \mathcal{F}_I} \int_F \varrho \jump{p_h}^2 \geq 0.
		\end{aligned}
	\end{equation}
	By exploiting \eqref{eq:basic_est_step2} and using Cauchy-Schwarz and Young inequalities for the right hand side of \eqref{eq:basic_est_step1}, we obtain \eqref{eq:basic_est_i}.
	For proving \eqref{eq:basic_est_ii}, we need to control also the divergence of the discrete velocity field. To this aim, we test problem \eqref{eq:nonlinear_discrete_form_DG} with $(\mathbf{v}_h, q_h, S_h) = (\mathbf{0}, -\divh{\mathbf{u}_h}, 0)$ and we find:
	$$
	\mathcal{B}_h(-\divh{\mathbf{u}_h}, \mathbf{u}_h) - \mathcal{D}_p(p_h, \divh{\mathbf{u}_h}) = 0.
	$$
	Now, by using Cauchy-Schwarz, Young, and trace-inverse inequality (cf. \eqref{eq:RobQSTPE_trace_inverse_ineq}), we get:
	\begin{equation}
		\label{eq:basic_est_step4}
		\left(1 - \frac{1}{2\epsilon_1} - \frac{1}{2\epsilon_2}\right) \| \divh{\mathbf{u}_h}\|^2 - \frac{\epsilon_1}{2} \sum_{F \in \mathcal{F}_I} \| \xi^{1/2} \, \jump{\mathbf{u}_h}_n \|_F^2 - \frac{\epsilon_2}{2} \sum_{F \in \mathcal{F}_I} \int_F \varrho \jump{p_h}^2 \leq 0.
	\end{equation}
	Last, we fix $\epsilon_1 = \epsilon_2 = 3/2$ and we combine \eqref{eq:basic_est_step2} and \eqref{eq:basic_est_step4} to obtain \eqref{eq:basic_est_ii}.
\end{proof}
We are now ready to state the main result of this section, establishing the stability estimate for the nonlinear discrete Darcy-Forchheimer flow problem coupled with an advection-diffusion equation for the temperature.
\begin{theorem}
	\label{thm:stab_est}
	Let the assumptions of Lemmata \ref{lem:infsup}, \ref{lem:boundcoerc_bil_forms}, and \ref{lem:basic_estimates} be satisfied. Then, the solution $(\mathbf{u}_h,p_h,T_h) \in \mathbf{V}_h^{\ell} \times W_h^{m} \times V_h^{\ell}$ to \eqref{eq:nonlinear_discrete_form_DG} satisfies the a-priori bound
	\begin{equation}
		\label{eq:stab_est}
		\|\,(\mathbf{u}_h, p_h, T_h)\,\|_{\mathcal{E}}^2 \lesssim \| \mathbf{f} \|^2 + \|g\|^2 + \left(\|\mathbf{f}\|^2 + \|g\|^2\right)^{\frac{2}{3}} + \left(\|\mathbf{f}\|^2 + \|g\|^2\right)^{\frac{4}{3}}, 
	\end{equation}
	where the hidden constant is independent of the mesh size $h$.
\end{theorem}
\begin{proof}
	The first step of the proof consists in using the inf-sup condition (cf. Lemma~\ref{lem:infsup}) to find a bound for the pressure field. Taking $(\mathbf{v}_h, q_h, S_h) = (\mathbf{v}_h, 0, 0)$ in \eqref{eq:nonlinear_discrete_form_DG} we have:
	$$
	\begin{aligned}
		-\mathcal{B}_h(p_h, \mathbf{v}_h) = & \mathcal{M}_{\nu}(T_h, \mathbf{u}_h, \mathbf{v}_h) + \mathcal{M}_{\beta}(\mathbf{u}_h, \mathbf{u}_h, \mathbf{v}_h) + \mathcal{D}_u(\mathbf{u}_h, \mathbf{v}_h) -(\mathbf{f},\mathbf{v}_h).
	\end{aligned}
	$$
	By plugging the previous identity into \eqref{eq:infsup} and by using Cauchy-Schwarz and H\"older inequalities we obtain:
	\begin{equation}
		\label{eq:stab_est_step2}
		\begin{aligned}
			\frac{\mathbb{B}^2}{2} \, \|p_h\|^2 \leq \, & \left( \underset{\mathbf{0} \neq \mathbf{v}_h \in \mathbf{V}^{\ell}_h}{\sup} \frac{\mathcal{M}_{\nu}(T_h, \mathbf{u}_h, \mathbf{v}_h) + \mathcal{M}_{\beta}(\mathbf{u}_h, \mathbf{u}_h, \mathbf{v}_h) + \mathcal{D}_u(\mathbf{u}_h, \mathbf{v}_h) -(\mathbf{f},\mathbf{v}_h)}{\triplenorm{\mathbf{v}_h}_{dG,\text{div}}} \right)^2 \\
			& + \mathcal{D}_p(p_h, p_h)\\[5pt]
			\leq \, & \left\| (\nu_M / k_m) \, \mathbf{u}_h\right\|^2 + \|\sqrt{\beta} \mathbf{u}_h\|_{L^3}^4 + \triplenorm{\mathbf{u}_h}_{dG,\text{div}}^2 + \| \mathbf{f} \|^2 + \mathcal{D}_p(p_h, p_h),
		\end{aligned}
	\end{equation}
	where the second bound follows observing that
	$$
	\begin{aligned}
		\mathcal{M}_{\beta}(\mathbf{u}_h, \mathbf{u}_h, \mathbf{v}_h) = \left( \beta |\mathbf{u}_h| \mathbf{u}_h, \mathbf{v}_h\right) \leq  \| \beta \, |\mathbf{u}_h| \mathbf{u}_h \, \|_{L^{3/2}} \, \| \mathbf{v}_h \|_{L^3} \leq \| \sqrt{\beta} \mathbf{u}_h \|_{L^{3}}^2 \, \triplenorm{\mathbf{v}_h}_{dG, \text{div}}.
	\end{aligned}
	$$
	In the second step of the proof, we find a bound for $\| \sqrt{\beta} \, \mathbf{u}_h\|_{L^3}^4$. From \eqref{eq:basic_est_i}, we can easily see that:
	$$
	\|\sqrt{\beta} \, \mathbf{u}_h \|_{L^3}^3 = \sqrt{\beta} \|\sqrt[3]{\beta} \, \mathbf{u}_h \|_{L^3}^3 \leq \frac{k_M \, \sqrt{\beta}}{2 \nu_m}\| \mathbf{f} \|^2 + \frac{C_p^2 \, \sqrt{\beta}}{2 \alpha_T \theta_m}\| g \|^2, 
	$$
	and therefore
	\begin{equation}
		\label{eq:stab_est_step3}
		\|\sqrt{\beta} \, \mathbf{u}_h \|_{L^3}^4 \leq \left( \frac{k_M \, \sqrt{\beta}}{2 \nu_m}\| \mathbf{f} \|^2 + \frac{C_p^2 \, \sqrt{\beta}}{2 \alpha_T \theta_m}\| g \|^2 \right)^{\frac{4}{3}}.
	\end{equation}
	We now consider \eqref{eq:basic_est_ii}, we add $(\epsilon \mathbb{B}^2/2) \, \|p_h\|^2$ to the left and right hand side, and we use \eqref{eq:stab_est_step2}:
	$$
	\begin{aligned}
		& \left\| \sqrt{\frac{\nu_m}{2 k_M}} \, \mathbf{u}_h \right\|^2 \hspace{-1pt} + \|\sqrt[3]{\beta} \, \mathbf{u}_h \|_{L^3}^3 + \frac{1}{3}\| \divh{\mathbf{u}_h} \|^2 + \frac{1}{4}\sum_{F \in \mathcal{F}_I} \| \xi^{1/2} \, \jump{\mathbf{u}_h}_n \|_F^2 + \frac{\alpha_T}{2}\| T_h \|_{dG,T}^2 \\
		& + \frac{\epsilon \mathbb{B}^2}{2} \, \|p_h\|^2 \leq \mathcal{M}_{\nu}(T_h, \mathbf{u}_h, \mathbf{u}_h) + \mathcal{M}_{\beta}(\mathbf{u}_h, \mathbf{u}_h, \mathbf{u}_h) + \mathcal{A}_h(T_h,T_h) +  \mathcal{C}_h(\mathbf{u}_h, T_h, T_h) \\
		& + \mathcal{D}_u(\mathbf{u}_h, \mathbf{u}_h) + \mathcal{D}_p(p_h, p_h) + \epsilon \left( \left\| \frac{\nu_M}{k_m} \, \mathbf{u}_h\right\|^2 + \|\sqrt{\beta} \mathbf{u}_h\|_{L^3}^4 + \triplenorm{\mathbf{u}_h}_{dG,\text{div}}^2 + \| \mathbf{f} \|^2 \right). 
	\end{aligned}
	$$
	We observe that the first six terms at right hand side are equal to $((\mathbf{f},g),(\mathbf{u}_h,T_h))$ (cf. \eqref{eq:basic_est_step1}), we apply Cauchy-Schwarz and Young inqualities on it and we exploit \eqref{eq:stab_est_step3} for bounding the eighth term at right hand side:
	$$
	\begin{aligned}
		& \left\| \sqrt{\frac{\nu_m}{4 k_M}} \, \mathbf{u}_h \right\|^2 \hspace{-1pt} + \|\sqrt[3]{\beta} \, \mathbf{u}_h \|_{L^3}^3 + \frac{1}{3}\| \divh{\mathbf{u}_h} \|^2 + \frac{1}{4}\sum_{F \in \mathcal{F}_I} \| \xi^{1/2} \, \jump{\mathbf{u}_h}_n \|_F^2 + \frac{\alpha_T}{4}\| T_h \|_{dG,T}^2 \\
		& + \frac{\epsilon \mathbb{B}^2}{2} \, \|p_h\|^2 \leq \frac{k_M}{\nu_m}\| \mathbf{f} \|^2 + \frac{C_p^2}{\alpha_T \theta_m}\| g \|^2 + \epsilon \left( \frac{k_M \, \sqrt{\beta}}{2 \nu_m}\| \mathbf{f} \|^2 + \frac{C_p^2 \, \sqrt{\beta}}{2 \alpha_T \theta_m}\| g \|^2 \right)^{\frac{4}{3}} + \epsilon \| \mathbf{f}\|^2 \\
		& + \epsilon \left( \left\| \frac{\nu_M}{k_m} \, \mathbf{u}_h\right\|^2 + \triplenorm{\mathbf{u}_h}_{dG,\text{div}}^2 \right). 
	\end{aligned}
	$$
	Finally, we choose the auxiliary parameter $\epsilon$ to be:
	$$
	\epsilon = \frac{1}{4} \, \min \left( \frac{\nu_m \, k_m^2}{4 \, k_M \, \nu_M^2}, \frac{1}{3}, \frac{1}{4} \right),
	$$
	we bound the $\| \mathbf{u}_h \|_{L^3}^2$ contribution appearing from the last term of right hand side as done in \eqref{eq:stab_est_step3} and this concludes the proof.
\end{proof}


\subsection{Existence}
Existence of discrete solutions to problem \eqref{eq:nonlinear_discrete_form_DG} is established in the framework of Browder--Minty theory for mapped coercive nonlinear operators \cite{Becker2025}. 

First, we rewrite problem \eqref{eq:bilinear_forms_discr_DG} in the form 
\begin{equation}\label{eq:pb_operatorform}
	\mathbb{A}_h(X_h) = F_h, \quad\text{with}\quad X_h=(\mathbf{u}_h, p_h, T_h) \in \mathbf{V}_h^{\ell} \times W_h^{m} \times V_h^{\ell}
\end{equation}
and, for all $Y_h=(\mathbf{v}_h, q_h, S_h) \in \mathbf{V}_h^{\ell} \times W_h^{m} \times V_h^{\ell}$, the nonlinear operator $\mathbb{A}_h$ and the linear functional $F_h$ are defined as
$$
\begin{aligned}
	(\mathbb{A}_h(X_h), Y_h) &=   \mathcal{M}_{\nu}(T_h, \mathbf{u}_h, \mathbf{v}_h) + \mathcal{M}_{\beta}(\mathbf{u}_h, \mathbf{u}_h, \mathbf{v}_h) + \mathcal{B}_h(p_h, \mathbf{v}_h) - \mathcal{B}_h(q_h, \mathbf{u}_h) \\ 
	&\quad + \mathcal{A}_h(T_h,S_h)  
	+ \mathcal{C}_h(\mathbf{u}_h, T_h, S_h) + \mathcal{D}_u(\mathbf{u}_h, \mathbf{v}_h) + \mathcal{D}_p(p_h, q_h), \\
	(F_h, Y_h)     &= ((\mathbf{f},0, g),  (\mathbf{v}_h,q_h, S_h)).
\end{aligned}
$$
Owing to Assumption \ref{ass:pbdata}, we observe that $F_h$ and all the bilinear form in the definition of $\mathbb{A}_h$ are continuous. Moreover, proceeding as in the proof of Theorem \ref{thm:stab_est} we can infer the continuity of the nonlinear functions $\mathcal{M}_{\nu}$, $\mathcal{M}_{\beta}$, and $\mathcal{C}_h$:
\begin{equation}\label{eq:cont_nonlinear}
	\begin{aligned}
		\mathcal{M}_{\nu}(T_h, \mathbf{u}_h, \mathbf{v}_h) &\leq
		\left\| (\nu_M / k_m) \, \mathbf{u}_h\right\|\,
		\|\mathbf{v}_h\|
		\qquad\quad\,\forall\, (\mathbf{u}_h, \mathbf{v}_h, T_h) \in \mathbf{V}_h^{\ell} \times \mathbf{V}_h^{\ell} \times V_h^{\ell},
		\\
		\mathcal{M}_{\beta}(\mathbf{u}_h, \mathbf{u}_h, \mathbf{v}_h) &\leq \| \sqrt{\beta} \mathbf{u}_h \|_{L^{3}}^2 \, 
		\|\mathbf{v}_h\|_{L^{3}}
		\qquad\quad\;\;\,\forall\, (\mathbf{u}_h, \mathbf{v}_h) \in \mathbf{V}_h^{\ell} \times \mathbf{V}_h^{\ell},
		\\
		\mathcal{C}_{h}(\mathbf{u}_h, T_h, S_h)
		&\lesssim \|\mathbf{u}_h \|_{L^{3}} \, \|T_{h}\|_{dG,T} \, \|S_h\|_{dG,T}
		\;\,\forall\, (\mathbf{u}_h, T_h, S_h) \in \mathbf{V}_h^{\ell} \times V_h^{\ell} \times V_h^{\ell}.
	\end{aligned}
\end{equation}

We are now ready to establish the existence of discrete solutions.
\begin{theorem}
	The operator $\mathbb{A}_h$ in \eqref{eq:pb_operatorform} is linearly mapped coercive, i.e. there exist a linear bijection $\Phi:\mathbf{V}_h^{\ell} \times W_h^{m} \times V_h^{\ell}\to \mathbf{V}_h^{\ell} \times W_h^{m} \times V_h^{\ell}$ such that
	\begin{equation}\label{eq:mappedcoercivity}
		\lim_{\| X_h \|_{\mathcal{E}}\to\infty} 
		\frac{(\mathbb{A}_h(X_h), \Phi(X_h))}{\| X_h \|_{\mathcal{E}}} = +\infty.
	\end{equation}
	As a result, problem \eqref{eq:nonlinear_discrete_form_DG} has a solution.
\end{theorem}
\begin{proof}
	Owing to the continuity of the operator $\mathbb{A}_h$ and right-hand side $F_h$, the existence of a solution to the nonlinear problem \eqref{eq:nonlinear_discrete_form_DG} follows from \cite[Theorem 4]{Becker2025} as a consequence of the mapped coercivity property \eqref{eq:mappedcoercivity}.
	
	First, given $p_h\in W_h^m$, we observe that Proposition \ref{lem:infsup} and the fact that $\mathbf{V}_h^{\ell}$ is finite dimensional imply the existence of a vector field $\mathbf{w}_h\in \mathbf{V}_h^{\ell}$ such that 
	\begin{equation}
		\label{eq:rightinverse1}
		\mathcal{B}_h(\mathbf{w}_h, p_h) + \mathcal{D}_p(p_h, p_h)^{\frac12} \geq \mathbb{B} \|p_h\| \qquad \text{and} \qquad
		{\triplenorm{\mathbf{w}_h}_{dG,\mathrm{div}}} \le 1.
	\end{equation}
	Taking $\mathbf{z}_h = \|p_h\|^\frac12 \mathbf{w}_h\in \mathbf{V}_h^{\ell}$, we have ${\triplenorm{\mathbf{z}_h}_{dG,\mathrm{div}}}=\| p_h\|^{\frac12}{\triplenorm{\mathbf{w}_h}_{dG,\mathrm{div}}} \le \| p_h\|^{\frac12}$ and
	\[
	\begin{aligned}
		\mathbb{B}\|p_h\|^{\frac32} &\leq
		\|p_h\|^{\frac12} \mathcal{B}_h(\mathbf{w}_h, p_h) +
		\|p_h\|^{\frac12} \mathcal{D}_p(p_h, p_h)^{\frac12} \leq
		\mathcal{B}_h(\mathbf{z}_h, p_h) +
		\|p_h\|^{\frac12} \mathcal{D}_p(p_h, p_h)^{\frac12}
		\\ &\leq
		\mathcal{B}_h(\mathbf{z}_h, p_h) + \frac{2}{3\mathbb{B}}\mathcal{D}_p(p_h, p_h)^{\frac34} + \frac{\mathbb{B}}3\|p_h\|^{\frac32},
	\end{aligned}
	\]
	where, to infer the last inequality, we have used the generalized Young inequality. Hence, for each $p_h\in W_h^m$, we have shown the existence of $\mathbf{z}_h \in \mathbf{V}_h^{\ell}$ such that
	\begin{equation}
		\label{eq:rightinverse2}
		\mathcal{B}_h(\mathbf{z}_h, p_h) + \mathcal{D}_p(p_h, p_h)^{\frac34} \geq C_{\mathbb{B}} \|p_h\|^{\frac32} \qquad \text{and} \qquad
		{\triplenorm{\mathbf{z}_h}_{dG,\mathrm{div}}} \le \| p_h\|^{\frac12},
	\end{equation}
	with $C_{\mathbb{B}} = \mathbb{B} \min\{2/3; \mathbb{B} \}$.
	Without loss of generality, since we are considering $\| X_h \|_{\mathcal{E}}\to\infty$, we assume that $\|p_h\|\geq 1$ and $\mathcal{D}_p(p_h, p_h)\geq1$, so that 
	$\mathcal{D}_p(p_h, p_h)^{\frac34}\le \mathcal{D}_p(p_h, p_h)$.
	Then, we define the linear map $\Phi$ as 
	$
	\Phi(\mathbf{u}_h, p_h, T_h)
	= (\mathbf{u}_h + \gamma\mathbf{z}_h, p_h, T_h)
	$, with $\gamma>0$ to be selected. Using \eqref{eq:basic_est_ii} and \eqref{eq:rightinverse2},
	it is inferred that
	\begin{equation}\label{eq:coer1}
		\begin{aligned}
			& (\mathbb{A}_h(X_h), \Phi(X_h)) \geq\\
			&\quad
			\left\| \sqrt{\frac{\nu_m}{2k_M}} \, \mathbf{u}_h \right\|^2 + \|\sqrt[3]{\beta} \, \mathbf{u}_h \|_{L^3}^3 
			+ \frac{1}{4} \sum_{F \in \mathcal{F}_I} \| \xi^{\frac12} \, \jump{\mathbf{u}_h}_n \|_F^2 + \frac{\alpha_T}{2}\| T_h \|_{dG,T}^2 
			\\&\quad +
			\gamma C_{\mathbb{B}} \|p_h\|^{\frac32} + \gamma\left(
			\mathcal{M}_{\nu}(T_h, \mathbf{u}_h, \mathbf{z}_h) +
			\mathcal{M}_{\beta}(\mathbf{u}_h, \mathbf{u}_h, \mathbf{z}_h)  + \mathcal{D}_u(\mathbf{u}_h, \mathbf{z}_h) \right).
		\end{aligned}
	\end{equation}
	Additionally, the continuity of the nonlinear function $\mathcal{M}_{\nu}$, $\mathcal{M}_{\beta}$ in \eqref{eq:cont_nonlinear} together with the second inequality in \eqref{eq:rightinverse2} and the generalized Young inequality give
	\[
	\begin{aligned}
		\mathcal{M}_{\nu}(T_h, \mathbf{u}_h, \mathbf{z}_h) &\geq
		-\left\| (\nu_M / k_m) \, \mathbf{u}_h\right\|\,
		{\triplenorm{\mathbf{z}_h}_{dG,\mathrm{div}}}
		\geq -C_1\left\| \sqrt{\nu_m /2 k_M} \, \mathbf{u}_h\right\|^2
		-\frac{C_{\mathbb{B}}}6\|p_h\|,
		\\
		\mathcal{M}_{\beta}(\mathbf{u}_h, \mathbf{u}_h, \mathbf{z}_h) &\geq -\| 
		\sqrt{\beta} \mathbf{u}_h \|_{L^{3}}^2 \, 
		{\triplenorm{\mathbf{z}_h}_{dG,\mathrm{div}}}
		\geq -C_2\|\sqrt[3]{\beta} \, \mathbf{u}_h \|_{L^3}^3
		-\frac{C_{\mathbb{B}}}6\|p_h\|^{\frac32},
		\\
		\mathcal{D}_u(\mathbf{u}_h, \mathbf{z}_h) 
		\geq &-\hspace{-1mm}\left(\sum_{F \in \mathcal{F}_I}\hspace{-1mm} \| \xi^{\frac12}  \jump{\mathbf{u}_h}_n \|_F^2\right)^{\frac12}
		\hspace{-1mm}\|p_h\|^{\frac12}\geq
		-\frac{C_3}{4}\hspace{-1mm} \sum_{F \in \mathcal{F}_I} \hspace{-1mm}\| \xi^{\frac12}  \jump{\mathbf{u}_h}_n \|_F^2
		-\frac{C_{\mathbb{B}}}6\|p_h\|,
	\end{aligned}
	\]
	with $C_1,C_2, C_3 >0$ possibly depending on the material coefficients but independent of the discretization parameters. 
	Plugging the previous bounds into \eqref{eq:coer1} and recalling that $\|p_h\|\geq1$, we obtain
	$$
	\begin{aligned}
		(\mathbb{A}_h(X_h), \Phi(X_h)) &\geq
		(1-\gamma C_1)\left\| \sqrt{\frac{\nu_m}{2k_M}} \mathbf{u}_h \right\|^2 + (1-\gamma C_2)\|\sqrt[3]{\beta} \, \mathbf{u}_h \|_{L^3}^3 
		\\& + \frac{1-\gamma C_3}{4} \sum_{F \in \mathcal{F}_I} \| \xi^{\frac12} \, \jump{\mathbf{u}_h}_n \|_F^2 
		+ \frac{\alpha_T}{2}\| T_h \|_{dG,T}^2 +
		\frac{\gamma C_{\mathbb{B}}}2 \|p_h\|^{\frac32}.
	\end{aligned}
	$$
	Choosing $\gamma>0$ such that $\gamma\max\{C_1,C_2,C_3\}<1$, the previous bound implies \eqref{eq:mappedcoercivity}.
\end{proof}

\subsection{Uniqueness}
The aim of this section is to prove that, under a suitable assumption, the solution of the nonlinear problem \eqref{eq:nonlinear_discrete_form_DG} is unique.

\begin{theorem}\label{thm:uniqueness}
	Let $(\mathbf{u}_{h}, p_{h}, T_{h}) \in \mathbf{V}_h^{\ell} \times W_h^m \times V_h^{\ell}$ be the solution of \eqref{eq:nonlinear_discrete_form_DG} and $C_{\text{s},q}$ be the discrete Sobolev embedding constant of \cite[Theorem 1.5]{Botti2025} allowing to bound $\| \cdot \|_{L^q}$ with $\| \cdot \|_{dG,T}$. Then, if the discrete velocity and temperature fields are such that
	$$
	\begin{aligned}
		& \|\mathbf{u}_{h}\|_{L^3} \, \|T_{h}\|_{dG,T} \leq \frac{\alpha_T \, \nu_m }{C_{\text{s},3} \, C_{\text{s},6} \, k_M} \sqrt{\frac{k_m}{C_{\text{s},6} \, L_{\nu}}},
	\end{aligned}
	$$
	the solution of \eqref{eq:nonlinear_discrete_form_DG} is unique.
\end{theorem}

\begin{proof}
	We consider $(\mathbf{u}_{h,1}, p_{h,1}, T_{h,1})$ and $(\mathbf{u}_{h,2}, p_{h,2}, T_{h,2})$ to be two different sets of solutions for problem \eqref{eq:nonlinear_discrete_form_DG}. We denote the differences between them by
	$$
	(\varepsilon_{\mathbf{u}}, \varepsilon_p, \varepsilon_T) = (\mathbf{u}_{h,1} - \mathbf{u}_{h,2}, p_{h,1} - p_{h,2}, T_{h,1} - T_{h,2})
	$$
	We consider problem \eqref{eq:nonlinear_discrete_form_DG} solved by $(\mathbf{u}_{h,1}, p_{h,1}, T_{h,1})$, $(\mathbf{u}_{h,2}, p_{h,2}, T_{h,2})$ and we subtract the resulting equations:
	\begin{equation}
		\label{eq:uniq1}
		\begin{aligned}
			& \mathcal{M}_{\nu}(T_{h,1}, \mathbf{u}_{h,1}, \mathbf{v}_h) - \mathcal{M}_{\nu}(T_{h,2}, \mathbf{u}_{h,2}, \mathbf{v}_h) + \mathcal{M}_{\beta}(\mathbf{u}_{h,1}, \mathbf{u}_{h,1}, \mathbf{v}_h) \\
			& - \mathcal{M}_{\beta}(\mathbf{u}_{h,2}, \mathbf{u}_{h,2}, \mathbf{v}_h) + \mathcal{B}_h(\varepsilon_p, \mathbf{v}_h) - \mathcal{B}_h(q_h, \varepsilon_{\mathbf{u}}) + \mathcal{A}_h(\varepsilon_T,S_h) \\
			& + \mathcal{C}_h(\mathbf{u}_{h,1}, T_{h,1}, S_h) - \mathcal{C}_h(\mathbf{u}_{h,2}, T_{h,2}, S_h) + \mathcal{D}_u(\varepsilon_{\mathbf{u}}, \mathbf{v}_h) + \mathcal{D}_p(\varepsilon_p, q_h) = 0
		\end{aligned}
	\end{equation}
	We can divide the proof into three steps.
	\medskip
	
	\noindent $(i)$ We start by focusing on the flow equation taking $(\mathbf{v}_h, q_h, S_h) = (\varepsilon_{\mathbf{u}}, \varepsilon_p, 0)$ in \eqref{eq:uniq1}:
	\begin{equation}
		\label{eq:uniq2}
		\begin{aligned}
			& (\nu(T_{h,1}) \mathbf{K}^{-1} \, \varepsilon_{\mathbf{u}}, \varepsilon_{\mathbf{u}}) + (\beta(|\mathbf{u}_{h,1}|\mathbf{u}_{h,1} - |\mathbf{u}_{h,2}|\mathbf{u}_{h,2}),\varepsilon_{\mathbf{u}}) + \mathcal{D}_u(\varepsilon_{\mathbf{u}}, \varepsilon_{\mathbf{u}}) \\
			& + \mathcal{D}_p(\varepsilon_p, \varepsilon_p) = ((\nu(T_{h,2}) - \nu(T_{h,1})) \mathbf{K}^{-1} \, \mathbf{u}_{h,2}, \varepsilon_{\mathbf{u}}).
		\end{aligned}
	\end{equation}
	Using \eqref{eq:basic_est_step2}, the Lipschitz continuity of $\nu(\cdot)$, and H\"older inequalities, we obtain
	\begin{equation}
		\label{eq:uniq3}
		\begin{aligned}
			\left\| \sqrt{\frac{\nu_m}{k_M}} \, \varepsilon_{\mathbf{u}} \right\|^2 + \left\| \sqrt[3]{\beta} \, \varepsilon_{\mathbf{u}} \right\|_{L^3}^3 + \mathcal{D}_u(\varepsilon_{\mathbf{u}}, \varepsilon_{\mathbf{u}}) + \mathcal{D}_p(\varepsilon_p, \varepsilon_p) \leq & \, \frac{L_{\nu}}{k_m} \|\varepsilon_T \|_{L^6} \|\mathbf{u}_{h,2}\|_{L^3} \|\varepsilon_{\mathbf{u}} \|.
		\end{aligned}
	\end{equation}
	Observing that the last three terms at left hand side of \eqref{eq:uniq3} are positive, using Young inequality, and exploiting \cite[Theorem 1.5]{Botti2025} we have
	\begin{equation}
		\label{eq:uniq4}
		\begin{aligned}
			\left\| \sqrt{\frac{\nu_m}{2 \, k_M}} \, \varepsilon_{\mathbf{u}}  \right\|^2 \leq & \, \frac{C_{\text{s},6} \, L_{\nu} \, k_M}{2 \nu_m \, k_m} \|\varepsilon_T \|_{dG,T}^2 \, \|\mathbf{u}_{h,2}\|_{L^3}^2.
		\end{aligned}
	\end{equation}
	
	We now consider the temperature equation and $(\mathbf{v}_h, q_h, S_h) = (\mathbf{0}, 0, \varepsilon_T)$ as test functions in \eqref{eq:uniq1}:
	\begin{equation}
		\label{eq:uniq5}
		\begin{aligned}
			& \mathcal{A}_h(\varepsilon_T,\varepsilon_T) + \mathcal{C}_h(\mathbf{u}_{h,1}, \varepsilon_T, \varepsilon_T) = - \mathcal{C}_h(\varepsilon_{\mathbf{u}}, T_{h,2}, \varepsilon_T)
		\end{aligned}
	\end{equation}
	Using Lemma~\ref{lem:Ch_temam_skewsymm}, Lemma~\ref{lem:boundcoerc_bil_forms}, H\"older's inequality, and \cite[Theorem 1.5-1.6]{Botti2025}, we get
	\begin{equation}
		\label{eq:uniq6}
		\begin{aligned}
			& \alpha_T \|\varepsilon_T\|_{dG,T}^2 \leq \|\varepsilon_{\mathbf{u}}\| \, \|T_{h,2}\|_{dG,T} \|\varepsilon_T\|_{L^6} \leq C_{\text{s},3} \, C_{\text{s},6} \, \|\varepsilon_{\mathbf{u}}\| \, \|T_{h,2}\|_{dG,T} \, \|\varepsilon_T\|_{dG,T}
		\end{aligned}
	\end{equation}
	We now plug \eqref{eq:uniq4} into \eqref{eq:uniq6}:
	\begin{equation}
		\label{eq:uniq7}
		\begin{aligned}
			& \alpha_T \|\varepsilon_T\|_{dG,T}^2 \leq \frac{C_{\text{s},3} \, C_{\text{s},6} \, k_M}{\nu_m}\sqrt{\frac{C_{\text{s},6} \, L_{\nu}}{k_m}} \, \|\mathbf{u}_{h,2}\|_{L^3} \, \|T_{h,2}\|_{dG,T} \, \|\varepsilon_T\|_{dG,T}^2,
		\end{aligned}
	\end{equation}
	that is equivalent to:
	\begin{equation}
		\label{eq:uniq8}
		\begin{aligned}
			& \|\varepsilon_T\|_{dG,T}^2 \left(\alpha_T - \frac{C_{\text{s},3} \, C_{\text{s},6} \, k_M}{\nu_m}\sqrt{\frac{C_{\text{s},6} \, L_{\nu}}{k_m}} \, \|\mathbf{u}_{h,2}\|_{L^3} \, \|T_{h,2}\|_{dG,T} \right) \leq  0.
		\end{aligned}
	\end{equation}
	If we consider $\|\mathbf{u}_{h,2}\|_{L^3}$, $\|T_{h,2}\|_{dG,T}$, such that
	\begin{equation}
		\label{eq:uniq9}
		\begin{aligned}
			& \|\mathbf{u}_{h,2}\|_{L^3} \, \|T_{h,2}\|_{dG,T} \leq \frac{\alpha_T \, \nu_m }{C_{\text{s},3} \, C_{\text{s},6} \, k_M} \sqrt{\frac{k_m}{C_{\text{s},6} \, L_{\nu}}},
		\end{aligned}
	\end{equation}
	we have that $\|\varepsilon_T\|_{dG,T}^2 = 0$, then $\varepsilon_T = 0$.
	\medskip
	
	\noindent $(ii)$ As $\varepsilon_T = 0$, from \eqref{eq:uniq4} we can conclude that $\varepsilon_{\mathbf{u}} = \mathbf{0}$.
	\medskip
	
	\noindent $(iii)$ In the last step, we consider $\varepsilon_p$. To this aim, we take $(\mathbf{v}_h, q_h, S_h) = (\mathbf{v}_h, 0, 0)$ as test functions for \eqref{eq:uniq1} and we obtain:
	\begin{equation}
		\label{eq:uniq10}
		\begin{aligned}
			-\mathcal{B}_h(\varepsilon_p, \mathbf{v}_h) = & \, \mathcal{M}_{\nu}(T_{h,1}, \mathbf{u}_{h,1}, \mathbf{v}_h) - \mathcal{M}_{\nu}(T_{h,2}, \mathbf{u}_{h,2}, \mathbf{v}_h) + \mathcal{M}_{\beta}(\mathbf{u}_{h,1}, \mathbf{u}_{h,1}, \mathbf{v}_h) \\
			& \, - \mathcal{M}_{\beta}(\mathbf{u}_{h,2}, \mathbf{u}_{h,2}, \mathbf{v}_h) + \mathcal{D}_u(\varepsilon_{\mathbf{u}}, \mathbf{v}_h) \\
			= & \, (\nu(T_{h,1}) \mathbf{K}^{-1} \, \varepsilon_{\mathbf{u}}, \mathbf{v}_h) + (\beta(|\mathbf{u}_{h,1}|\mathbf{u}_{h,1} - |\mathbf{u}_{h,2}|\mathbf{u}_{h,2}),\mathbf{v}_h) \\
			& \, + ((\nu(T_{h,1}) - \nu(T_{h,2})) \mathbf{K}^{-1} \, \mathbf{u}_{h,2}, \mathbf{v}_h) + \mathcal{D}_u(\varepsilon_{\mathbf{u}}, \varepsilon_{\mathbf{u}}). 
		\end{aligned}
	\end{equation}
	As $\varepsilon_{\mathbf{u}} = \mathbf{0}$ and $\varepsilon_T = 0$, we can observe that $\mathcal{B}_h(\varepsilon_p, \mathbf{v}_h) = 0$. Then, putting \eqref{eq:uniq10} into \eqref{eq:infsup} we obtain:
	\begin{equation}
		\label{eq:uniq11}
		\frac{\mathbb{B}}{2} \|\varepsilon_p\|^2 \lesssim \mathcal{D}_p(\varepsilon_p, \varepsilon_p)
	\end{equation}
	Last, we substitute \eqref{eq:uniq11} into \eqref{eq:uniq3} and this ends the proof.
\end{proof}

\begin{remark}
	We observe that the condition on the discrete solutions \eqref{eq:uniq9} can be rewritten as a condition on the problem's data. To do so, we need to use Young inequality in \eqref{eq:uniq9} and bound the norms of the discrete solutions via the stability estimate \eqref{eq:stab_est}. 
\end{remark}

\subsection{Convergence of the fixed-point algorithm}
\label{sec:convergence_fixedpoint}
The aim of this section is to prove the convergence of the linearization algorithm, cf. Section~\ref{sec:linearization}. To this aim, we show that the difference of approximate solutions at two successive iterations defines a contracting sequence. Let $(\mathbf{u}_h^{k+1}, p_h^{k+1}, T_h^{k+1})$ and $(\mathbf{u}_h^k, p_h^k, T_h^k)$ be the solutions to \eqref{eq:linearized_discrete_form} at the $(k+1)^{\text{th}}$ and $k^{\text{th}}$ iterations, respectively. For all $k \geq 1$, we define:
$$
\boldsymbol{\delta}_{\mathbf{u}}^k = \mathbf{u}^{k+1} - \mathbf{u}^{k}, \qquad \delta_{p}^k = p^{k+1} - p^{k}, \qquad \delta_{T}^k = T^{k+1} - T^{k}. 
$$
Then, it can be observed that $(\boldsymbol{\delta}_{\mathbf{u}}^k, \delta_p^k, \delta_T^k)$ solves the problem:
\begin{equation}
	\label{eq:error_equation_1}
	\begin{aligned}
		& \mathcal{M}_{\nu}(T_h^{k}, \mathbf{u}_h^{k+1}, \mathbf{v}_h) - \mathcal{M}_{\nu}(T_h^{k-1}, \mathbf{u}_h^{k}, \mathbf{v}_h) + \mathcal{M}_{\beta}(\mathbf{u}_h^{k}, \mathbf{u}_h^{k+1}, \mathbf{v}_h) \\
		& - \mathcal{M}_{\beta}(\mathbf{u}_h^{k-1}, \mathbf{u}_h^k, \mathbf{v}_h) + \mathcal{B}_h(\delta_p^k, \mathbf{v}_h) - \mathcal{B}_h(q_h, \boldsymbol{\delta}_{\mathbf{u}}^k) + \mathcal{A}_h(\delta_T^k,S_h) \\
		& + \mathcal{C}_h(\mathbf{u}_h^{k}, T_h^{k+1}, S_h) -  \mathcal{C}_h(\mathbf{u}_h^{k-1}, T_h^k, S_h) + \mathcal{D}_u(\boldsymbol{\delta}_{\mathbf{u}}^k, \mathbf{v}) + \mathcal{D}_p(\delta_p^k, q) = 0.
	\end{aligned}
\end{equation}
In the following theorem, we state the conditions under which the fixed-point iterative method converges. We start by observing that, as we are considering the linearized problem instead of the nonlinear one, we want to control the error for the velocity field in the following norm:
$$
\begin{aligned}
	& \| \mathbf{v} \|_{dG, \text{div}}^2 = \| \mathbf{v} \|^2 + \| \divh{\mathbf{v}} \|^2 + \sum_{F\in\mathcal{F}_I} \|\xi^{1/2} \jump{\mathbf{v}}_n \|_F^2 \quad &&  \forall \ \mathbf{v} \in \mathbf{V}_h^{\ell},
\end{aligned}
$$
while the norms for the temperature and pressure fields are the ones defined in \eqref{eq:energy_norms}. Moreover, we introduce the auxiliary $\| S \|_{dG, 3}$-norm of the functions belonging to $V_h^{\ell}$:
$$
\begin{aligned}
	& \| S \|_{dG, 3}^3 = 
	\| \nabla_h S \|_{\mathbf{L}^{3}(\Omega)}^3 
	+ \sum_{F \in \mathcal{F}} h_{\kappa}^{-2} \, \| \jump{T_h^k} \|_{L^3(F)}^3 
	\qquad \forall \ S \in V_h^{\ell}.
\end{aligned}
$$
\begin{theorem}
	\label{thm:conv_fp}
	Let the assumptions of Theorem~\ref{thm:stab_est} be satisfied. Additionally, assume
	\begin{equation}
		\label{eq:ass_conv_fp}
		\|\mathbf{u}_h^k\|_{\mathbf{L}^{\infty}} \lesssim \min\left( \frac{\nu_m}{k_M \beta}, \, \sqrt{\frac{\alpha_T \, \nu_m \, k_m^2}{k_M \, L_{\nu}^2}}\right) \quad \text{and} \quad \| T_h^k\|_{dG,3} \lesssim \sqrt{\frac{\nu_m \, \theta_M \, \alpha_T}{k_M}}
	\end{equation}
	where $\alpha_T$ has been defined in Lemma~\ref{lem:boundcoerc_bil_forms}. Then, the linearization strategy defined in Section~\ref{sec:linearization} converges, namely
	$
	\mathbf{V}_h^{\ell} \times W_h^m \times V_h^{\ell} \ni (\boldsymbol{\delta}_{\mathbf{u}}^k, \delta_p^k, \delta_T^k) \rightarrow \mathbf{0} \,\,\, \text{as} \,\,\, k \rightarrow \infty.
	$
\end{theorem}
\begin{remark}
	Owing to the definition of the linearization scheme in \eqref{eq:linearized_discrete_form}, the limit towards which the iterations converge is a solution of the nonlinear problem \eqref{eq:nonlinear_discrete_form_DG}. Under the assumptions of Theorem \ref{thm:uniqueness}, this corresponds to the unique discrete solution.
\end{remark}
\begin{proof}
	We start by adding and subtracting $\mathcal{M}_{\nu}(T_h^{k}, \mathbf{u}_h^{k}, \mathbf{v}_h)$, $\mathcal{M}_{\beta}(\mathbf{u}_h^{k}, \mathbf{u}_h^{k}, \mathbf{v}_h)$, and $\mathcal{C}_h(\mathbf{u}_h^{k}, T_h^{k}, S_h)$ to  \eqref{eq:error_equation_2}. Then, expanding the trilinear forms we obtain:
	\begin{equation}
		\label{eq:error_equation_2}
		\begin{aligned}
			& \left((\nu(T_h^k)\mathbf{K}^{-1} + \beta \lvert \mathbf{u}_h^k \rvert) \, \boldsymbol{\delta}_{\mathbf{u}}^{k}, \mathbf{v}_h\right) + \mathcal{B}_h(\delta_p^k, \mathbf{v}_h) - \mathcal{B}_h(q_h, \boldsymbol{\delta}_{\mathbf{u}}^k) + \mathcal{A}_h(\delta_T^k,S_h) \\
			& + \mathcal{C}_h(\mathbf{u}_h^{k}, \delta_T^{k}, S_h) + \mathcal{D}_u(\boldsymbol{\delta}_{\mathbf{u}}^k, \mathbf{v}) + \mathcal{D}_p(\delta_p^k, q) = - \mathcal{C}_h(\boldsymbol{\delta}_{\mathbf{u}}^{k-1}, T_h^k, S_h) \\
			& + \left((\nu(T_h^{k-1})\mathbf{K}^{-1} - \nu(T_h^k)\mathbf{K}^{-1} + \beta \lvert \mathbf{u}_h^{k-1} \rvert - \beta \lvert \mathbf{u}_h^k \rvert) \, \mathbf{u}_h^{k}, \mathbf{v}_h \right) .
		\end{aligned}
	\end{equation}
	In the second step of the proof, we focus on the fluid flow and incompressibility equations. We consider $(\mathbf{v}_h, q_h, S_h) = (\boldsymbol{\delta}_{\mathbf{u}}^k, \delta_p^k, 0)$ as test functions in \eqref{eq:error_equation_2}:
	$$
	\begin{aligned}
		& \left((\nu(T_h^k)\mathbf{K}^{-1} + \beta \lvert \mathbf{u}_h^k \rvert) \, \boldsymbol{\delta}_{\mathbf{u}}^{k}, \boldsymbol{\delta}_{\mathbf{u}}^{k}\right) + \mathcal{D}_u(\boldsymbol{\delta}_{\mathbf{u}}^k, \boldsymbol{\delta}_{\mathbf{u}}^k) + \mathcal{D}_p(\delta_p^k, \delta_p^k)  \\
		& = \left((\nu(T_h^{k-1})\mathbf{K}^{-1} - \nu(T_h^k)\mathbf{K}^{-1} + \beta \lvert \mathbf{u}_h^{k-1} \rvert - \beta \lvert \mathbf{u}_h^k \rvert) \, \mathbf{u}_h^{k}, \boldsymbol{\delta}_{\mathbf{u}}^{k} \right).
	\end{aligned}
	$$
	We start by noting that $\left((\nu(T_h^k)\mathbf{K}^{-1} + \beta \lvert \mathbf{u}_h^k \rvert) \, \boldsymbol{\delta}_{\mathbf{u}}^{k}, \boldsymbol{\delta}_{\mathbf{u}}^{k}\right) \geq \frac{\nu_m}{k_M} \| \mathbf{e}_u^k \|^2$ by Assumption~\ref{assumption:model_problem}. Next, we test problem \eqref{eq:error_equation_2} with $(\mathbf{v}_h, q_h, S_h) = (\mathbf{0}, -\divh{\boldsymbol{\delta}_{\mathbf{u}}^k}, 0)$ and we obtain
	$$
	\begin{aligned}
		& \mathcal{B}_h(-\divh{\boldsymbol{\delta}_{\mathbf{u}}^k}, \boldsymbol{\delta}_{\mathbf{u}}^k) + \mathcal{D}_p(\delta_p^k, -\divh{\boldsymbol{\delta}_{\mathbf{u}}^k}) = 0.
	\end{aligned}
	$$
	By proceeding as in Section~\ref{sec:StabilityAnalysis} (cf. Equation~\eqref{eq:basic_est_step4}) and using Cauchy-Schwarz and Young inequalities we get
	\begin{equation}
		\label{eq:error_equation_4_old}
		\begin{aligned}
			\frac{\nu_m}{2k_M} \|\boldsymbol{\delta}_{\mathbf{u}}^{k}\|^2 & + \|\divh{\boldsymbol{\delta}_{\mathbf{u}}^{k}}\|^2 + \sum_{F \in \mathcal{F}_I} \| \sqrt{\xi} \jump{\boldsymbol{\delta}_{\mathbf{u}}^k}_n \|_{F}^2 + \mathcal{D}_p(\delta_p^k, \delta_p^k) \\
			& \lesssim \frac{2 k_M}{\nu_m}\left\|(\nu(T_h^{k-1})\mathbf{K}^{-1} - \nu(T_h^k)\mathbf{K}^{-1}  + \beta \lvert \mathbf{u}_h^{k-1} \rvert - \beta \lvert \mathbf{u}_h^k \rvert) \, \mathbf{u}_h^{k} \right\|^2.
		\end{aligned}
	\end{equation}
	We are left to control the right hand side of \eqref{eq:error_equation_4_old}. To this aim, we use the Lipschitz-continuity of the viscosity coefficient $\nu$ (we recall that $L_{\nu}$ is its Lipschitz constant) and the triangle inequality to get:
	$$
	\begin{aligned}
		\frac{\nu_m}{2k_M} \|\boldsymbol{\delta}_{\mathbf{u}}^{k}\|^2 + \|\divh{\boldsymbol{\delta}_{\mathbf{u}}^{k}}\|^2 + & \sum_{F \in \mathcal{F}_I} \| \sqrt{\xi} \jump{\boldsymbol{\delta}_{\mathbf{u}}^k}_n \|_{F}^2 + \mathcal{D}_p(\delta_p^k, \delta_p^k) \\
		& \lesssim \frac{4 k_M}{\nu_m} \| \mathbf{u}_h^k\|_{\mathbf{L}^{\infty}}^2 \left( \frac{L_{\nu}^2}{k_m^2} \| \delta_T^{k-1}\|_{dG,T}^2+ \beta^2 \| \boldsymbol{\delta}_{\mathbf{u}}^{k-1}\|^2 \right)
	\end{aligned}
	$$
	Now, in the third step of the proof, we focus on the temperature equation. We consider $(\mathbf{v}_h, q_h, S_h) = (\mathbf{0}, 0, \delta_T^k)$ in \eqref{eq:error_equation_2} and we obtain:
	$$
	\begin{aligned}
		\mathcal{A}_h(\delta_T^k, \delta_T^k) + \mathcal{C}_h(\mathbf{u}_h^{k}, \delta_T^{k}, \delta_T^k) = - \mathcal{C}_h(\boldsymbol{\delta}_{\mathbf{u}}^{k-1}, T_h^k, \delta_T^k),
	\end{aligned}
	$$
	where, due to Lemma~\ref{lem:Ch_temam_skewsymm}, we have that $\mathcal{C}_h(\mathbf{u}_h^{k}, \delta_T^{k}, \delta_T^k) \geq 0$. Then, we are left with:
	$$
	\begin{aligned}
		\alpha_T \| \delta_T^k \|_{dG,T}^2 \leq - 
		\mathcal{C}_h(\boldsymbol{\delta}_{\mathbf{u}}^{k-1}, T_h^{k}, \delta_T^k) = \mathcal{I}_1 + \mathcal{I}_2 + \mathcal{I}_3.
	\end{aligned}
	$$
	To bound the term $\mathcal{I}_1$ we use the H\"older, Young, and Sobolev inequalities to get:
	$$
	\begin{aligned}
		\mathcal{I}_1 = & \, - ( \boldsymbol{\delta}_{\mathbf{u}}^{k-1} \cdot \nabla 
		T_h^k, \delta_T^k) - \frac{1}{2}(\divh{\boldsymbol{\delta}_{\mathbf{u}}^{k-1}} \,  
		T_h^k, \delta_T^k),\\
		\leq & \, \| \nabla_h T_h^k\|_{\mathbf{L}^3} \| \boldsymbol{\delta}_{\mathbf{u}}^{k-1}\| \| \delta_T^k \|_{L^6} + \frac{1}{2}\| T_h^k \|_{L^3} \| \divh{\boldsymbol{\delta}_{\mathbf{u}}^{k-1}}\| \| \delta_T^k \|_{L^6} \\
		\leq & \, \frac{\epsilon}{2} \| \nabla_h T_h^k\|_{\mathbf{L}^3}^2 \| \boldsymbol{\delta}_{\mathbf{u}}^{k-1}\|^2 + \frac{1}{2 \epsilon} \|\delta_T^k\|_{L^6}^2 + \frac{\epsilon}{2} \| T_h^k\|_{L^3}^2 \| \divh{\boldsymbol{\delta}_{\mathbf{u}}^{k-1}}\|^2 + \frac{1}{4 \epsilon} \|\delta_T^k\|_{L^6}^2 \\
		\leq & \, \epsilon\left(\frac12+\frac{C_p^2}2\right) \| T_h^k\|_{dG,3}^2 \| \boldsymbol{\delta}_{\mathbf{u}}^{k-1}\|_{dG,\text{div}}^2 + \frac{3 \, C_{s,6}}{4 \epsilon \, \theta_M} \|\delta_T^k\|_{dG,T}^2.
	\end{aligned}
	$$
	For the next step $\mathcal{I}_2$, we use Cauchy-Schwarz, triangle, H\"older, Young, Sobolev, and trace-inverse inequalities:
	$$
	\begin{aligned}
		\mathcal{I}_2 = & \, \sum_{F \in \mathcal{F}_I} \int_{F} \left( \avg{\boldsymbol{\delta}_{\mathbf{u}}^{k-1}} \cdot \jump{T_h^k} \right) \avg{\delta_T^k} - \frac{1}{2}\sum_{F \in \mathcal{F}} \int_F \ \left| \avg{\boldsymbol{\delta}_{\mathbf{u}}^{k-1}} \cdot \mathbf{n} \right| \jump{T_h^k} \mkern-2.5mu \cdot \mkern-2.5mu \jump{\delta_T^k} \\
		& \, + \frac{1}{2}\sum_{F \in \mathcal{F}_B} \int_F \  (\boldsymbol{\delta}_{\mathbf{u}}^{k-1} \cdot \mathbf{n}) \, T_h^k \, \delta_T^k, \\
		\leq & \, \sum_{F \in \mathcal{F}} \int_F |\avg{\boldsymbol{\delta}_{\mathbf{u}}^{k-1}}| \, | \jump{T_h^k} | \frac{2 | \avg{\delta_T^k} | + | \jump{\delta_T^k} |}{2} \\
		\leq & \, \sum_{F \in \mathcal{F}} h_{\kappa}^{1/2}\,\| \boldsymbol{\delta}_{\mathbf{u}}^{k-1} \|_F \, h_{\kappa}^{-2/3}\,\| \jump{T_h^k} \|_{L^3(F)} \, h_{\kappa}^{1/6}\,\| \delta_T^k\|_{L^6(F)} \\
		%
		%
		\leq & \, \left( \sum_{\kappa \in \mathcal{T}_h} h_{\kappa} \,\| \boldsymbol{\delta}_{\mathbf{u}}^{k-1} \|_{\partial \kappa}^2 \right)^{1/2} \hspace{-1.5pt} \left( \sum_{F \in \mathcal{F}} h_{\kappa}^{-2} \, \| \jump{T_h^k} \|_{L^3(F)}^3 \right)^{1/3} \hspace{-1.5pt} \left( \sum_{F \in \mathcal{F}} h_{\kappa} \,\| \delta_T^k\|_{L^6(\partial \kappa)}^6 \right)^{1/6} \\
		%
		%
		\leq & \, \epsilon C_{\text{tr}}^4 \| T_h^k \|_{dG,3}^2 \| \boldsymbol{\delta}_{\mathbf{u}}^{k-1} \|_{dG, \text{div}}^2 + \frac{C_{s,6}}{4 \epsilon \, \theta_M} \| \delta_T^k \|_{dG,T}^2.
	\end{aligned}
	$$
	Last, we bound the term $\mathcal{I}_3$ by use of Cauchy-Schwarz, H\"older, Young, Poincarè, Sobolev, and trace-inverse inequalities:
	$$
	\begin{aligned}
		\mathcal{I}_3 = & \, \frac{1}{2}\sum_{F \in \mathcal{F}_I} \int_F \jump{\mathbf{e}_\mathbf{u}^{k-1}}_n \, \avg{T_h^k \, \delta_T^k}\\
		\leq & \, \sum_{F \in \mathcal{F}} h_{\kappa}^{-1/2}\,\| \jump{\boldsymbol{\delta}_{\mathbf{u}}^{k-1}}_n \|_F \, h_{\kappa}^{1/3}\,\| T_h^k \|_{L^3(F)} \, h_{\kappa}^{1/6}\,\| \delta_T^k\|_{L^6(F)} \\
		\leq & \, \left( \sum_{F \in \mathcal{F}} h_{\kappa}^{-1}
		\,\| \jump{\boldsymbol{\delta}_{\mathbf{u}}^{k-1}}_n \|_F^2 \right)^{1/2} \hspace{-3.5pt} \left( \sum_{F \in \mathcal{F}} h_{\kappa} \, \| T_h^k \|_{L^3(\partial \kappa)}^3 \right)^{1/3} \hspace{-3.5pt} \left( \sum_{F \in \mathcal{F}} h_{\kappa} \,\| \delta_T^k\|_{L^6(\partial \kappa)}^6 \right)^{1/6} \\
		\leq & \, \epsilon C_{\text{tr}}^4 C_p^2 \| T_h^k \|_{dG,3}^2 \| \boldsymbol{\delta}_{\mathbf{u}}^{k-1} \|_{dG, \text{div}}^2 + \frac{C_{s,6}}{4\epsilon \, \theta_M}\| \delta_T^k\|_{dG,T}^2     
	\end{aligned}
	$$
	By grouping all the results together, we obtain:
	$$
	\begin{aligned}
		& \frac{\nu_m}{2 k_M} \|\boldsymbol{\delta}_{\mathbf{u}}^{k}\|^2 + \|\divh{\boldsymbol{\delta}_{\mathbf{u}}^{k}}\|^2 + \sum_{F \in \mathcal{F}_I} \| \sqrt{\xi} \jump{\boldsymbol{\delta}_{\mathbf{u}}^k}_n \|_{F}^2 + \mathcal{D}_p(\delta_p^k, \delta_p^k) + \alpha_T \| \delta_T^k \|_{dG,T}^2 \lesssim \\
		& \frac{4 k_M L_{\nu}^2}{\nu_m k_m^2} \| \mathbf{u}_h^k\|_{\mathbf{L}^{\infty}}^2 \| \delta_T^{k-1}\|_{dG,T}^2 + \frac{4 k_M \beta^2}{\nu_m} \| \mathbf{u}_h^k\|_{\mathbf{L}^{\infty}}^2 \| \boldsymbol{\delta}_{\mathbf{u}}^{k-1}\|^2 \\
		& + \epsilon \left( 1+ C_p^2 + C_p^2 C_{\text{tr}}^4 + C_{\text{tr}}^4 \right) \| T_h^k \|_{dG,3}^2\| \boldsymbol{\delta}_{\mathbf{u}}^{k-1} \|_{dG,\text{div}}^2 + \frac{5 \, C_{s,6}}{4 \epsilon \, \theta_M} \| \delta_T^k \|_{dG,T}^2
	\end{aligned}
	$$
	we set $\epsilon = 5 \, C_{s,6}/(2 \theta_M \alpha_T)$ (cf. Lemma~\ref{lem:boundcoerc_bil_forms}) and we obtain: 
	$$
	\begin{aligned}
		\frac{\nu_m}{k_M} \|\boldsymbol{\delta}_{\mathbf{u}}^{k}\|_{dG,\text{div}}^2 & + \alpha_T \| \delta_T^k \|_{dG,T}^2 \lesssim \frac{k_M L_{\nu}^2}{\nu_m k_m^2} \| \mathbf{u}_h^k\|_{\mathbf{L}^{\infty}}^2 \| \delta_T^{k-1}\|_{dG,T}^2 \\
		& + \max \left( \frac{k_M \beta^2}{\nu_m} \| \mathbf{u}_h^k\|_{\mathbf{L}^{\infty}}^2 , \, \, \frac{1}{\theta_M \, \alpha_T} \| T_h^k \|_{dG,3}^2 \right) \| \boldsymbol{\delta}_{\mathbf{u}}^{k-1} \|_{dG,\text{div}}^2.
	\end{aligned}
	$$
	Given \eqref{eq:ass_conv_fp}, we infer that the map $(\boldsymbol{\delta}_{\mathbf{u}}^{k-1}, \delta_p^{k-1}, \delta_T^{k-1}) \rightarrow (\boldsymbol{\delta}_{\mathbf{u}}^k, \delta_p^k, \delta_T^k)$ is a contraction. Then, by applying the Banach fixed-point theorem the proof is concluded.
\end{proof}

\section{Numerical results}
\label{sec:NumericalResults}
This section assesses the performance of the proposed scheme in terms of accuracy and demonstrates its application to physically relevant test cases. All computations are performed in \texttt{FEniCS} \cite{Alnaes2015}. Convergence tests are carried out for both the dG–dG–dG and RT–dG–dG schemes, while physically relevant test cases use only the RT–dG–dG scheme. The two- and three-dimensional meshes consist of triangles and tetrahedra, respectively. The penalty coefficients $\alpha_1$, $\alpha_2$, and $\alpha_3$ in \eqref{eq:stabilization_functions} are set equal to $10$.
We denote by $\ell$ the polynomial degree for the temperature field and by $m$ the degree for the pressure field. For the velocity field, $\ell$ is used in the dG–dG–dG scheme and $m$ in the RT–dG–dG scheme, with $m = \ell - 1$ in all tests. With this choice, we expect the same accuracy in the $L^2$-norm for the pressure and in the energy norms for the velocity and temperature fields. In addition, the dG–dG–dG scheme gains one order of accuracy in the $L^2$-norm for the velocity field. For the numerical investigation, we consider the fluid viscosity to be given by a particular model function $\nu(S) = 1 + e^{-S}$ \cite{Deugoue2022}.
\subsection{Convergence test case in two-dimensions}
\label{sec:conv_test2D}
We set $\Omega = (0,1)^2$ with the following manufactured analytical solution:
\begin{equation}
	\begin{aligned}
		\mathbf{u}(x,y) & \ = \left( \begin{aligned}
			& x^2 \sin(2 \pi y) \\
			& \frac{x}{\pi} \cos(2 \pi y)
		\end{aligned} \right), \
		p(x,y) = (x^2 + 3y - 2xy) \sin(2 \pi x), \\
		T(x,y) & \ = (-y^2 + 2x) \cos(2 \pi x);
	\end{aligned}
\end{equation}
the boundary conditions and forcing terms are set accordingly. The model coefficients are reported in Table~\ref{tab:params_convtest}. 
\begin{table}[ht]
	\centering 
	\begin{tabular}{l | c  c  l | c  c  l | c }
		$\mathbf{K} \ [\si[per-mode = symbol]{\metre\squared}]$ & $\mathbf{I}$ & & 
		$\beta \ [\si[per-mode = symbol]{\pascal \second\squared \per \metre\cubed}]$&  1 & & 
		$\boldsymbol{\Theta} \ [\si[per-mode = symbol]{\metre\squared \per \second}]$ & $\mathbf{I}$
	\end{tabular}
	\caption{Convergence tests of Section~\ref{sec:conv_test2D} and Section~\ref{sec:conv_test3D}: model parameters}
	\label{tab:params_convtest}
\end{table}
We test the convergence of the dG-dG-dG scheme with respect to both the mesh size $h$ and the polynomial degrees $\ell$, $m$. The convergence of the RT-dG-dG scheme is tested only with respect to the mesh size $h$. For the $h$-convergence, we consider a sequence of successively fined triangular meshes and we set $\ell = 2$ and $m = 1$. For the $p$-convergence, we fix a mesh made of $N = 64$ triangular elements and we vary the polynomial degrees $\ell = [2,3,4,5]$, $m = [1,2,3,4]$. For what concerns the fixed-point iterative algorithm, we set a tolerance of $10^{-8}$ on the relative difference between two successive iterations. 
In Figure~\ref{fig:ConvH_2D} we show the computed errors versus the mesh-size $h$ (loglog scale). 
\begin{center}
\begin{figure}[!h]
\begin{subfigure}[b]{1\textwidth}
\begin{subfigure}[b]{0.4\textwidth}
\begin{tikzpicture}
\begin{axis}[%
width=3.2cm,
height=3cm,
at={(0\textwidth,0\textwidth)},
scale only axis,
xmode=log,
xmin=6,
xmax=180,
xminorticks=true,
xlabel={\footnotesize $1/h$},
ymode=log,
ymin=1e-7,
ymax=2e-2,
yminorticks=true,
ylabel={\footnotesize $L^2$-errors},
legend style={draw=none,fill=none,legend cell align=left},
legend pos=outer north east,
]
\addplot [color=myred,solid,line width=1pt, mark size=2.5pt,mark=o, mark options={color=myred}]
  table[row sep=crcr]{
8     2.83374773e-04   \\
16    3.52389721e-05    \\
32    4.40420891e-06   \\
64    5.53120312e-07   \\
};

\addplot [color=myred,solid,line width=1pt, mark size=2.5pt,mark=x,mark options={color=myred}]
  table[row sep=crcr]{
8     5.26930080e-03   \\
16    1.32713299e-03    \\
32    3.32524288e-04   \\
64    8.31911792e-05   \\
128   2.08032872e-05   \\
};

\addplot [color=myblue,solid,line width=1pt, mark size=2.5pt,mark=o,mark options={color=myblue}]
  table[row sep=crcr]{
8    1.071916e-02 \\
16   2.6813e-03 \\
32   6.7046e-04 \\
64   1.6762e-04 \\
};

\addplot [color=myblue,solid,line width=1pt, mark size=2.5pt,mark=x,mark options={color=myblue}]
  table[row sep=crcr]{
8    1.07196639e-02 \\
16   2.68132912e-03 \\
32   6.70457448e-04 \\
64   1.67622721e-04 \\
128  4.19062111e-05 \\
};

\addplot [color=mygreen,solid,line width=1pt, mark size=2.5pt,mark=o,mark options={color=mygreen}]
  table[row sep=crcr]{
8    9.89424389e-04  \\
16   1.26601063e-04 \\
32   1.59889619e-05 \\
64   2.00789422e-06 \\
};

\addplot [color=mygreen,solid,line width=1pt, mark size=2.5pt,mark=x,mark options={color=mygreen}]
  table[row sep=crcr]{
8    9.89510914e-04  \\
16   1.26603599e-04 \\
32   1.59890381e-05 \\
64   2.00789572e-06 \\
128  2.51534307e-07  \\
};

\addplot [color=black,solid,line width=0.5pt]
  table[row sep=crcr]{
 32     32e-6 \\
 64     32e-6 \\
 64     4e-6 \\
 32     32e-6 \\ 
};
\node[right, align=left, text=black, font=\footnotesize]
at (axis cs:64,15e-6) {3}; 

\addplot [color=black,solid,line width=0.5pt]
table[row sep=crcr]{
 32     11.2e-4 \\
 64     11.2e-4 \\
 64     2.8e-4 \\
 32     11.2e-4 \\  
};
\node[right, align=left, text=black, font=\footnotesize]
at (axis cs:64,5e-4) {2}; 

\end{axis}
\end{tikzpicture}
\end{subfigure}
\begin{subfigure}[b]{0.59\textwidth}
\begin{tikzpicture}
\begin{axis}[%
width=3.2cm,
height=3cm,
at={(0\textwidth,0\textwidth)},
scale only axis,
xmode=log,
xmin=6,
xmax=180,
xminorticks=true,
xlabel={\footnotesize $1/h$},
ymode=log,
ymin=1e-05,
ymax=0.2,
yminorticks=true,
ylabel={\footnotesize Energy-errors},
legend style={draw=none,fill=none,legend cell align=left},
legend pos=outer north east,
]
\addplot [color=myred,solid,line width=1pt, mark size=2.5pt,mark=o, mark options={color=myred}]
  table[row sep=crcr]{
8     1.17363936e-03 \\
16    2.97667952e-04    \\
32    7.47844545e-05   \\
64    1.87314773e-05   \\
};
\addlegendentry{\scriptsize $\mathbf{u}_h$(dG-dG-dG)}

\addplot [color=myred,solid,line width=1pt, mark size=2.5pt,mark=x, mark options={color=myred}]
  table[row sep=crcr]{
8     5.26930080e-03   \\
16    1.32713299e-03    \\
32    3.32524288e-04   \\
64    8.31911792e-05   \\
128   2.08032872e-05   \\
};
\addlegendentry{\scriptsize $\mathbf{u}_h$(RT-dG-dG)}

\addplot [color=myblue,solid,line width=1pt, mark size=2.5pt,mark=o,mark options={color=myblue}]
  table[row sep=crcr]{
10000    100 \\
10001    100 \\
};
\addlegendentry{\scriptsize $p_h$(dG-dG-dG)}

\addplot [color=myblue,solid,line width=1pt, mark size=2.5pt,mark=x,mark options={color=myblue}]
  table[row sep=crcr]{
10000    100 \\
10001    100 \\
};
\addlegendentry{\scriptsize $p_h$(RT-dG-dG)}

\addplot [color=mygreen,solid,line width=1pt, mark size=2.5pt, mark=o, mark options={color=mygreen}]
  table[row sep=crcr]{
8     0.09344198 \\
16    0.02337781  \\
32    0.00583484 \\
64    0.00145674 \\
};
\addlegendentry{\scriptsize $T_h$(dG-dG-dG)}

\addplot [color=mygreen,solid,line width=1pt, mark size=2.5pt, mark=x, mark options={color=mygreen}]
  table[row sep=crcr]{
8     0.09344219 \\
16    0.02337784  \\
32    0.00583484 \\
64    0.00145674 \\
128   0.00036389 \\
};
\addlegendentry{\scriptsize $T_h$(RT-dG-dG)}

\addplot [color=black,solid,line width=0.5pt]
  table[row sep=crcr]{
 32   0.0112 \\
 64   0.0112 \\
 64   0.0028 \\
 32   0.0112 \\ 
};
\node[right, align=left, text=black, font=\footnotesize]
at (axis cs:64, 0.006) {2}; 

\end{axis}
\end{tikzpicture}
\end{subfigure}
\end{subfigure}

\caption{Test case of Section~\ref{sec:conv_test2D}. Computed errors in $L^2$-norm (left) and energy-norms (right) versus $1/h$ (\texttt{log-log} scale).}
\label{fig:ConvH_2D}

\end{figure}
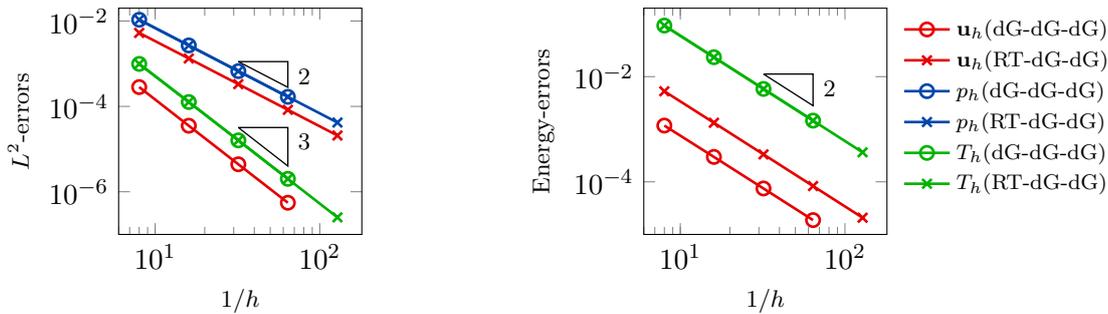
\end{center}
We observe that the $L^2$-error for the pressure field and the energy-errors for the velocity and temperature fields decrease as $h^2$ for both the schemes. Moreover, as in both the methods the temperature equation is discretized with the use of dG method, we see that the $L^2$-error for the temperature field decays as $h^3$. The only difference we can clearly observe between the two schemes is in the $L^2$- error of the velocity field; indeed, in the dG-dG-dG scheme it decreases as $h^3$, while in the RT-dG-dG scheme it decreases as $h^2$. Thus, we gain one order of accuracy when using a full-dG discretization. We remark that by choosing $\ell = 2, \, m = 1$, we observe the same accuracy  when looking at the energy-norm of the velocity error, at the $L^2$-norm of the pressure error, and at the $dG$-norm of the temperature error. On average, $19$ iterations of the fixed-point algorithm are required for achieving convergence of the dG-dG-dG algorithm and $14$ for the RT-dG-dG one.
It is worth noting that, apart from the velocity field errors all computed errors (in absolute value) are very similar.

For the dG-dG-dG scheme we also test the convergence with respect to the polynomial approximation degree. In Figure~\ref{fig:ConvP_2D} we report the computed error versus the polynomial degree (semilog scale). We observe that the $L^2$-errors and energy-errors for all the three fields decay exponentially  with respect to $\ell$ (and $m$, consequently).

\begin{center}
\begin{figure}[!h]
\begin{subfigure}[b]{1\textwidth}
\begin{subfigure}[b]{0.4\textwidth}
\begin{tikzpicture}
\begin{axis}[%
width=3.2cm,
height=3cm,
at={(0\textwidth,0\textwidth)},
scale only axis,
xmin=0.01,
xmax=6,
xminorticks=true,
xtick = {2,5},
xticklabels = {\shortstack{\tiny{$\ell=2$} \\ \tiny{$m=1$}}, \shortstack{\tiny{$\ell=5$} \\ \tiny{$m=4$}}},
xlabel={\footnotesize polynomial degrees},
ymode=log,
ymin=2e-5,
ymax=3e-1,
yminorticks=true,
ylabel={\footnotesize $L^2$-errors},
legend style={draw=none,fill=none,legend cell align=left},
legend pos=outer north east,
]
\addplot [color=myred,solid,line width=1pt, mark size=2.5pt,mark=o, mark options={color=myred}]
  table[row sep=crcr]{
2    2.13285687e-02   \\
3    2.25654572e-03   \\
4    1.94661335e-04   \\
5    2.90406775e-05   \\
};

\addplot [color=myblue,solid,line width=1pt, mark size=2.5pt,mark=o, mark options={color=myblue}]
  table[row sep=crcr]{
2    2.68597211e-01   \\
3    4.29855143e-02    \\
4    4.02141016e-03   \\
5    4.01701014e-04   \\
};

\addplot [color=mygreen,solid,line width=1pt, mark size=2.5pt,mark=o, mark options={color=mygreen}]
  table[row sep=crcr]{
2    4.81894949e-02   \\
3    7.70509916e-03    \\
4    5.99676127e-04   \\
5    6.73425294e-05   \\
};

\addplot [color=black,dashed,line width=1.0pt]
  table[row sep=crcr]{
2   0.20214 \\
3   0.016593 \\
4   0.001362 \\
5   0.0001118 \\
};

\end{axis}
\end{tikzpicture}
\end{subfigure}
\begin{subfigure}[b]{0.59\textwidth}
\begin{tikzpicture}
\begin{axis}[%
width=3.2cm,
height=3cm,
at={(0\textwidth,0\textwidth)},
scale only axis,
xmin=0.01,
xmax=6,
xminorticks=true,
xtick = {2,5},
xticklabels = {\shortstack{\tiny{$\ell=2$} \\ \tiny{$m=1$}}, \shortstack{\tiny{$\ell=5$} \\ \tiny{$m=4$}}},
xlabel={\footnotesize polynomial degrees},
ymode=log,
ymin=1e-5,
ymax=1.8,
yminorticks=true,
ylabel={\footnotesize Energy-errors},
legend style={draw=none,fill=none,legend cell align=left},
legend pos=outer north east,
]
\addplot [color=myred,solid,line width=1pt, mark size=2.5pt,mark=o, mark options={color=myred}]
  table[row sep=crcr]{
2    2.65859834e-02   \\
3    3.04420826e-03    \\
4    2.46123306e-04   \\
5    3.10489244e-05   \\
};
\addlegendentry{\scriptsize $\mathbf{u}_h$(dG-dG-dG)}

\addplot [color=myblue,solid,line width=1pt, mark size=2.5pt,mark=o,mark options={color=myblue}]
  table[row sep=crcr]{
100    100 \\
101    100 \\
};
\addlegendentry{\scriptsize $p_h$(dG-dG-dG)}

\addplot [color=mygreen,solid,line width=1pt, mark size=2.5pt,mark=o, mark options={color=mygreen}]
  table[row sep=crcr]{
2    1.63966062e+00    \\
3    3.71443883e-01     \\
4    3.75741760e-02   \\
5    4.82733192e-03   \\
};
\addlegendentry{\scriptsize $T_h$(dG-dG-dG)}

\addplot [color=black,dashed,line width=1.0pt]
  table[row sep=crcr]{
2   0.20214 \\
3   0.016593 \\
4   0.001362 \\
5   0.0001118 \\
};
\addlegendentry{\scriptsize $e^{-2.5 \ell}$} 

\end{axis}
\end{tikzpicture}
\end{subfigure}
\end{subfigure}

\caption{Test case of Section~\ref{sec:conv_test2D}. Computed errors in $L^2$-norm (left) and energy-norms (right) versus the polynomial approximation degree $\ell$ (\texttt{semilog} scale).}
\label{fig:ConvP_2D}

\end{figure}
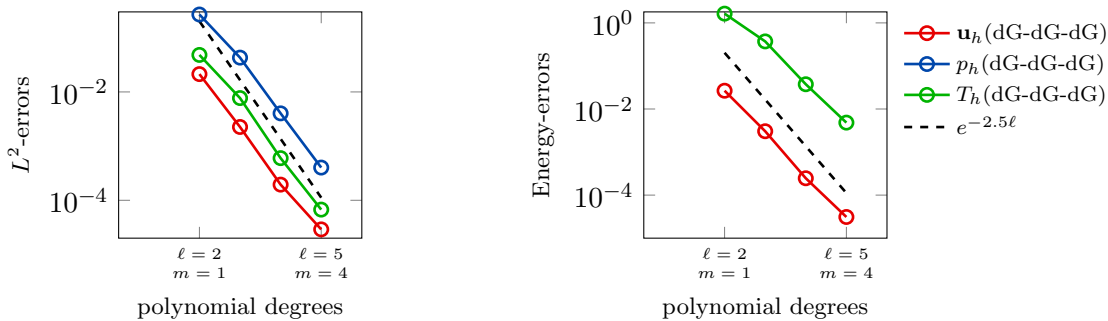
\end{center}


\subsection{Convergence test case in three-dimensions}
\label{sec:conv_test3D}
In this section, we set $\Omega = (0,1)^3$, take the coefficients as in Table~\ref{tab:params_convtest}, and consider boundary conditions and forcing terms inferred from the  manufactured solution:
\begin{equation}
	\begin{aligned}
		& \mathbf{u}(x,y,z) = \left( \begin{aligned}
			x^2 \sin(2 \pi y) \sin(2 \pi z), -\frac{x}{\pi} \cos(2\pi y) \sin(2 \pi x), \frac{2x}{\pi}  \sin(2 \pi y) \cos(2 \pi z)
		\end{aligned} \right)^T, \\
		& p(x,y,z) = (x^2 + 3y - 2xy + xz - z^2) \sin(2 \pi x)         \sin(2 \pi y) \cos(2 \pi z), \\
		& T(x,y,z) = (-3x + 2y^2 + 4yz + z) \cos(2 \pi x) \cos(2 
		\pi y) \sin(2 \pi z).
	\end{aligned}
\end{equation}

\begin{center}
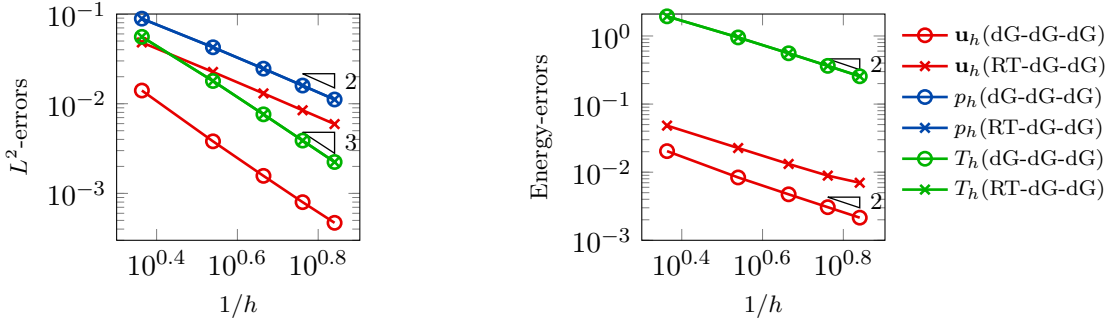
\begin{figure}[!htbp]
\begin{subfigure}[b]{1\textwidth}
\begin{subfigure}[b]{0.4\textwidth}
\begin{tikzpicture}
\begin{axis}[%
width=3.2cm,
height=3cm,
at={(0\textwidth,0\textwidth)},
scale only axis,
xmode=log,
xmin=2,
xmax=8,
xminorticks=true,
xlabel={\footnotesize $1/h$},
ymode=log,
ymin=0.0003,
ymax=0.1,
yminorticks=true,
ylabel={\footnotesize $L^2$-errors},
legend style={draw=none,fill=none,legend cell align=left},
legend pos=outer north east,
]
\addplot [color=myred,solid,line width=1pt, mark size=2.5pt,mark=o, mark options={color=myred}]
  table[row sep=crcr]{
2.30940108     0.01402791   \\
3.46410162    0.00380732   \\
4.61880215    0.0015686   \\
5.77350269   0.00079877   \\
6.92820323   0.00046987   \\
};

\addplot [color=myred,solid,line width=1pt, mark size=2.5pt,mark=x,mark options={color=myred}]
  table[row sep=crcr]{
2.30940108     0.04821547   \\
3.46410162    0.02259596   \\
4.61880215    0.0130285   \\
5.77350269   0.0084494   \\
6.92820323   0.0059297   \\
};

\addplot [color=myblue,solid,line width=1pt, mark size=2.5pt,mark=o,mark options={color=myblue}]
  table[row sep=crcr]{
2.30940108    0.08879431 \\
3.46410162   0.04257719 \\
4.61880215   0.02459869 \\
5.77350269   0.01593971 \\
6.92820323  0.01114419 \\
};

\addplot [color=myblue,solid,line width=1pt, mark size=2.5pt,mark=x,mark options={color=myblue}]
  table[row sep=crcr]{
2.30940108    0.08880542 \\
3.46410162   0.04258123 \\
4.61880215   0.02460065 \\
5.77350269   0.01594097 \\
6.92820323  0.01114643 \\
};

\addplot [color=mygreen,solid,line width=1pt, mark size=2.5pt,mark=o,mark options={color=mygreen}]
  table[row sep=crcr]{
2.30940108    0.05566321\\
3.46410162   0.01792246 \\
4.61880215   0.00761197 \\
5.77350269   0.0038808 \\
6.92820323  0.00223896 \\
};

\addplot [color=mygreen,solid,line width=1pt, mark size=2.5pt,mark=x,mark options={color=mygreen}]
  table[row sep=crcr]{
2.30940108    0.05566931\\
3.46410162   0.01792388 \\
4.61880215   0.00761286 \\
5.77350269   0.00388139 \\
6.92820323  0.00223951 \\
};

\addplot [color=black,solid,line width=0.5pt]
  table[row sep=crcr]{
 5.77350269   0.0048 \\
 6.92820323   0.0048 \\
 6.92820323   0.0028 \\
 5.77350269   0.0048 \\ 
};

\node[right, align=left, text=black, font=\footnotesize]
at (axis cs:6.93,0.0037) {3}; 

\addplot [color=black,solid,line width=0.5pt]
table[row sep=crcr]{
	5.77350269   0.0216 \\
	6.92820323   0.0216 \\
	6.92820323   0.015 \\
	5.77350269   0.0216 \\ 
};

\node[right, align=left, text=black, font=\footnotesize]
at (axis cs:6.93,0.018) {2};

\end{axis}
\end{tikzpicture}
\end{subfigure}
\begin{subfigure}[b]{0.59\textwidth}
\begin{tikzpicture}
\begin{axis}[%
width=3.2cm,
height=3cm,
at={(0\textwidth,0\textwidth)},
scale only axis,
xmode=log,
xmin=2,
xmax=8,
xminorticks=true,
xlabel={\footnotesize $1/h$},
ymode=log,
ymin=0.001,
ymax=2.1,
yminorticks=true,
ylabel={\footnotesize Energy-errors},
legend style={draw=none,fill=none,legend cell align=left},
legend pos=outer north east,
]
\addplot [color=myred,solid,line width=1pt, mark size=2.5pt,mark=o, mark options={color=myred}]
  table[row sep=crcr]{
2.30940108    0.020377793    \\
3.46410162   0.00836954    \\
4.61880215   0.00473185    \\
5.77350269   0.003063    \\
6.92820323  0.00214872   \\
};
\addlegendentry{\scriptsize $\mathbf{u}_h$(dG-dG-dG)}

\addplot [color=myred,solid,line width=1pt, mark size=2.5pt,mark=x, mark options={color=myred}]
  table[row sep=crcr]{
2.30940108    0.04829844    \\
3.46410162   0.02263057    \\
4.61880215   0.01321004    \\
5.77350269   0.00890197    \\
6.92820323  0.00699724   \\
};
\addlegendentry{\scriptsize $\mathbf{u}_h$(RT-dG-dG)}

\addplot [color=myblue,solid,line width=1pt, mark size=2.5pt,mark=o,mark options={color=myblue}]
  table[row sep=crcr]{
10000    100 \\
10001    100 \\
};
\addlegendentry{\scriptsize $p_h$(dG-dG-dG)}

\addplot [color=myblue,solid,line width=1pt, mark size=2.5pt,mark=x,mark options={color=myblue}]
  table[row sep=crcr]{
10000    100 \\
10001    100 \\
};
\addlegendentry{\scriptsize $p_h$(RT-dG-dG)}

\addplot [color=mygreen,solid,line width=1pt, mark size=2.5pt, mark=o, mark options={color=mygreen}]
  table[row sep=crcr]{
2.30940108    1.94260558  \\
3.46410162   0.95207666   \\
4.61880215   0.55701888  \\
5.77350269   0.36397315  \\
6.92820323  0.25593983 \\
};
\addlegendentry{\scriptsize $T_h$(dG-dG-dG)}

\addplot [color=mygreen,solid,line width=1pt, mark size=2.5pt, mark=x, mark options={color=mygreen}]
  table[row sep=crcr]{
2.30940108    1.94257787  \\
3.46410162   0.9520698   \\
4.61880215   0.55701532  \\
5.77350269   0.36397099  \\
6.92820323  0.25593733 \\
};
\addlegendentry{\scriptsize $T_h$(RT-dG-dG)}

\addplot [color=black,solid,line width=0.5pt]
  table[row sep=crcr]{
 5.77350269    0.4608 \\
 6.92820323    0.4608 \\
 6.92820323    0.32 \\
 5.77350269    0.4608 \\ 
};
\node[right, align=left, text=black, font=\footnotesize]
at (axis cs:6.93, 0.39) {2}; 

\addplot [color=black,solid,line width=0.5pt]
  table[row sep=crcr]{
 5.77350269    0.0043 \\
 6.92820323    0.0043 \\
 6.92820323    0.003 \\
 5.77350269    0.0043 \\ 
};
\node[right, align=left, text=black, font=\footnotesize]
at (axis cs:6.92820323, 0.0037) {2};

\end{axis}
\end{tikzpicture}
\end{subfigure}
\end{subfigure}

\caption{Test case of Section~\ref{sec:conv_test3D}. Computed errors in $L^2$-norm (left) and energy-norms (right) versus $1/h$ (\texttt{log-log} scale).}
\label{fig:ConvH_3D}

\end{figure}
\end{center}
We consider a sequence of successively refined tetrahedral meshes and set $\ell = 2, \, m = 1$. We have repeated the previous test case and reported in Figure~\ref{fig:ConvH_3D} the computed errors versus $1/h$ (log-log scale). 
We observe that the convergence rates are as expected. In this set of simulations, the fixed-point iterative algorithm is stopped when the relative difference between two successive iterations is below $10^{-8}$. On average, $17$ iterations of the fixed-point algorithm are required for achieving convergence of the dG-dG-dG method, while $12$ iterations are needed for the RT-dG-dG one. 

\subsection{Advection-dominated temperature transport in two-dimensions}
\label{sec:advectionstep2D}
In this section we propose a test case, inspired by \cite{Antonietti2022}, in which we investigate the transport of temperature by a flow that is governed by the DF law. We set $\Omega = (0,4) \times (0,2) \, \setminus \, (2,4) \times (0,1)$, with $\Gamma_{\text{in}}={0}\times (0,2)$ and  $\Gamma_{\text{out}}={4}\times (1,2)$ and set the following boundary conditions:
\begin{equation}
	\left\{
	\begin{aligned}
		& \mathbf{u} \cdot \mathbf{n} = u_{\text{in}} && \text{on} \ \ \Gamma_{\text{in}} \\
		& \mathbf{u} \cdot \mathbf{n} = u_{\text{out}} && \text{on} \ \ \Gamma_{\text{out}} \\
		& \mathbf{u}\cdot\mathbf{n} = 0 && \text{on} \ \ \partial \Omega \setminus \left( \Gamma_{\text{in}} \cup \ \Gamma_{\text{out}} \right)\\
		& T = T_{\text{in}} && \text{on} \ \ \Gamma_{\text{in}} \\
		& \boldsymbol{\Theta} \nabla T \cdot \mathbf{n} + \gamma(T - T_{\text{ext}})  = 0 \hspace{-4pt} && \text{on} \ \ \Omega \setminus \Gamma_{\text{in}}, 
	\end{aligned}
	\right.
\end{equation}
where $\gamma = 0.1$, $T_{\text{ext}} = 0.5$. The data $u_{\text{in}}$, $u_{\text{out}}$, and $T_{\text{in}}$ take the following general form
\begin{equation}
	\phi (x) = 
	\begin{cases}
		\phi_{m}, & \text{if} \ 0 \leq x < a\\
		\phi_{m} + \frac{\phi_{M}-\phi_{m}}{2}\left(1 - \cos\left(\pi \frac{x - a}{b-a}\right)\right),  & \text{if} \ a \leq x < b\\
		\phi_{M}, & \text{if} \ b \leq x < c\\
		\phi_{m} + \frac{\phi_{M}-\phi_{m}}{2}\left(1 - \cos\left(\pi \frac{x - c}{d-c}\right)\right), & \text{if} \ c \leq x < d\\
		\phi_{m}, & \text{if} \ d \leq x \leq e,
	\end{cases}
\end{equation}
with the following choice of parameters 
\begin{equation}
	\left\{
	\begin{aligned}
		&u_{\text{in}} = \phi(x) && \textrm{ with } (a,b,c,d,e,\phi_{m}, \phi_{M}) = (0.5, 0.9, 1.1, 1.5, 2,0,1)\\
		&u_{\text{out}} = \phi(x) && \textrm{ with } (a,b,c,d,e,\phi_{m}, \phi_{M}) = (1.25, 1.45, 1.55, 1.75, 2,0,1)\\
		&T_{\text{in}}= \phi(x) && \textrm{ with } (a,b,c,d,e,\phi_{m}, \phi_{M}) = (0.5, 0.9, 1.1, 1.5, 2,0,5).
	\end{aligned}
	\right.
\end{equation}
As $\mathbf{f} = \mathbf{0}$ and $g = 0$, the displacement, pressure, and temperature fields are determined only by the boundary conditions. The computational domain is discretized via a computational mesh made of $4635$ triangles with mesh size $h \sim 0.08 \si{\meter}$ and we set $\ell = 2$ for the temperature field and $m = \ell-1 = 1$ for the velocity and pressure fields. We use the RT-dG-dG scheme for solving the problem.
In Figure~\ref{fig:advectionstep2D} (left) and Figure~\ref{fig:advectionstep2D} (center), we show the results for the velocity and temperature fields, respectively. For clarity, velocity streamlines are superimposed on both fields. With the chosen boundary conditions, we mimic the injection of fluid in the central part of the inflow boundary $\Gamma_{\text{in}}$ and the extraction of fluid in the central part of the outflow boundary $\Gamma_{\text{out}}$. The injected fluid is hotter than the reference temperature of the domain, and a Robin boundary condition is used to model heat exchange with the surrounding subsoil.
\begin{figure}[ht!]
	\begin{subfigure}[b]{.32\textwidth}
		\centering
		\includegraphics[width=0.9\textwidth]{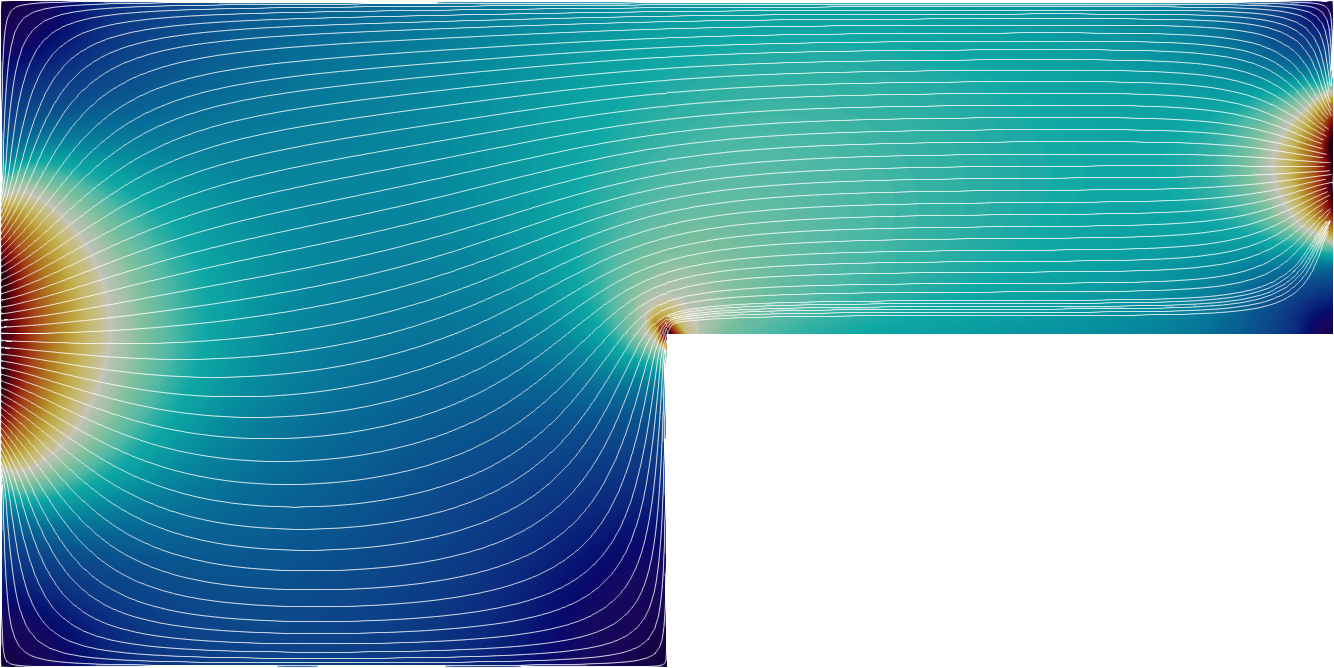}
		\label{fig:velocity2D}
	\end{subfigure}
	\begin{subfigure}[b]{.32\textwidth}
		\centering
		\includegraphics[width=0.9\textwidth]{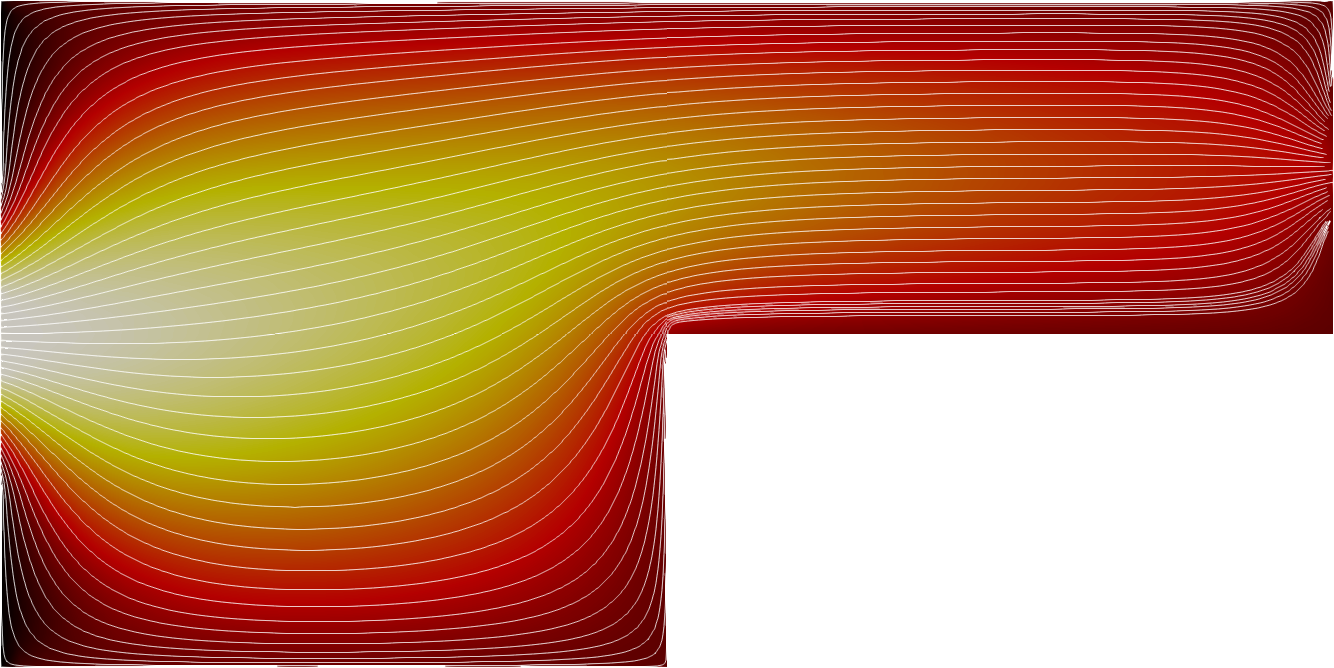}
		\label{fig:temperature2D}
	\end{subfigure}
	\begin{subfigure}[b]{.32\textwidth}
		\centering
		\includegraphics[width=0.9\textwidth]{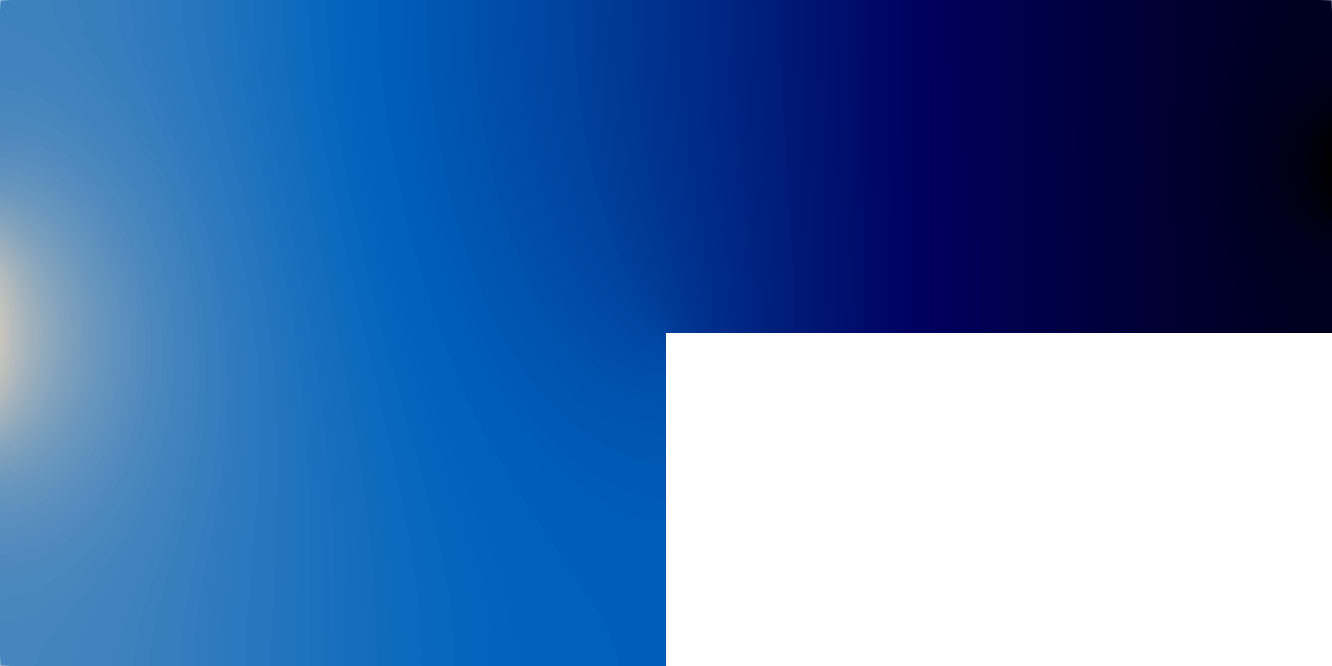}
		\label{fig:pressure2D}
	\end{subfigure}
	\begin{subfigure}[b]{.32\textwidth}
		\centering
		\includegraphics[width=0.9\textwidth]{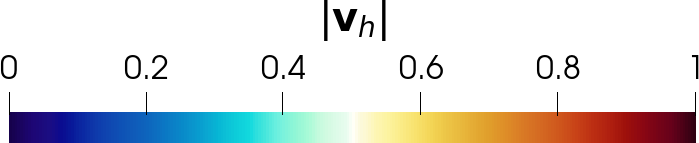}
		\label{fig:fig:velocity2D_colormap}
	\end{subfigure}
	\begin{subfigure}[b]{.32\textwidth}
		\centering
		\includegraphics[width=0.9\textwidth]{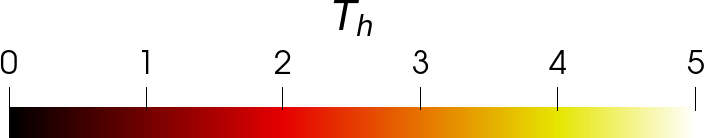}
		\label{fig:fig:temperature2D_colormap}
	\end{subfigure}
	\begin{subfigure}[b]{.32\textwidth}
		\centering
		\includegraphics[width=0.9\textwidth]{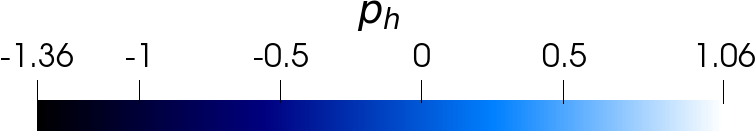}
		\label{fig:fig:pressure2D_colormap}
	\end{subfigure}
	\caption{Test case of Section~\ref{sec:advectionstep2D}: computed velocity field (left), temperature field (center), and pressure field (right).}
	\label{fig:advectionstep2D}
\end{figure}
In Figure~\ref{fig:advectionstep2D} (left), the velocity field is shown; as expected, it is driven by the inflow and outflow boundary conditions. The streamlines clearly illustrate the flow behavior and highlight the point of high velocity at the recessed corner of the L-shaped domain. This behavior is consistent with expectations. Regarding the temperature field (Figure~\ref{fig:advectionstep2D} (center)), we observe that, in this convection-dominated regime, the high temperature is transported by the fluid through the domain, and its distribution follows the flow lines.
Finally, in Figure~\ref{fig:advectionstep2D}(right), we display the pressure field. The pressure is higher in the fluid injection zone and then decreases in a manner resembling a linear drop towards the outflow boundary. To close the problem, we consider the pressure as an $L^2$ function with zero mean. Indeed, Figure~\ref{fig:advectionstep2D} (right) shows that the computed pressure field satisfies the zero-mean condition.

\subsection{Advection-dominated temperature transport in three-dimensions}
\label{sec:advectionstep3D}
In this section, we extend the test case presented in Section~\ref{sec:advectionstep2D} to a three-dimensional setting. Then, we set $\Omega = (0,4) \times (0,2) \times (0,2) \, \setminus \, (2,4) \times (0,1) \times (0,1)$ and consider a set of boundary conditions that is similar to the two-dimensional case, where $\Gamma_{\text{in}}$, $\Gamma_{\text{out}}$ are defined as $\Gamma_{\text{in}} = \{0\} \times (0,2) \times (0,2)$ and $\Gamma_{\text{out}} = \{4\} \times (0,2) \times (0,2) \setminus \{4\} \times (0,1) \times (0,1)$. The parameters of the functions $u_{\text{in}}$, $u_{\text{out}}$ are taken as in Section~\ref{sec:advectionstep2D} and they are considered constant along the $z$-direction. Moreover, as in the previous test case, we consider the forcing terms to be $\mathbf{f} = \mathbf{0}, g = 0$.  The computational domain is discretized with a mesh consisting of 40544 tetrahedrons with mesh size $h \sim 0.25\si{\meter}$; moreover we set $\ell = 2$ for the temperature field and $m = \ell-1 = 1$ for the velocity and pressure fields. The numerical results have been obtained based on employing the RT-dG-dG scheme. 
In Figure~\ref{fig:advectionstep3D_v}, we show the computed velocity field. Recall that this test case concerns convection-dominated temperature transport; thus a full understanding of the velocity field also provides insight into the temperature field.
We observe that the results are consistent with those obtained in the two-dimensional setting, and, as expected, peaks in the velocity field occur at the inflow and outflow boundaries and along the edges surrounding the removed corner of the parallelepipedal domain. This is evident in both the slices that highlight the $L$-shaped part of the domain, cf.  Figure~\ref{fig:advectionstep3D_v}  (top-right) and Figure~\ref{fig:advectionstep3D_v} (bottom-right), but it is also evident by looking at the rectangular slices, cf.  Figure~\ref{fig:advectionstep3D_v} (top-left) and Figure~\ref{fig:advectionstep3D_v} (bottom-left).
\begin{figure}[ht!]
	\centering
	\begin{subfigure}[b]{.4\textwidth}
		\centering
		\includegraphics[width=0.75\textwidth]{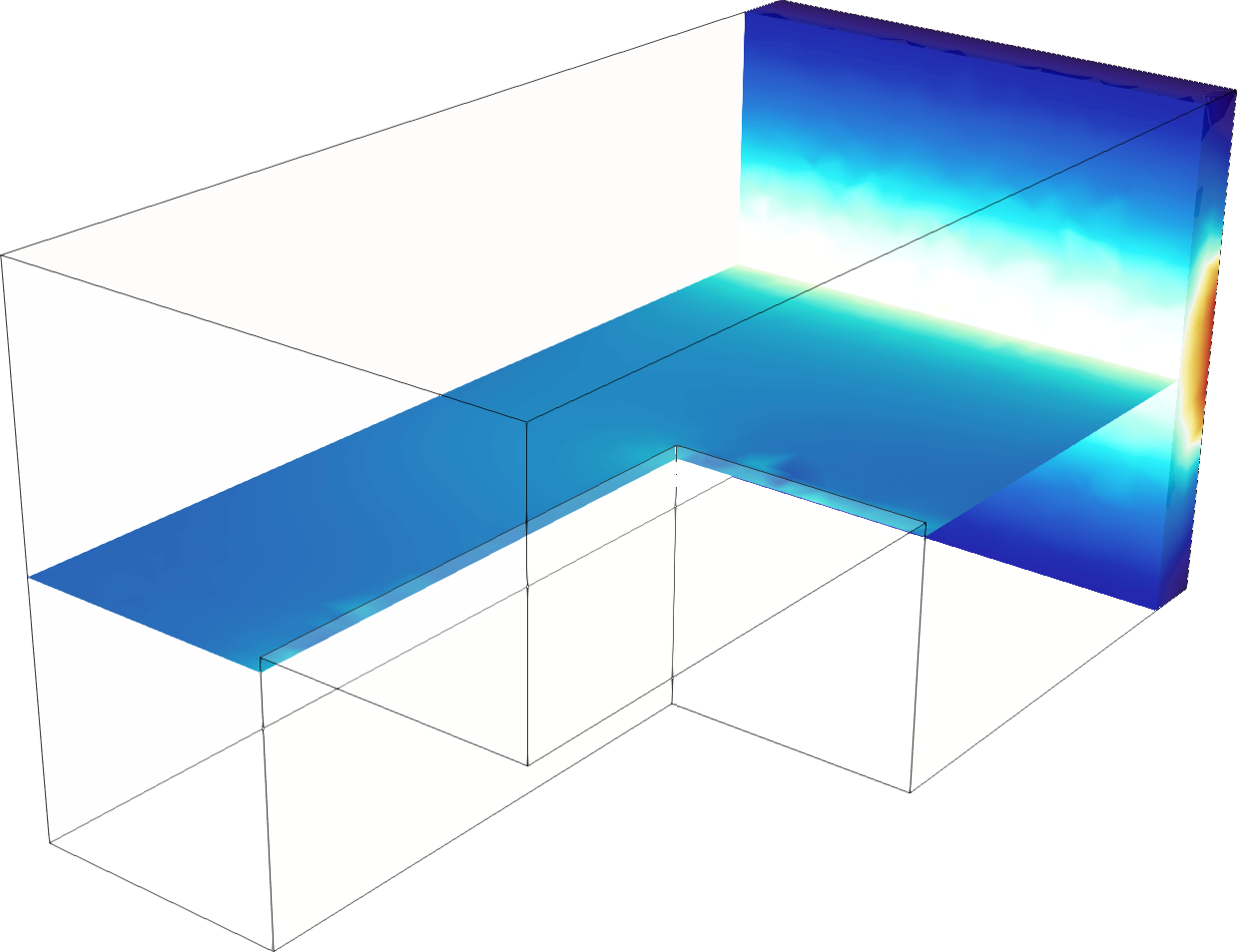}
		\label{fig:v3D_Hor095}
	\end{subfigure}
	\begin{subfigure}[b]{.4\textwidth}
		\centering
		\includegraphics[width=0.75\textwidth]{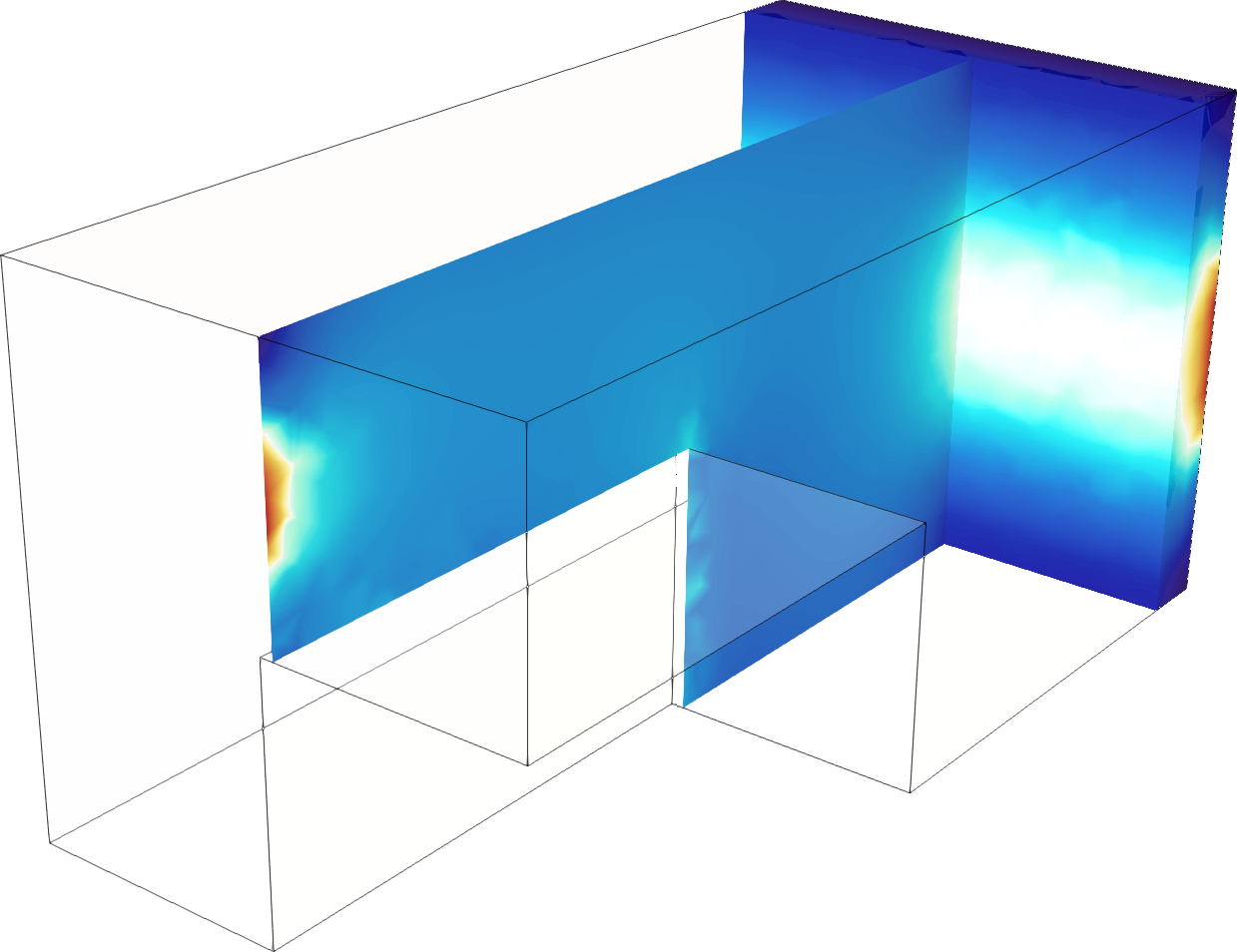}
		\label{fig:v3D_Ver095}
	\end{subfigure}
	
	\begin{subfigure}[b]{.4\textwidth}
		\centering
		\includegraphics[width=0.75\textwidth]{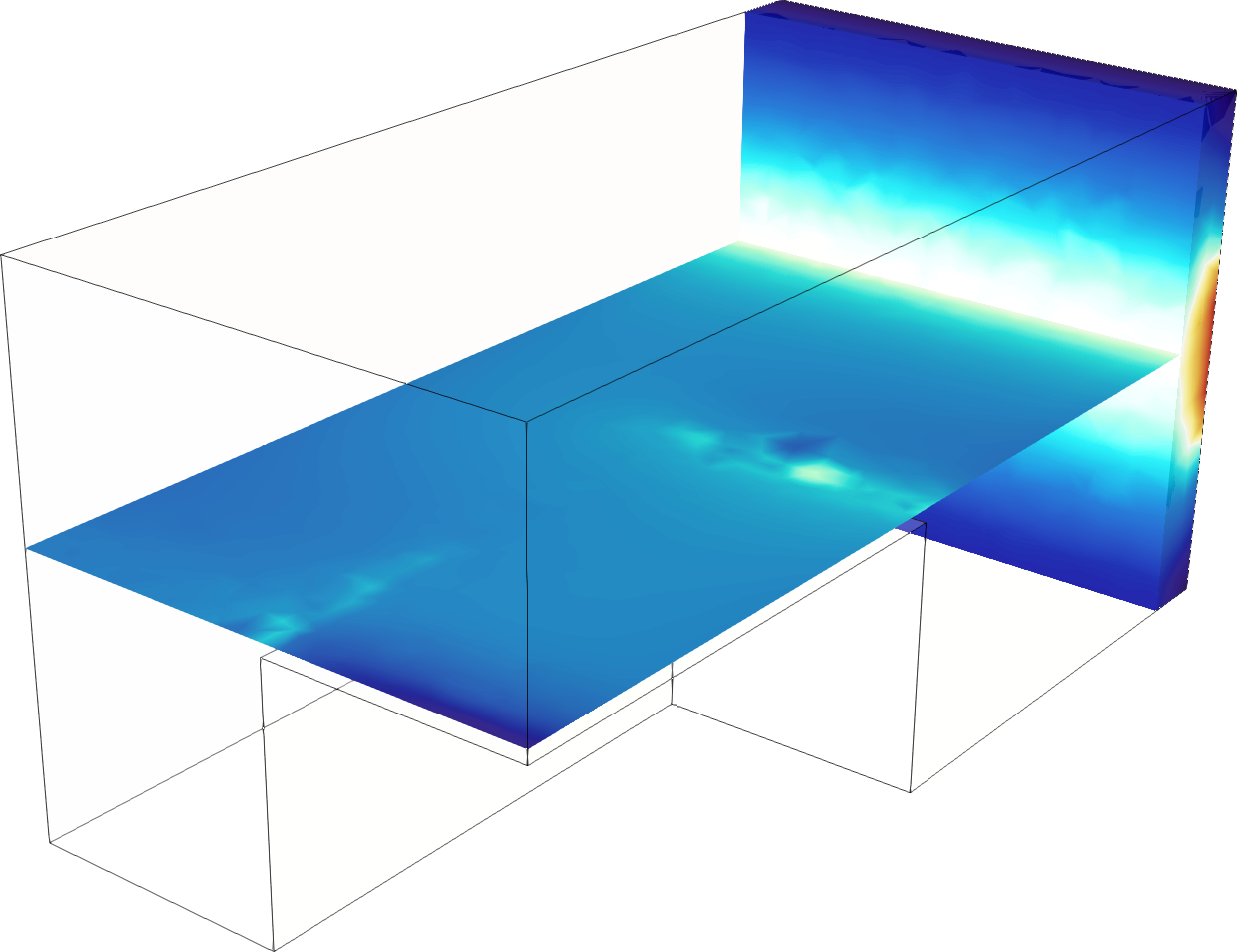}
		\label{fig:v3D_Hor105}
	\end{subfigure}
	\begin{subfigure}[b]{.4\textwidth}
		\centering
		\includegraphics[width=0.75\textwidth]{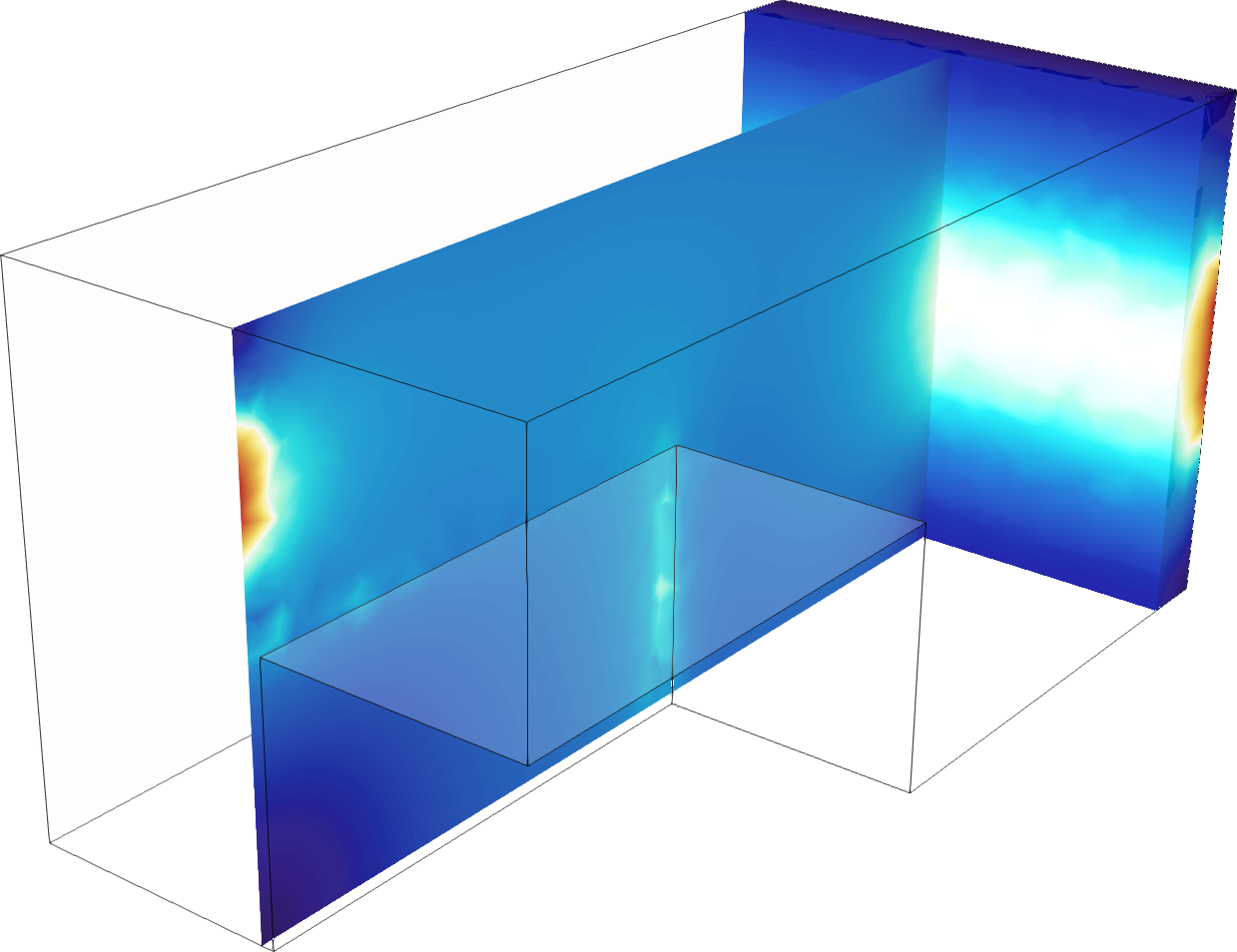}
		\label{fig:v3D_Ver105}
	\end{subfigure}
	
	\centering
	\begin{subfigure}[b]{.4\textwidth}
		\centering
		\includegraphics[width=0.6\textwidth]{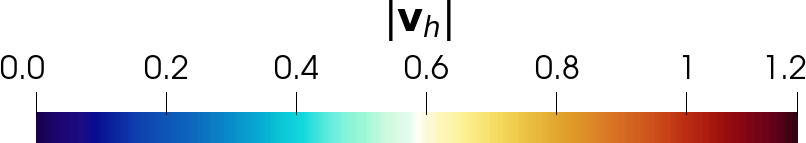}
		\label{fig:v3D_colormap}
	\end{subfigure}
	\caption{Test case of Section~\ref{sec:advectionstep3D}: computed velocity field. The domain is clipped at $x=0.5$ and the four slices are at $y=0.95$ (top-left), $y=1.05$ (bottom-left), $z=0.95$ (top-right), and $z=1.05$ (bottom-right)}
	\label{fig:advectionstep3D_v}
\end{figure}
In Figure~\ref{fig:advectionstep3D_T} we show the results for the temperature field. As in the previous test case, we consider a convection-dominated regime; then, as expected, in the behavior of the temperature field we clearly observe the convective phenomenon that predominates the diffusive one. We observe that, as in the two dimensional case, the temperature field follows the shape of the $L$-shaped domain (cf. Figure~\ref{fig:advectionstep3D_T} (top-right)). In the correspondent rectangular part of the domain (cf. Figure~\ref{fig:advectionstep3D_T} (bottom-right)) we still observe that shape of the field follows the $L$ shape and there is no temperature diffusion in the bottom-left corner of the domain. When looking at the Figure~\ref{fig:advectionstep3D_T} (top-left) and Figure~\ref{fig:advectionstep3D_T} (bottom-left) slices we observe that also in this direction, the shape of the temperature follows an $L$-shaped pattern, coherent with the velocity field observed in Figure~\ref{fig:advectionstep3D_v}.
\begin{figure}[ht!]
	\centering
	\begin{subfigure}[b]{.4\textwidth}
		\centering
		\includegraphics[width=0.75\textwidth]{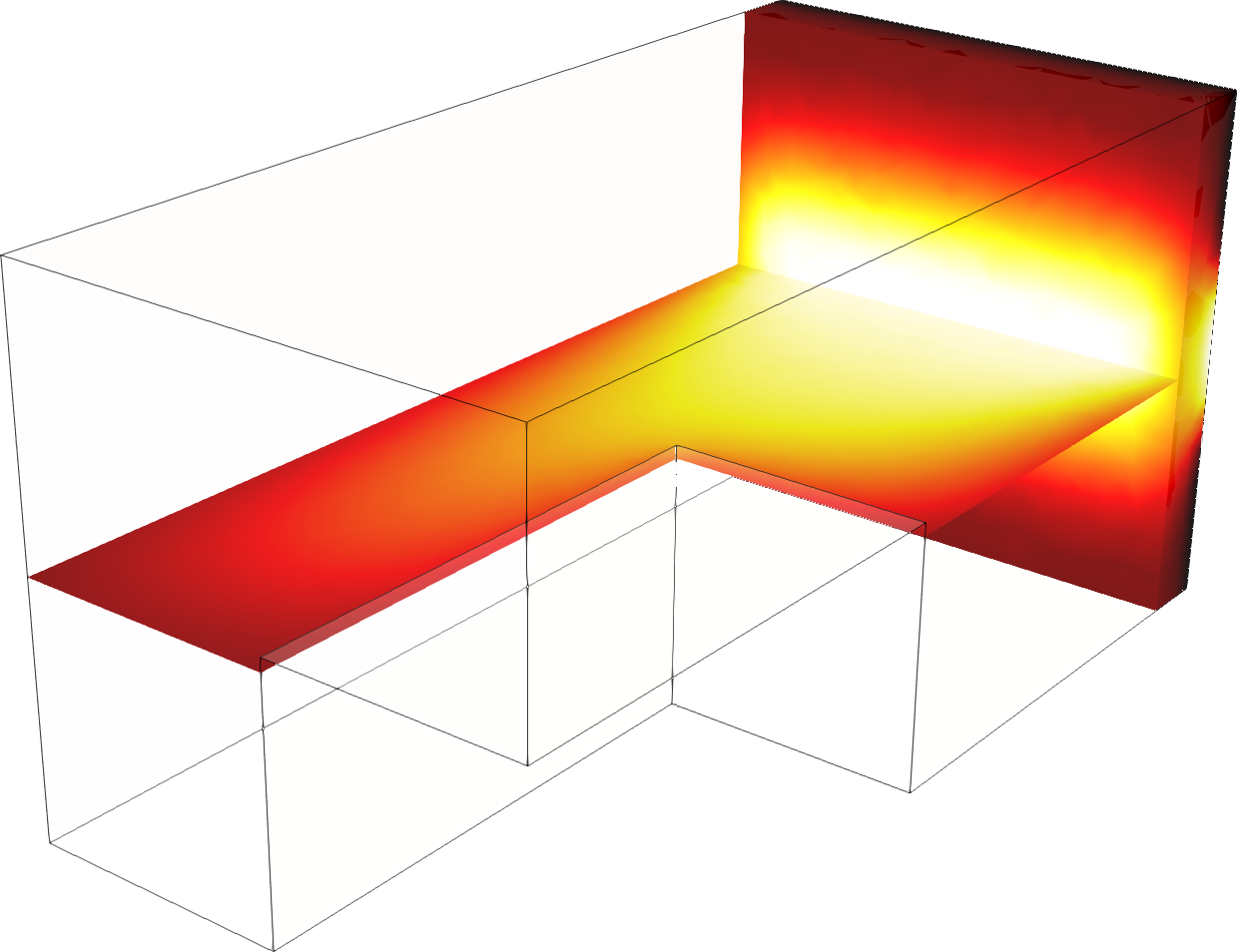}
		\label{fig:T3D_Hor095}
	\end{subfigure}
	\begin{subfigure}[b]{.4\textwidth}
		\centering
		\includegraphics[width=0.75\textwidth]{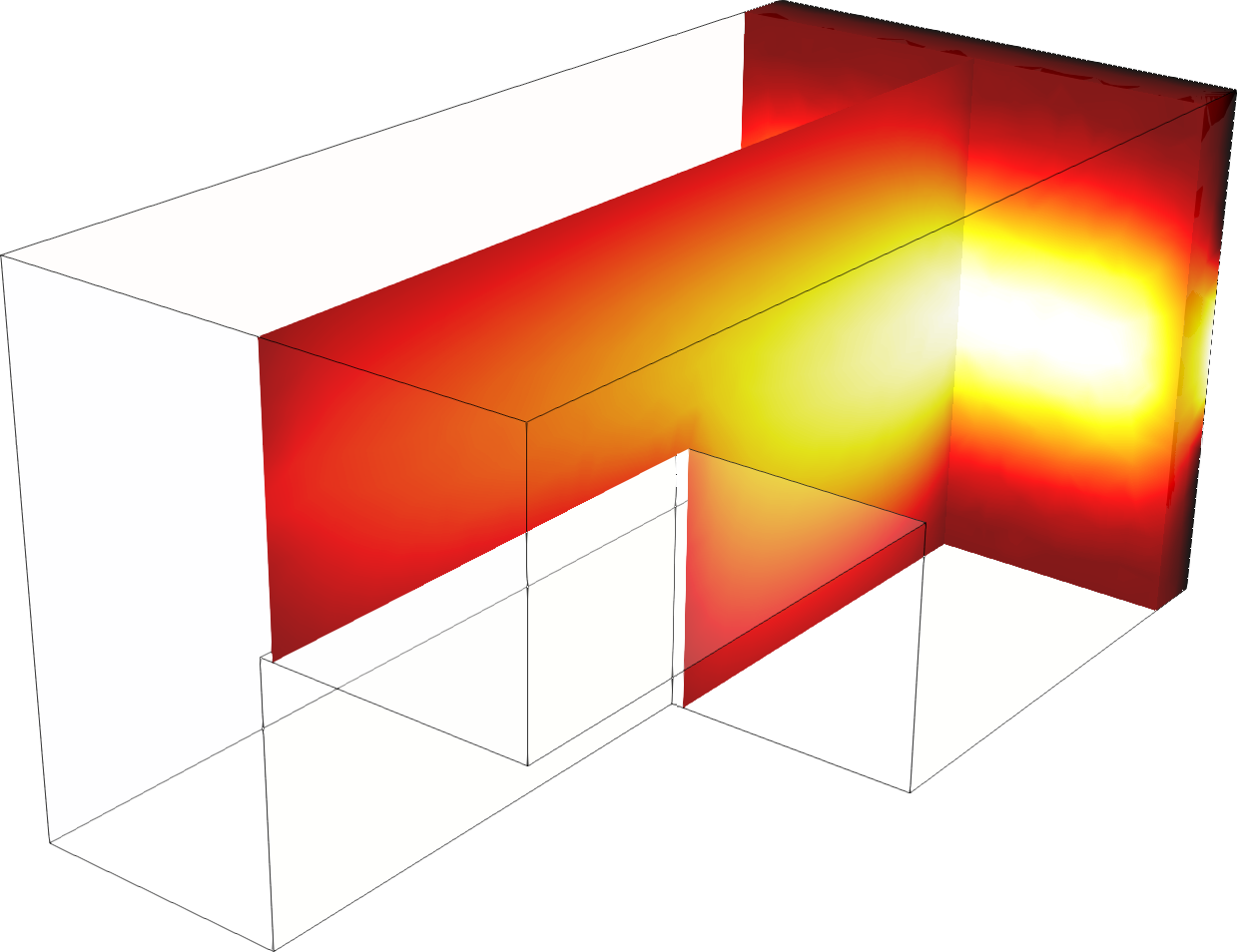}
		\label{fig:T3D_Ver095}
	\end{subfigure}
	\begin{subfigure}[b]{.4\textwidth}
		\centering
		\includegraphics[width=0.75\textwidth]{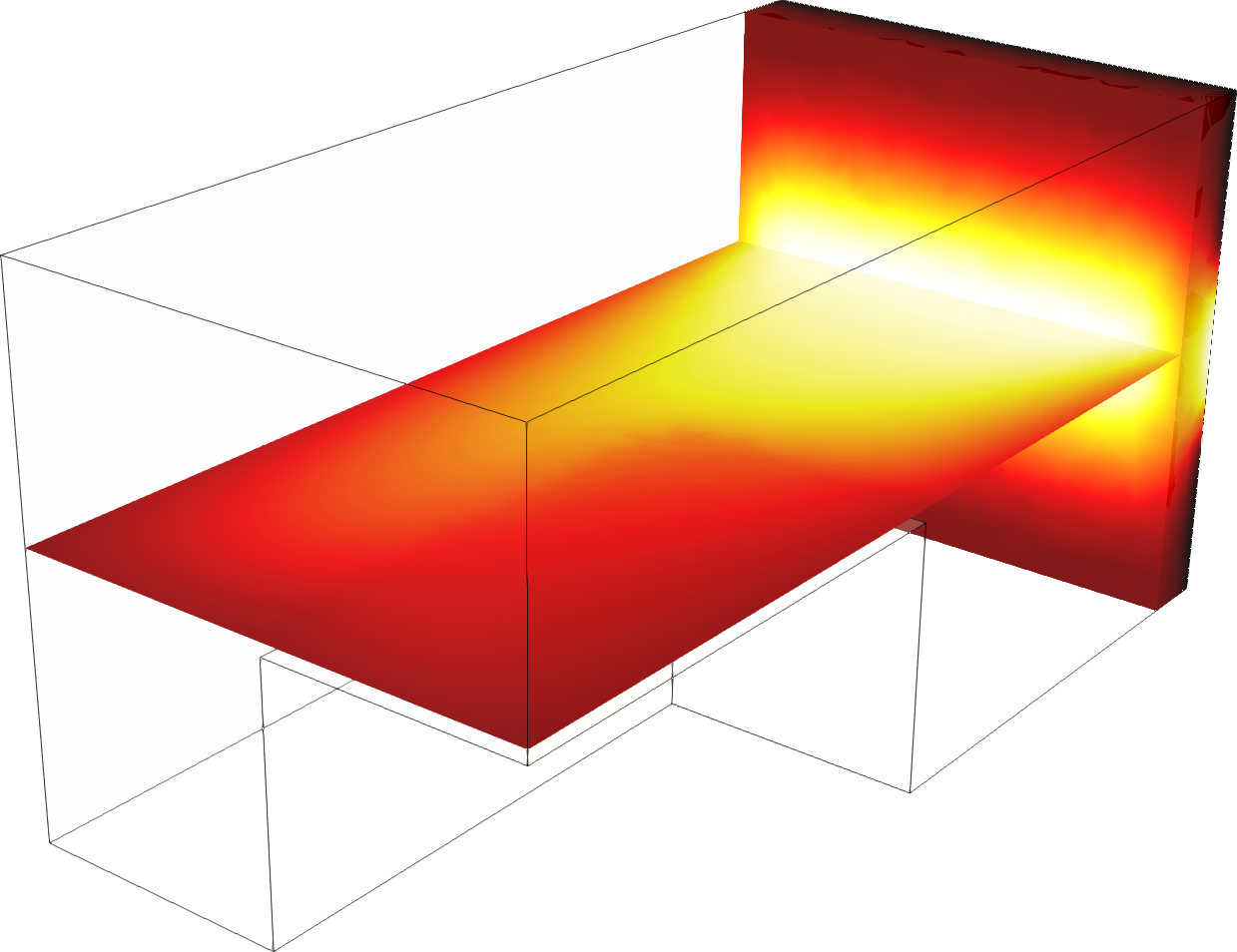}
		\label{fig:T3D_Hor105}
	\end{subfigure}
	\begin{subfigure}[b]{.4\textwidth}
		\centering
		\includegraphics[width=0.75\textwidth]{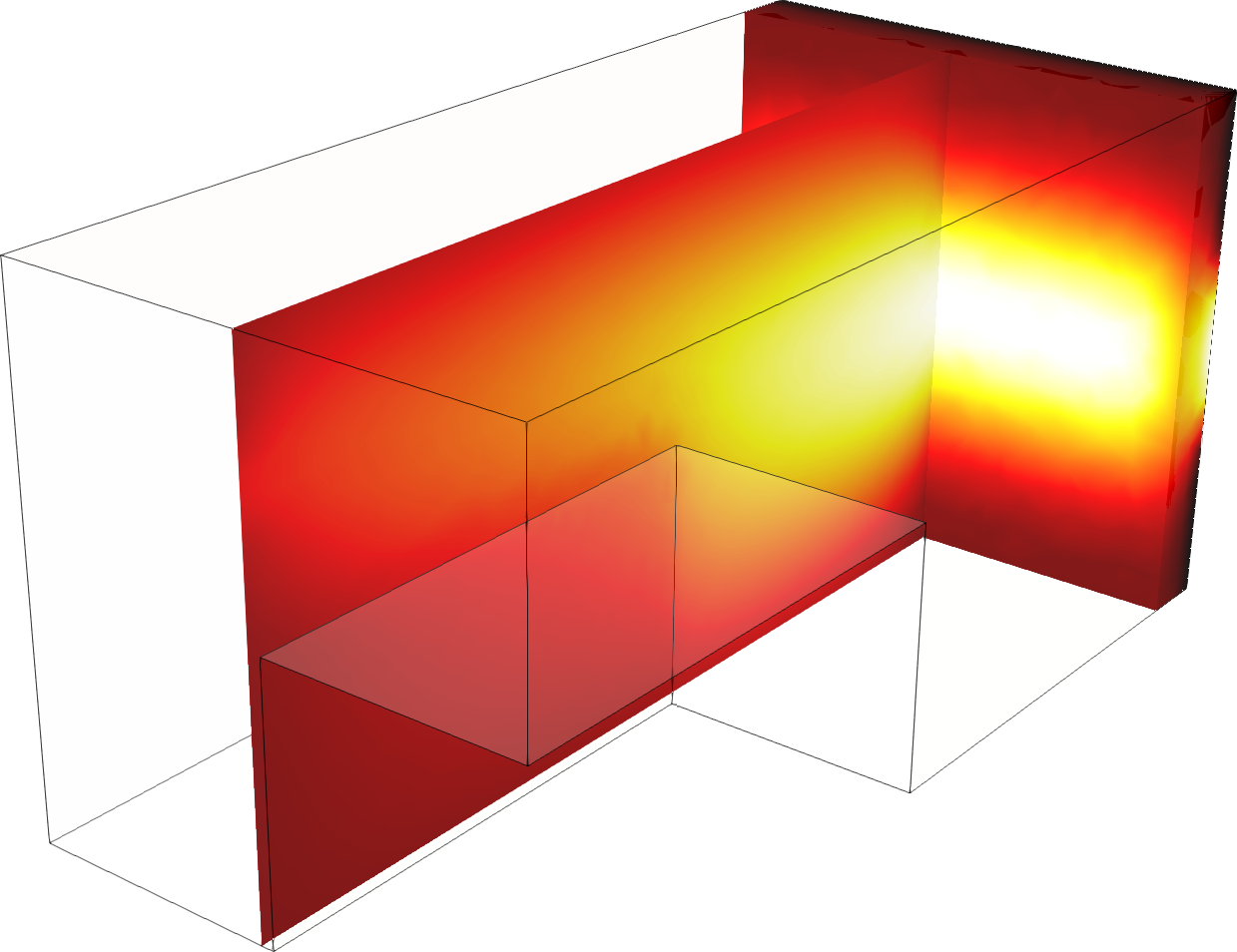}
		\label{fig:T3D_Ver105}
	\end{subfigure}
	\centering
	\begin{subfigure}[b]{.4\textwidth}
		\centering
		\includegraphics[width=0.6\textwidth]{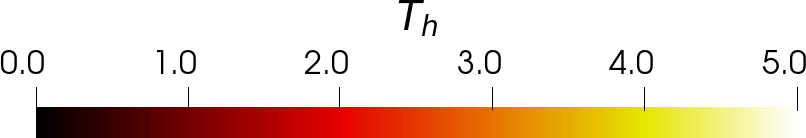}
		\label{fig:T3D_colormap}
	\end{subfigure}
	\caption{Test case of Section~\ref{sec:advectionstep3D}: computed temperature field. The domain is clipped at $x=0.5$ and the four slices are at $y=0.95$ (top-left), $y=1.05$ (bottom-left), $z=0.95$ (top-right), and $z=1.05$ (bottom-right)}
	\label{fig:advectionstep3D_T}
\end{figure}
In Figure~\ref{fig:advectionstep3D_p} we display the pressure field. We observe that the pressure value is higher in the zone of the fluid injection and then decreases with a behavior similar to a linear decay towards the outflow boundary. We can observe from the slices Figure~\ref{fig:advectionstep3D_p} (top-right), Figure~\ref{fig:advectionstep3D_p} (bottom-right) that the part of the domain in which we observe the lowest value of pressure is the outflow boundary of the $L$-shaped part. 
\begin{figure}[ht!]
	\centering
	\begin{subfigure}[b]{.4\textwidth}
		\centering
		\includegraphics[width=0.75\textwidth]{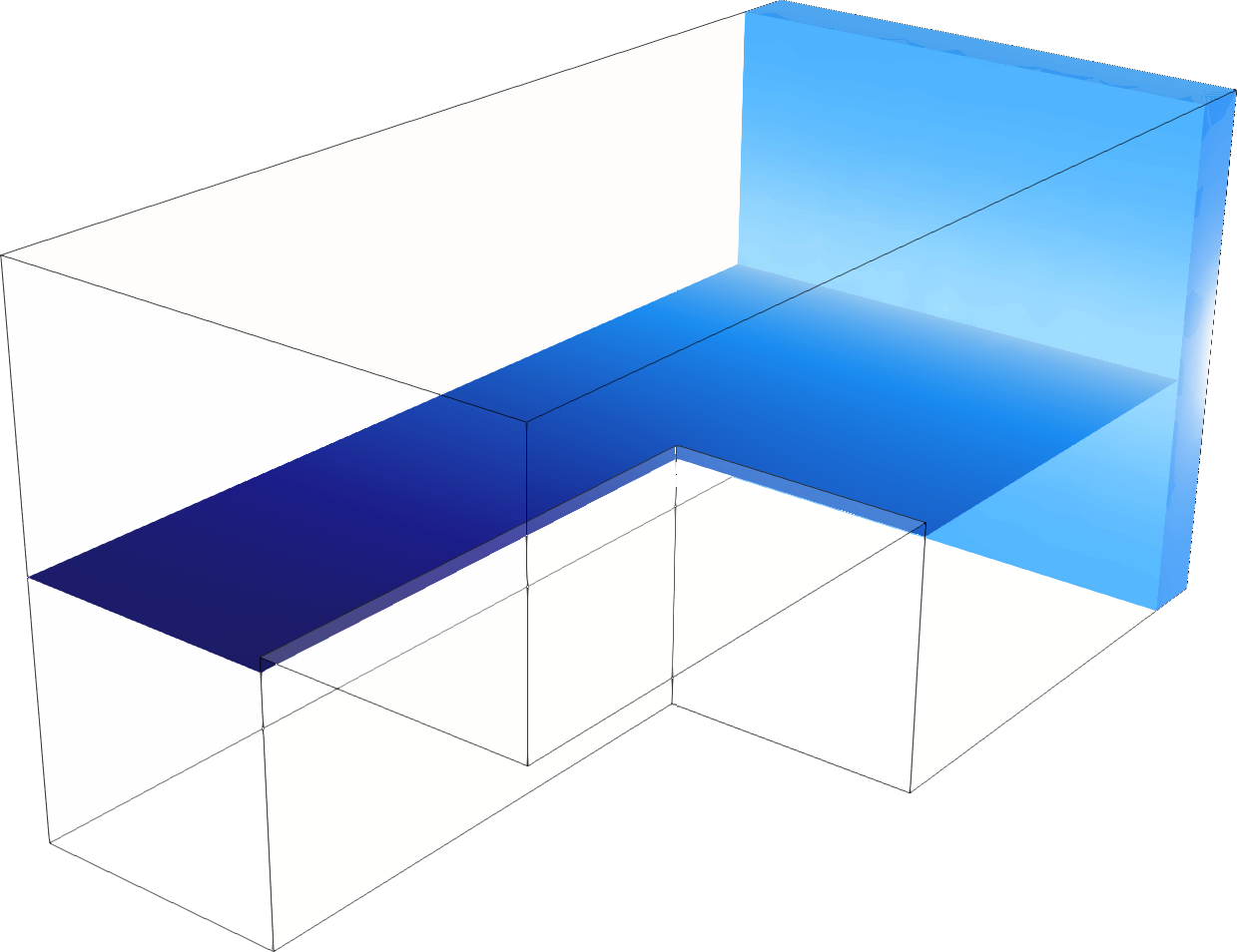}
		\label{fig:p3D_Hor095}
	\end{subfigure}
	\begin{subfigure}[b]{.4\textwidth}
		\centering
		\includegraphics[width=0.75\textwidth]{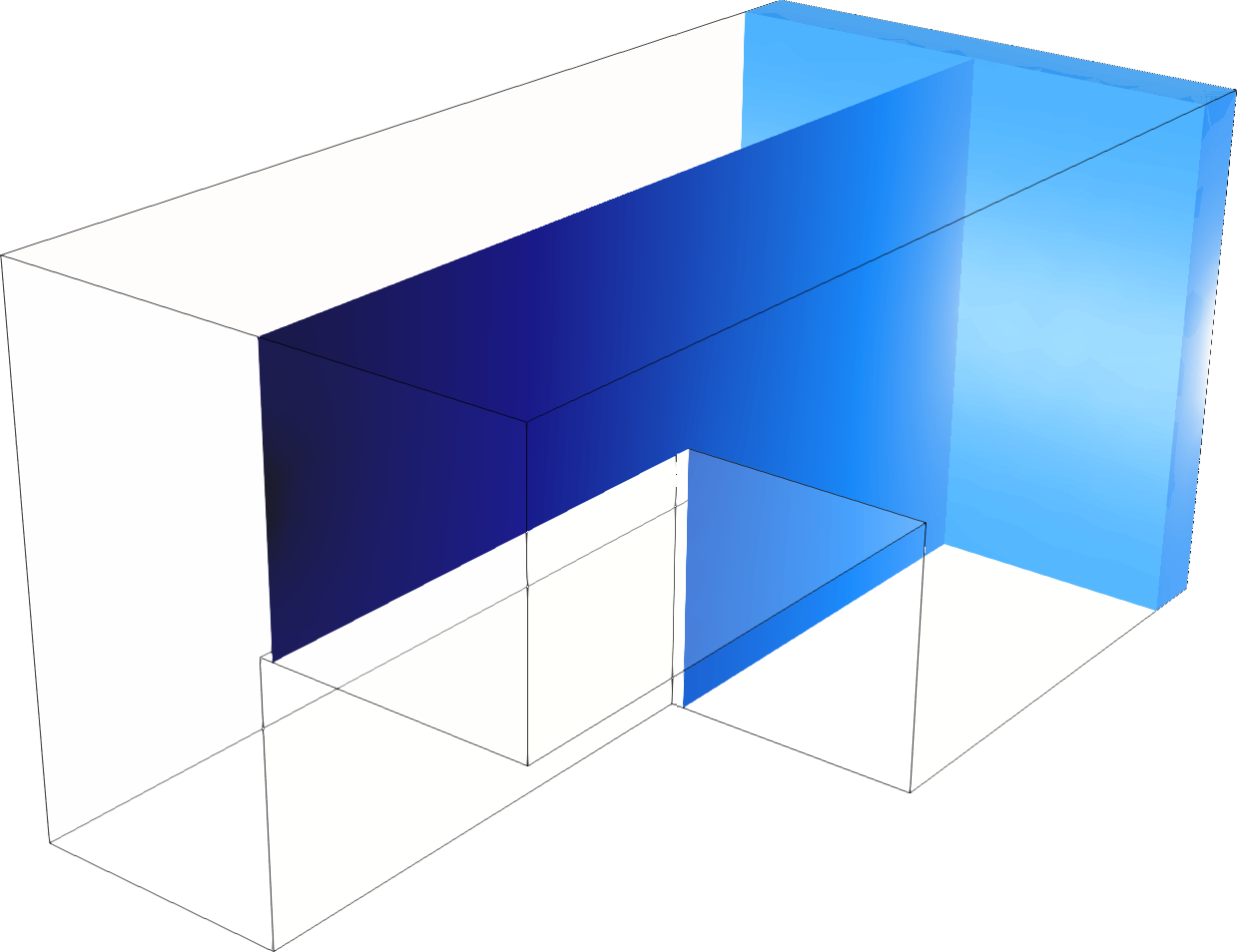}
		\label{fig:p3D_Ver095}
	\end{subfigure}
	\begin{subfigure}[b]{.4\textwidth}
		\centering
		\includegraphics[width=0.75\textwidth]{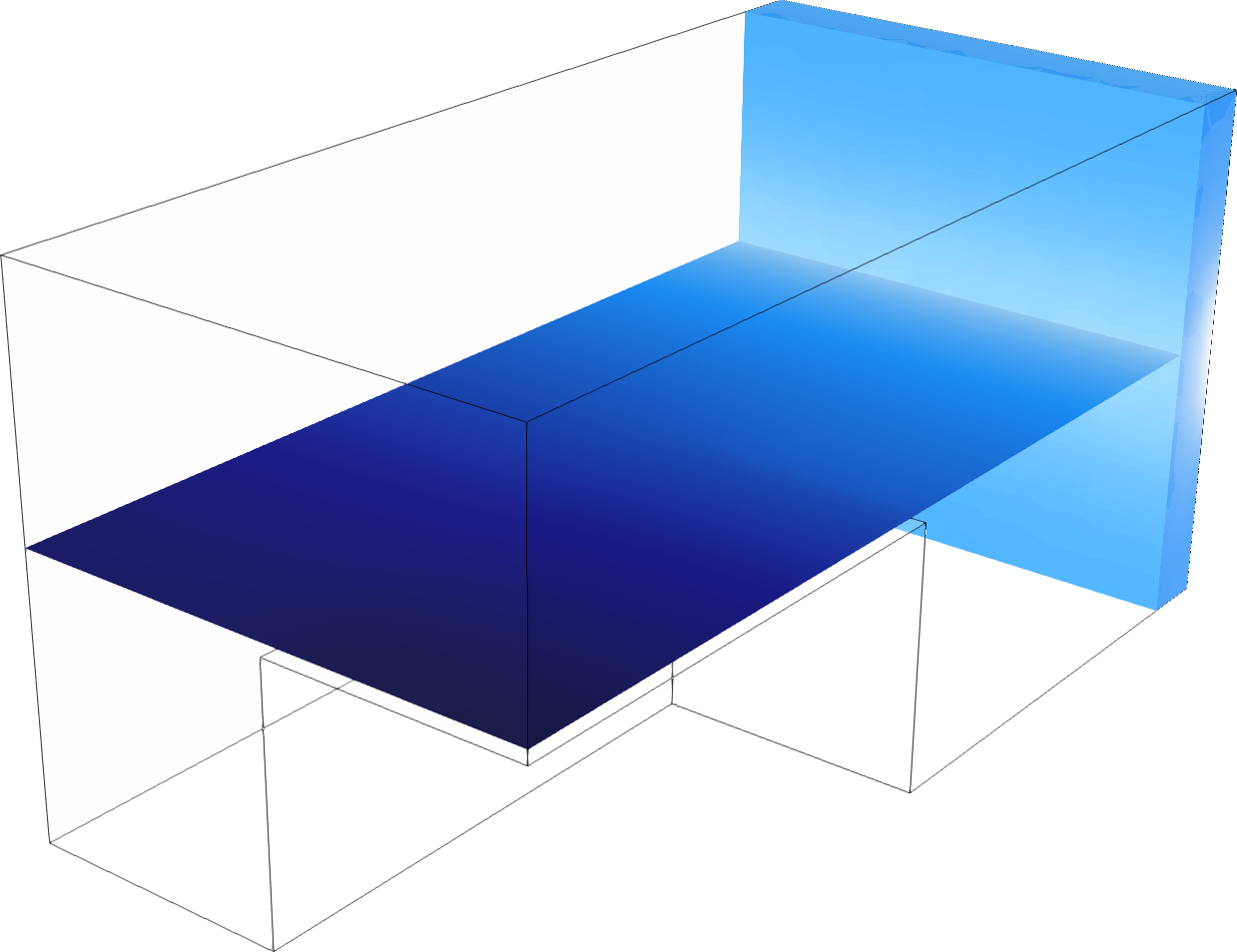}
		\label{fig:p3D_Hor105}
	\end{subfigure}
	\begin{subfigure}[b]{.4\textwidth}
		\centering
		\includegraphics[width=0.75\textwidth]{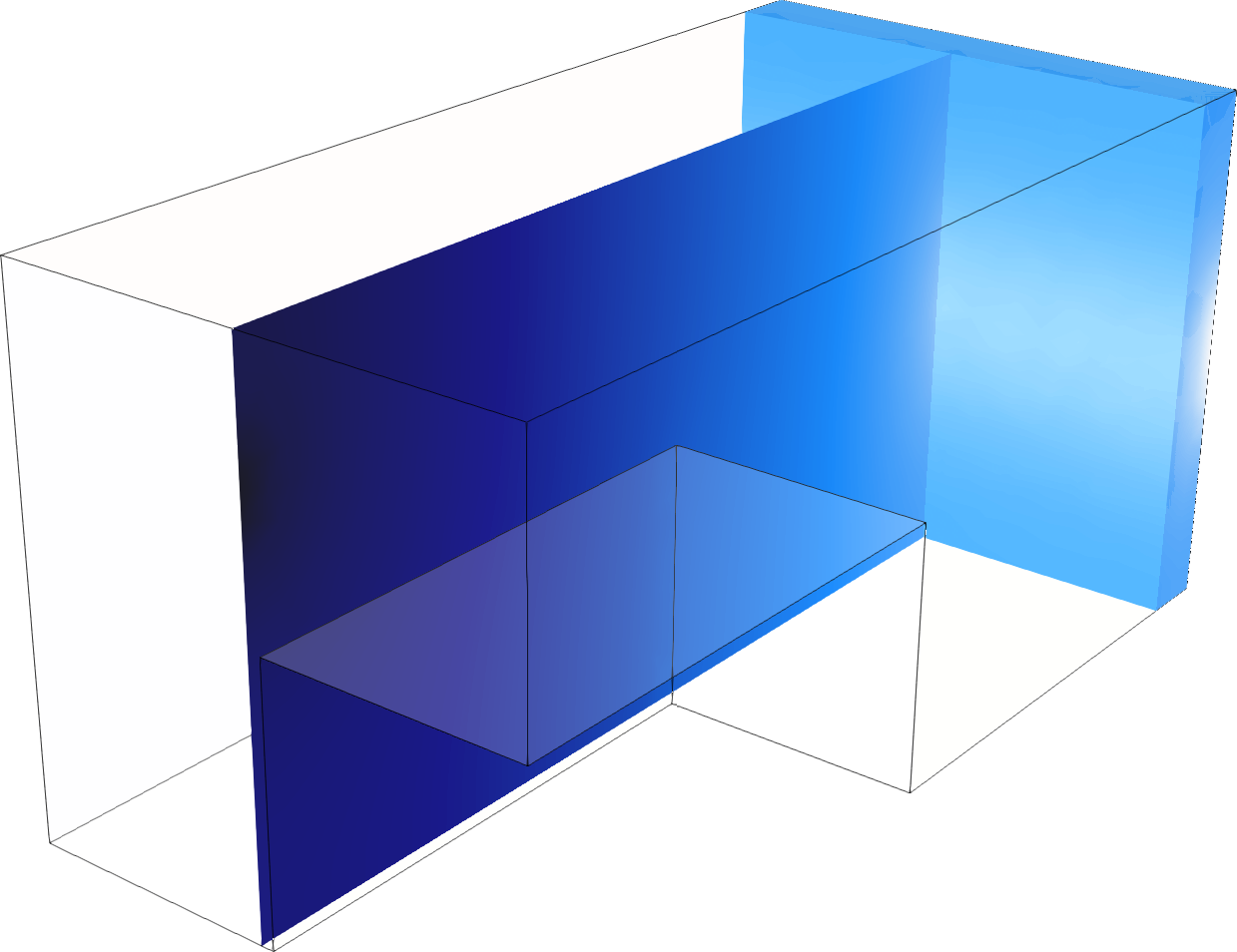}
		\label{fig:p3D_Ver105}
	\end{subfigure}
	\centering
	\begin{subfigure}[b]{.4\textwidth}
		\centering
		\includegraphics[width=0.6\textwidth]{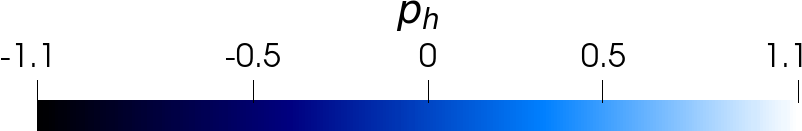}
		\label{fig:p3D_colormap}
	\end{subfigure}
	\caption{Test case of Section~\ref{sec:advectionstep3D}: computed pressure field. The domain is clipped at $x=0.5$ and the four slices are at $y=0.95$ (top-left), $y=1.05$ (bottom-left), $z=0.95$ (top-right), and $z=1.05$ (bottom-right)}
	\label{fig:advectionstep3D_p}
\end{figure}
In Figure~\ref{fig:advectionstep3D_volume} we give a volumetric representation of the three fields. In Figure~\ref{fig:advectionstep3D_volume} (left) it is possible to observe the behavior of the velocity field; we can observe that the fluid flows from the inflow boundary to the outflow one following the shape of the domain, with velocity peaks on the edges of the concave part of the domain, cf. Figure~\ref{fig:advectionstep3D_v} too. In Figure~\ref{fig:advectionstep3D_volume} (center) we observe the behavior of the temperature field. By comparing Figure~\ref{fig:advectionstep3D_volume} (left) and Figure~\ref{fig:advectionstep3D_volume} (center) it is evident the convection-dominated nature of the test case, as the regions of high temperature of the fluid follow the flow pattern. In particular, we observe high-temperature regions in the inner corner of the domain. 
\begin{figure}[ht!]
	\begin{subfigure}[b]{.32\textwidth}
		\centering
		\includegraphics[width=0.88\textwidth]{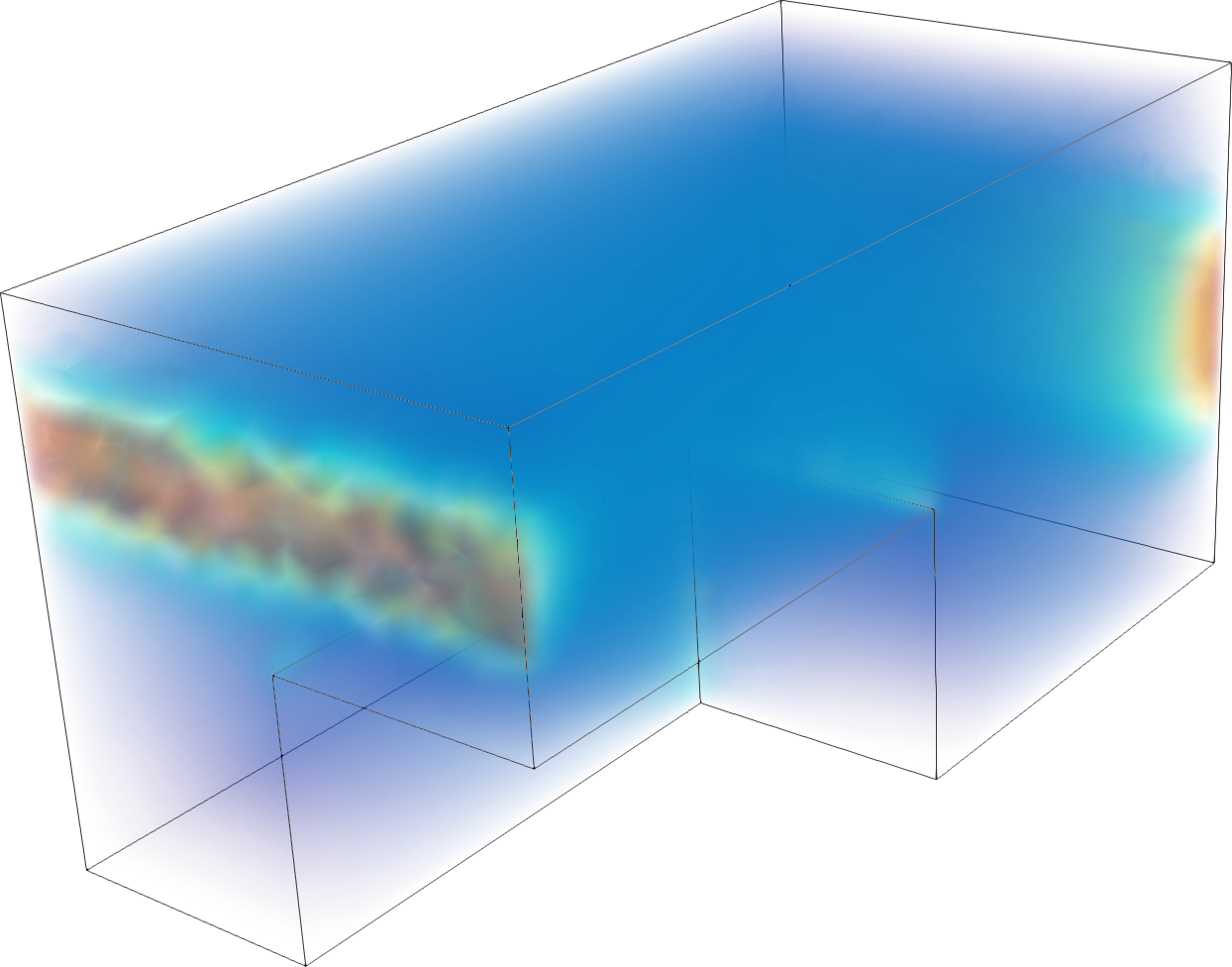}
		\label{fig:T3D_vol1}
	\end{subfigure}
	\begin{subfigure}[b]{.32\textwidth}
		\centering
		\includegraphics[width=0.89\textwidth]{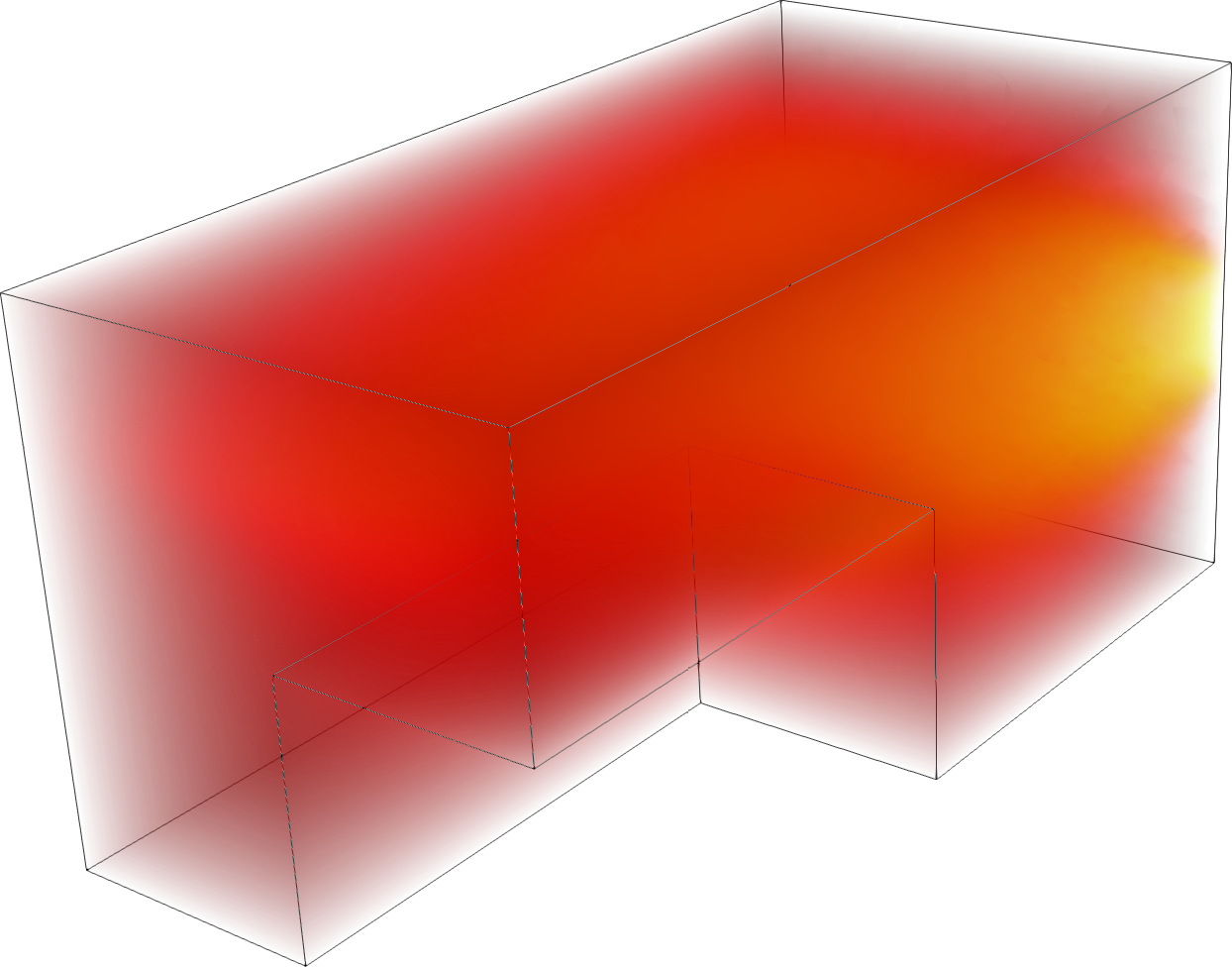}
		\label{fig:v3D_vol1}
	\end{subfigure}
	\begin{subfigure}[b]{.32\textwidth}
		\centering
		\includegraphics[width=0.89\textwidth]{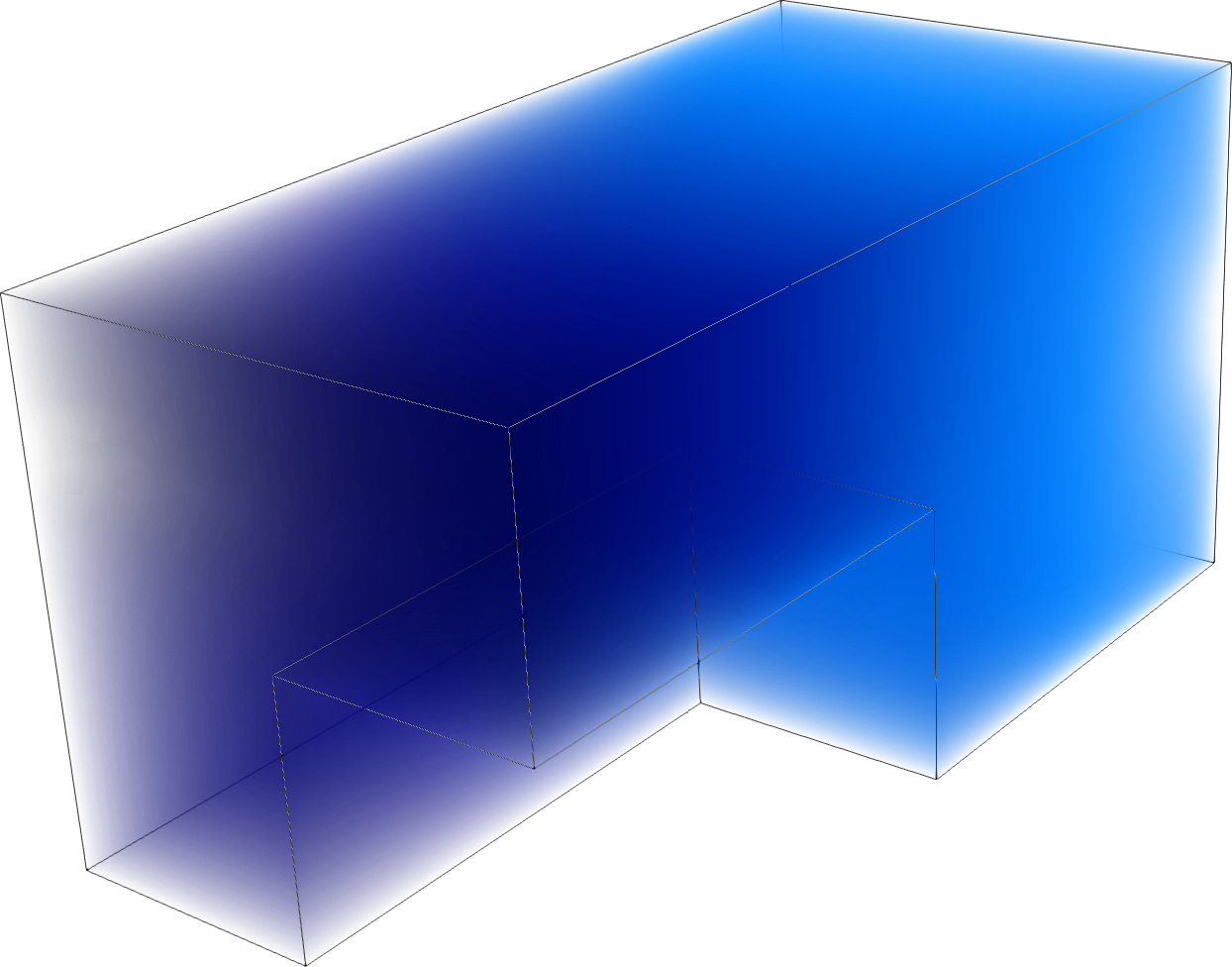}
		\label{fig:p3D_vol1}
	\end{subfigure}
	\begin{subfigure}[b]{.32\textwidth}
		\centering
		\includegraphics[width=0.89\textwidth]{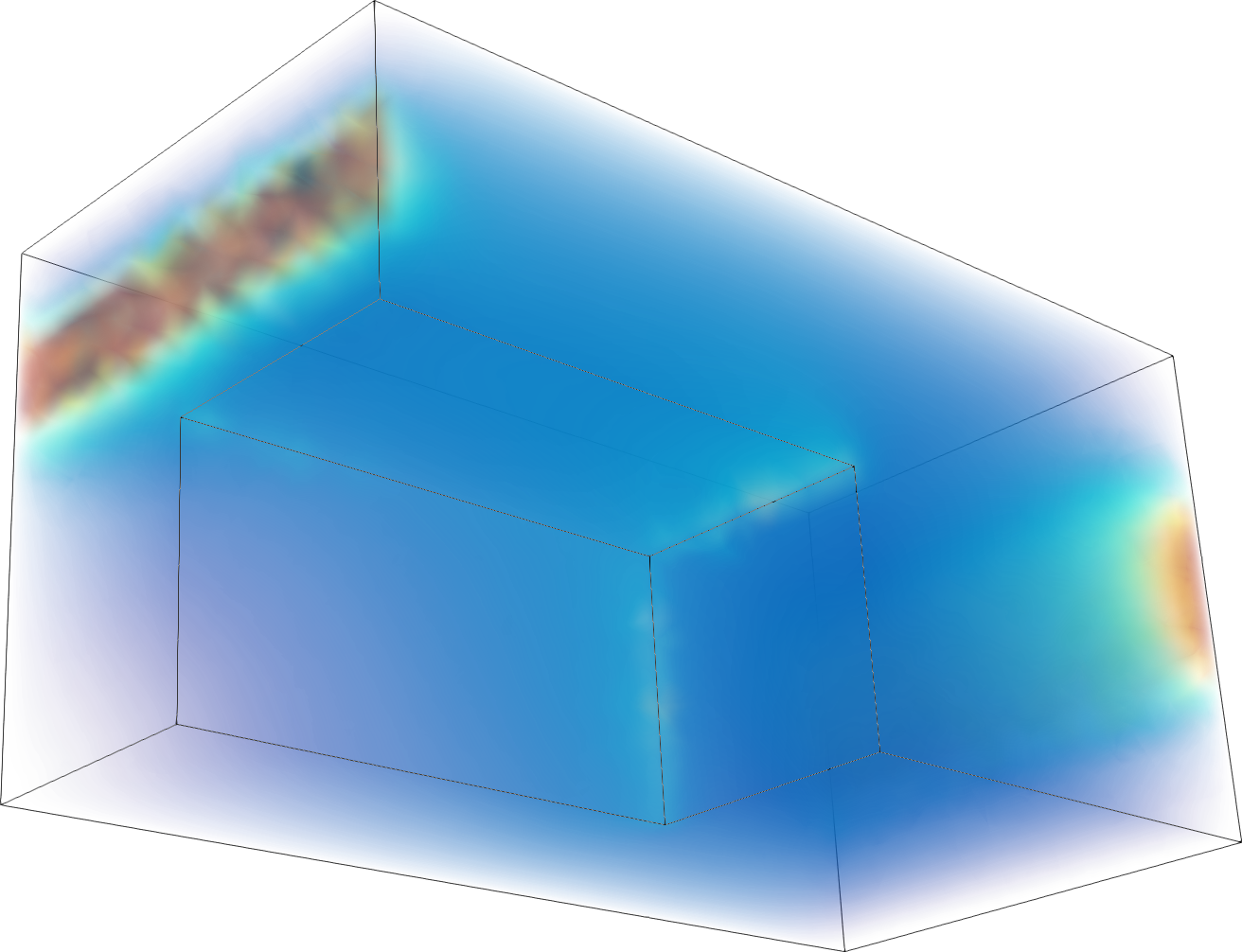}
		\label{fig:T3D_vol2}
	\end{subfigure}
	\begin{subfigure}[b]{.32\textwidth}
		\centering
		\includegraphics[width=0.89\textwidth]{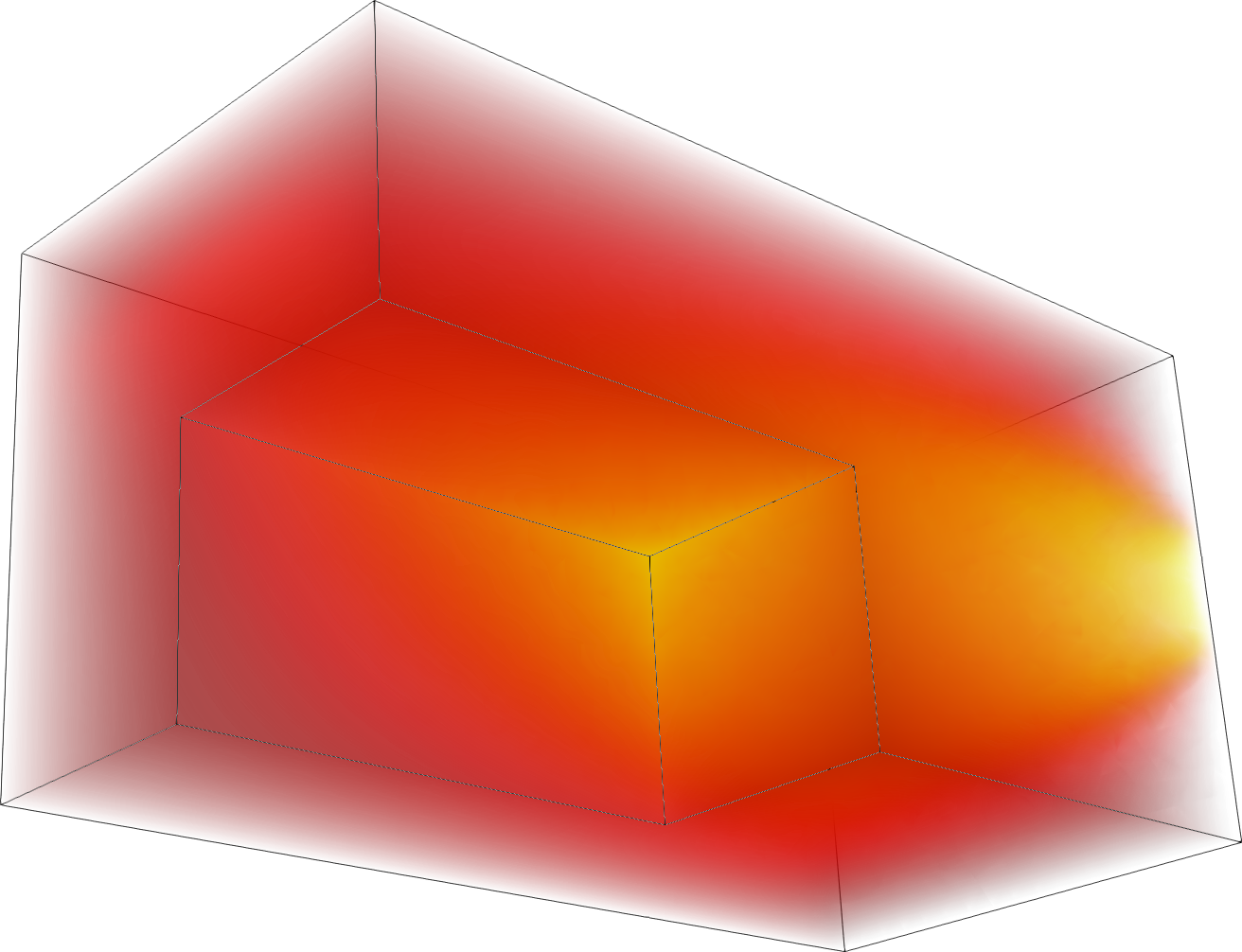}
		\label{fig:v3D_vol2}
	\end{subfigure}
	\begin{subfigure}[b]{.32\textwidth}
		\centering
		\includegraphics[width=0.89\textwidth]{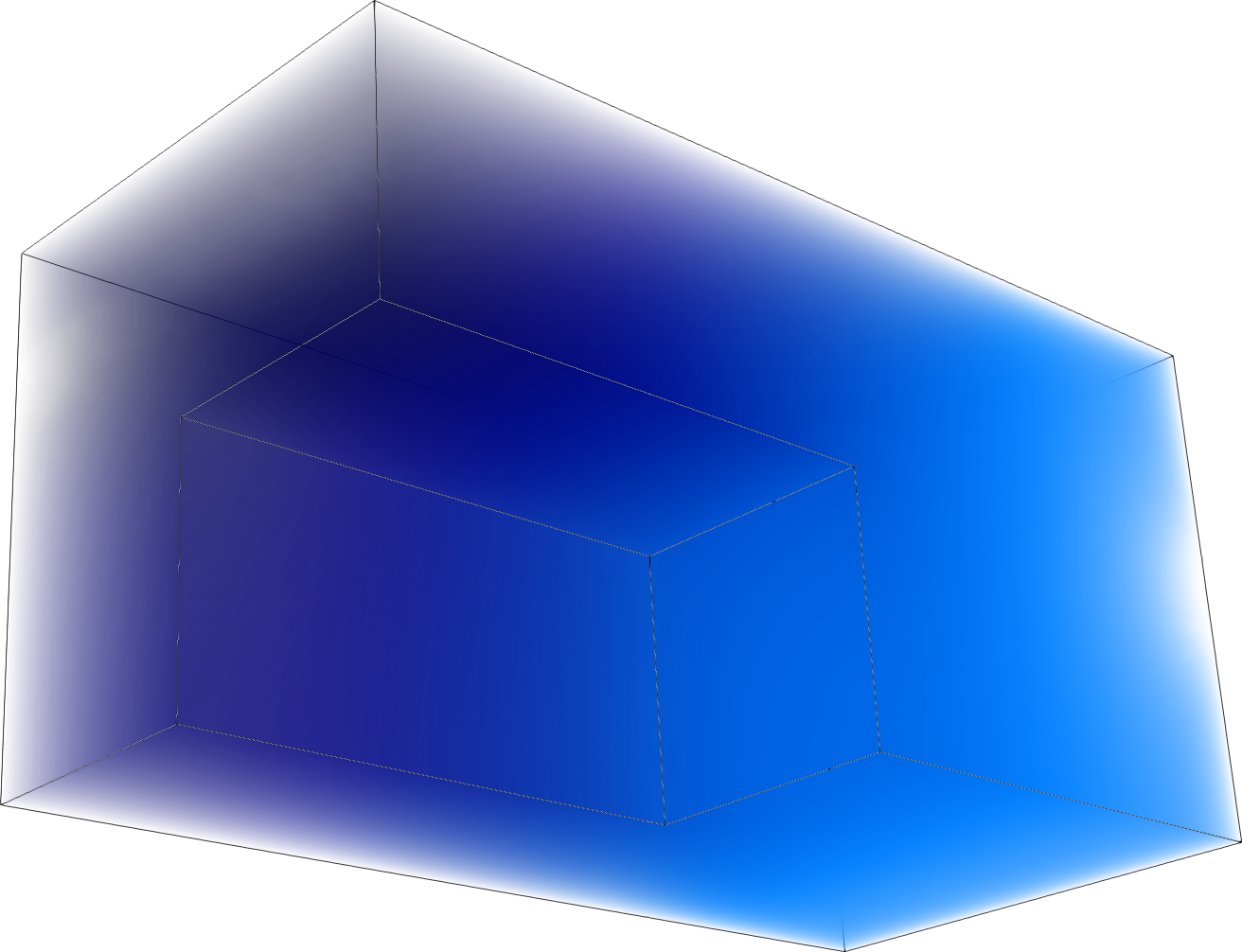}
		\label{fig:p3D_vol2}
	\end{subfigure}
	\begin{subfigure}[b]{.32\textwidth}
		\centering
		\includegraphics[width=0.89\textwidth]{IMG/AdvectionStep3D_v_colorbar.png}
		\label{fig:T3D_colorbar}
	\end{subfigure}
	\begin{subfigure}[b]{.32\textwidth}
		\centering
		\includegraphics[width=0.9\textwidth]{IMG/AdvectionStep3D_T_colorbar.png}
		\label{fig:v3D_colorbar}
	\end{subfigure}
	\begin{subfigure}[b]{.32\textwidth}
		\centering
		\includegraphics[width=0.9\textwidth]{IMG/AdvectionStep3D_p_colorbar.png}
		\label{fig:p3D_colorbar}
	\end{subfigure}
	\caption{Test case of Section~\ref{sec:advectionstep3D}: volume representation of the temperature field (left), velocity field (center), and pressure field (right).}
	\label{fig:advectionstep3D_volume}
\end{figure}

\section{Conclusions}
\label{sec:conclusion}
In this work, we presented and analyzed in a unified framework two schemes for the numerical discretization of a DF fluid flow model coupled with an advection–diffusion equation describing the temperature distribution inside a fluid. The first approach relied on discontinuous discrete spaces for velocity, pressure, and temperature fields. In the second approach, the velocity was discretized in the RT space, with pressure and temperature remaining discontinuous. A fixed-point linearization strategy—naturally leading to a splitting solution approach—was employed to address the nonlinearities. We carried out a unified stability analysis, investigated the well-posedness of the problem, and established the convergence of the fixed-point algorithm under mild assumptions on the problem data. Extensive two- and three-dimensional numerical experiments confirmed the theoretical results and demonstrated the efficiency and robustness of the proposed schemes in physically relevant test cases.
Future work could address the implementation of the dG-dG-dG scheme in the \texttt{lymph} software library \cite{Antonietti2025} to fully exploit the advantages of this formulation and the flexibility offered by polytopal elements. Moreover, for computational efficiency in three-dimensional simulations, proper preconditioning techniques for the two subproblems have to be developed. Finally, we can include the DF law for the flow field in more sophisticated models, e.g., in the thermo-poroelasticity theory.

\section*{Acknowledgments}
This work received funding from the European Union (ERC SyG, NEMESIS, project number 101115663). Views and opinions expressed are, however, those of the authors only and do not necessarily reflect those of the European Union or the European Research Council Executive Agency. Neither the European Union nor the granting authority can be held responsible for them. PFA, SB, and MB are members of INdAM-GNCS. SB and MB kindly acknowledge partial financial support by INdAM-GNCS project 2025 CUP E53C24001950001. The present research is part of the activities of “Dipartimento di Eccellenza 2023-2027”, funded by MUR, Italy.

\bibliography{bibliography.bib}

\end{document}